\RequirePackage{fix-cm}
\documentclass[smallextended]{svjour3}       
\smartqed  
\usepackage{graphicx}
\usepackage{algorithm}
\usepackage{algpseudocode}
\usepackage{caption}
\usepackage{xcolor}
\usepackage{lscape}   
\usepackage{pdflscape}
\usepackage{tabularx}    
\usepackage{booktabs}    
\usepackage{array}       
\usepackage{multirow}    
\usepackage{amsmath}
\usepackage{amsfonts}
\usepackage{rotating}
\usepackage{adjustbox}
\usepackage{caption}
\usepackage{mdframed} 

\usepackage[colorlinks,
            citecolor={blue!50!black}, 
            linkcolor={blue!50!black}]{hyperref}

\newcommand{\R}{\mathbb{R}}

\newcommand{\Q}{\mathbb{Q}}
\newcommand{\N}{\mathbb{N}}

\def\makeheadbox{\relax}
\begin{document}

\title{Accurate Linear Cutting-Plane Relaxations for ACOPF\thanks{This work was supported by an ARPA-E GO award.}}

\author{Daniel Bienstock         \and
        Mat\'ias Villagra
}


\institute{Daniel Bienstock \and Mat\'ias Villagra \at
              500 W 120th St \#315, New York, NY 10027 \\
              Tel.: +1 (212) 854-2942\\
              \email{dano@columbia.edu, mjv2153@columbia.edu}
}

\date{}
\maketitle

\begin{abstract}
    We present a pure linear cutting-plane relaxation approach for rapidly proving tight and accurate lower bounds for the Alternating Current Optimal Power Flow Problem (ACOPF) and its multi-period extension with ramping constraints. Our method leverages outer-envelope linear cuts for well-known second-order cone relaxations for ACOPF together with modern cut management techniques and reformulations to attain numerical stability. These techniques prove effective on a broad family of ACOPF instances, including the largest ones publicly available, quickly and robustly yielding tight bounds. Additionally, we consider the (frequent) case where an ACOPF instance is handled following a small or moderate change in problem data, e.g., load changes and generator or branch shut-offs. We provide significant computational evidence, on single and multi-period ACOPF instances, that the cuts computed on the prior instance provide very good lower bounds when warm-starting our algorithm on the perturbed instance.  
\end{abstract}

\section{Introduction}

The Alternating-Current Optimal Power Flow (ACOPF) problem is a well-known challenging computational task. It is nonlinear, non-convex and with feasible region that may be disconnected; see~\cite{hiskens+davy01},~\cite{bukhsh+etal13} and Proposition~\ref{proposition:nastyACOPF}. 
In~\cite{verma09,bienstock+verma19} it is shown that the feasibility version of the problem is strongly NP-hard, and  ~\cite{lehmann+etal16} proved that this decision problem is weakly NP-hard on star-networks. 

In the current state-of-the-art, some interior point methods are empirically successful at computing very good solutions but cannot provide any bounds on solution quality. Strong lower bounds are available through second-order cone (SOC) relaxations~\cite{jabr06,kocuk+etal16}; however all solvers do struggle when handling such relaxations for large or even medium cases (see~\cite{coffrin+etal16a}; we will provide additional evidence for this point). Other techniques, such as spatial-branch-and-bound methods applied to McCormick (linear) relaxations of quadratically-constrained formulations for ACOPF, tend to yield poor performance unless augmented by said SOC inequalities \textit{and} interior point methods.

In this paper we present a fast (linear) cutting-plane method used to obtain tight and accurate relaxations for even the largest ACOPF instances, and its multi-period extension with ramping constraints, by appropriately handling the SOC relaxations. Our research thrust is motivated by three key observations: 
\begin{itemize} 
\item[(1)] Linearly constrained convex quadratic optimization technology is, at this point, very mature -- many solvers are able to handle massively large instances quickly and robustly; these attributes extend to the case where formulations are dynamically constructed and updated, as would be the case with a cutting-plane algorithm. 
\item[(2)] Whereas the strength of the SOC relaxations has been known for some time, an adequate theoretical understanding for this behavior was not available. 
\item[(3)] In power engineering practice, it is often the case that a power flow problem (either in the AC or DC version) is solved on data that reflects a recent, and likely limited, update on a case that was previously handled. In power engineering language, a 'prior solution' was computed, and the problem is not solved 'from scratch.' In the context of our type of algorithm, this feature opens the door for the use of \textit{warm-started formulations}, i.e., the application of a cutting plane procedure that leverages previously computed cuts to obtain sharp bounds more rapidly than from scratch. 
\end{itemize}


As an additional attribute arising from our work, our formulations are, of course, linear. This feature paves the way for effective \textit{pricing} schemes, i.e., extensions of the locational marginal pricing setup currently used for real-time and day-ahead energy markets~\cite{oneill+etal05,gribik+etal07,bichler+etal23}.

\newpage

\paragraph{Our contributions}
\begin{itemize}
\item We describe very tight and numerically stable linearly-constrained relaxations for ACOPF. The relaxations can be constructed and solved robustly and quickly via a cutting-plane algorithm that relies on proper cut management. On medium to (very) large instances our algorithm is competitive with or better than what was previously possible using nonlinear relaxations, both in terms of bound quality and solution speed. The robustness of our method is especially prominent in the multi-period setting, where our method is able to provide tight lower bounds with high accuracy, while the nonlinear relaxations are simply out of reach for nonlinear solvers.
\item We provide a theoretical justification for the tightness of the SOC relaxation for ACOPF as well as for the use of our linear relaxations.  As we argue, the SOC constraints (or tight relaxations thereof, or equivalent formulations) are \textit{necessary} in order to achieve a tight relaxation.
More precisely, we point out that the (linear) active-power loss inequalities (introduced in \cite{bienstock+munoz14}) are outer-cone envelope approximations for the SOC constraints.  Moreover, the loss inequalities provide a fairly tight relaxation because they imply that every unit of demand and loss is matched by a corresponding unit of generation (a fact not otherwise guaranteed) via a flow-decomposition argument 
 \cite{bienstock15}.  
\item We demonstrate, through extensive numerical testing, that the warm-start feature for our cutting-plane algorithm is indeed achievable and effective. It is worth noting that this capability stands in contrast to what is possible using nonlinear (convex) solvers, in particular interior point methods. We also show numerical convergence of the dual variables of active-power balance constraints (which are used as input for computing locational marginal prices).
\item  The warm-start capability yields practicable and tight relaxations for \textit{multi-period} formulations of ACOPF. The tightness of our formulations is certified through a heuristic for such problems.
\item We bring forth to the ACOPF literature a discussion on approximate solutions to convex relaxations and their accuracy as lower bounds. We argue that linearly-constrained relaxations with convex quadratic objective possess robust theoretical bounding guarantees, in sharp contrast to what nonlinear relaxations such as SOCPs can offer.
\end{itemize}

The structure of this paper is as follows: in Section~\ref{section:theproblem} we formally define ACOPF and its multi-period extension; in Section~\ref{section:literature_review} we review the literature on convex relaxations for ACOPF with emphasis on linear relaxations; in Section~\ref{section:socps} we describe two second-order cone relaxations, provide theoretical justification for the tightness of these relaxations, and introduce a new numerically better-behaved conic relaxation; in Section~\ref{section:cutplane} we describe our cutting-plane framework, as well as the cut separation and cut management heuristics deployed in our algorithm; in Section~\ref{section:accuracy} we provide a discussion on the accuracy of lower bounds of convex relaxations for ACOPF; in Section~\ref{section:experiments} we present significant computational experiments on the performance of SOCPs versus our algorithm and warm-started formulations for the single and multi-period settings; and finally in Section~\ref{section:futurework} we conclude and outline future research directions. Our proofs can be found in the Appendix~\ref{appendix}.

This work significantly extends a much shorter conference version of this paper where the single-period setting is addressed~\cite{bienstock+villagra24}. The relaxation in Section~\ref{subsection:newrelaxation} is new, as well as the multi-period computational results and the sections on accuracy and behavior of the dual variables of active-power balance constraints.

\vspace{1em}

\begin{mdframed}
\captionsetup{type=figure} 
\caption{ACOPF Formulation using polar coordinates}
\begin{subequations}\label{AC:firstformulation}
\small
\begin{align} 
[\text{ACOPF}]: &\hspace{7em}\min \hspace{2em} \sum_{k \in \mathcal{G}} F_{k}(P_{k}^{g}) \label{AC:theobjective} \\
    \text{subject to:}& \nonumber\\
    &\hspace{-4em}\forall \text{  bus $k \in \mathcal{B}$:} \nonumber\\
    &\hspace{6em} \sum_{\{k,m\} \in \delta(k)} P_{km} = \sum_{\ell \in \mathcal{G}_{k}} P_{\ell}^{g} - P_{k}^{d} \label{AC:activepowerbal}\\
    &\hspace{6em} \sum_{\{k,m\} \in \delta(k)} Q_{km} = \sum_{\ell \in \mathcal{G}_{k}} Q_{\ell}^{g} - Q_{k}^{d} \label{AC:reactivepowerbal}\\
    &\hspace{-4em}\forall \text{  branch $\{k,m\} \in \mathcal{E}$:} \nonumber\\
    P_{km} &= G_{kk}|V_{k}|^{2} + |V_{k}| |V_{m}| ( G_{km}  \cos(\theta_{k} - \theta_{m}) +  B_{km} \sin(\theta_{k} - \theta_{m})) \label{AC:def_from_activepower} \\
    P_{mk} &= G_{mm}|V_{m}|^{2} + |V_{k}| |V_{m}|(G_{mk}\cos(\theta_{k} - \theta_{m}) -  B_{mk} \sin(\theta_{k} - \theta_{m})) \label{AC:def_to_activepower} \\
    Q_{km} &= - B_{kk}|V_{k}|^{2} + |V_{k}| |V_{m}|(B_{km}\cos(\theta_{k} - \theta_{m}) -  G_{km} \sin(\theta_{k} - \theta_{m})) \label{AC:def_from_reactivepower}\\
    Q_{mk} &= - B_{mm}|V_{m}|^{2} + |V_{k}| |V_{m}|(B_{mk}\cos(\theta_{k} - \theta_{m}) +  G_{mk} \sin(\theta_{k} - \theta_{m}) ) \label{AC:def_to_reactivepower} \\
    &\hspace{-4em}\forall \text{  generator $k \in \mathcal{G}$:} \nonumber \\ 
    &\hspace{9em} P^{\min}_{k} \leq P_{k}^g \leq P_{k}^{\max}  \label{AC:genactivepowerlimit}\\
    &\hspace{9em} Q^{\min}_{k} \leq Q_{k}^g \leq Q_{k}^{\max}  \label{AC:genreactivepowerlimit}\\
    &\hspace{-4em}\forall \text{  bus $k \in \mathcal{B}$:} \nonumber\\
    &\hspace{9em} (V^{\min}_{k})^{2} \leq |V_{k}|^{2} \leq (V^{\max}_{k})^{2} \label{AC:voltlimit}\\
    &\hspace{-4em}\forall \text{  branch $\{k,m\} \in \mathcal{E}$:} \nonumber\\
    &\hspace{9em} \max \left\{P_{km}^{2} + Q_{km}^{2},P_{mk}^{2} + Q_{mk}^{2} \right\}  \leq U_{km} \label{AC:capacity} \\
    &\hspace{9em} | \theta_{k} - \theta_{m} | \leq \bar{\Delta}_{km} \label{AC:maxanglediff}
\end{align}
\end{subequations}
\end{mdframed}

\section{Formal Problem Statement}\label{section:theproblem}

In standard (single-period) ACOPF we are given, as input, a \textit{power system} consisting of an undirected network endowed with numerical parameters describing physical attributes; a set of \textit{generators}, and a set of complex-valued \textit{loads} (or demands).  The objective of the problem is to set complex voltages at the buses, and generator outputs, so as to satisfy loads at minimum cost -- cost is incurred by the generators -- while transmitting power following laws of physics and equipment constraints. A mathematical description is given above in~\eqref{AC:firstformulation} using the so-called \textit{polar representation}.

We denote by $\mathcal{N} := (\mathcal{B},\mathcal{E})$ the network, where $\mathcal{B}$ denotes the set of nodes, which (following power engineering conventions) we will refer to as \textit{buses}, and $\mathcal{E}$ denotes the set of edges, which we will refer to as transmission \textit{lines} or \textit{transformers}, or generically, \textit{branches}. We denote by $\mathcal{G}$ the set of generators of the grid, each of which is located at some bus;  for each bus $k \in \mathcal{B}$, we denote by $\mathcal{G}_{k} \subseteq \mathcal{G}$ the generators at bus $k$. 
 
Each bus $k$ has a fixed load $P_{k}^d + j Q_{k}^{d}$, where $P_{k}^{d} \geq 0$ is termed active power load, and  $- \infty \leq Q_{k}^{d} \leq + \infty$ is reactive power load; lower $V_{k}^{\min} \geq 0$ and upper $V_{k}^{\max} \geq 0$ voltage limits.  For each branch $\{k,m\}$ we are given thermal limit $0 \leq U_{km} \leq +\infty$, and maximum angle-difference $|\Delta_{km}| \leq \pi$. Thus, the goal is to find voltage magnitude $|V_k|$ and phase angle $\theta_k$ at each bus $k$, active $P^{g}$ and reactive $Q^g$ power generation for every generator $g$, so that power is transmitted by the network so as to satisfy active $P^{d}$ and reactive $Q^{d}$ power demands at minimum cost. 

In the above formulation~\eqref{AC:firstformulation}, the physical parameters of each line $\{k,m\} \in \mathcal{E}$ are described by
\begin{equation*}
    Y_{km} := \begin{pmatrix}
    G_{kk} + j B_{kk} & G_{km} + j B_{km} \\
    G_{mk} + j B_{mk} & G_{mm} + j B_{mk},
    \end{pmatrix}
\end{equation*}
which is the (complex) admittance matrix of the transmission line $\{k,m\}$ (see~\cite{bergen+vittal99} for background on transmission line modeling). These parameters model, in~\eqref{AC:def_from_activepower}-\eqref{AC:def_to_reactivepower},  active and reactive power flows. Moreover, inequalities~\eqref{AC:capacity}-\eqref{AC:maxanglediff} correspond to flow capacity constraints, and inequalities~\eqref{AC:genactivepowerlimit}-\eqref{AC:voltlimit} impose operational limits on power generation and voltages. Constraints~\eqref{AC:activepowerbal}-\eqref{AC:reactivepowerbal} impose active and reactive power balance; the left-hand side represents power injection at bus $k \in \mathcal{B}$, while the right-hand side represents net power generation (generation minus demand) at bus $k$. Finally, for each generator $k \in \mathcal{G}$, it is customary to assume the functions $F_{k} : \R \to \R$ in the objective~\eqref{AC:theobjective} are convex piecewise-linear or convex quadratic. 

We remark that, often, constraint \eqref{AC:maxanglediff} is not present, and, when explicitly given, concerns angle limits $\bar{\Delta}_{km}$ that are \textit{small} (smaller than $\pi/2$).  Under such assumptions there are equivalent ways to restate \eqref{AC:maxanglediff} involving the arctangent function and other variables present in the formulation (the same applies to convex relaxations).

Please refer to the surveys~\cite{molzahn+hiskens19},~\cite{bienstock+etal20} for alternative, but equivalent, ACOPF formulations.

\subsection{Multi-Period Formulation}\label{multiperiod:formulation}

Our techniques apply to a multi-period setting for ACOPF; this extension is important because it corresponds to the way that energy markets are operated.  We will rely on a simple multi-period setting for ACOPF, where in each period $t$ there is a standard ACOPF formulation. The variables for this period are linked to those other periods $t' \neq t$\footnote{As a simple case, $t' = t-1$.} via generator ramping constraints (defined below). Other variants of this problem include ``uptime/downtime" costs for generators, energy storage reserve requirements, and network topology control, among others (see e.g.~\cite{castillo+etal16,hedman+etal11,go3}).

The multi-period problem can be described as follows: given a time-horizon of $T \in \N$ periods, a time-invariant undirected network endowed with numerical parameters describing physical attributes, and complex network loads $P_{i,t}^{d} + j Q_{i,t}^{d}$ for $i \in \mathcal{B}$ and $t \in \{0,\dots,T-1\}$, the goal is to set complex voltages $V_{i,t}$ at the buses, and generator outputs $P^{g}_{k,t} + j Q^{g}_{k,t}$ for $k \in \mathcal{G}$ and $t \in \{0,\dots,T-1\}$, so as to satisfy loads at minimum cost while transmitting power following laws of physics and equipment costs, and satisfying generator ramping constraints. For every generator $k \in \mathcal{G}$, and any time-period $t \in \{0,\dots,T-1\}$, the associated ramping up and down constraints are given by\footnote{Since $P_{k}^{g}$ can potentially be negative for some generator $g \in \mathcal{G}$, then the latter constraints need to be formulated in terms of the absolute value of $P^{g}_{k}$.}
\begin{subequations}\label{rampingconstraints}
  \begin{align}
    (1 + r^{u}_{k,t}) P^{g}_{k,t} - P^{g}_{k+1,t} &\geq 0 \\
    (1 - r^{d}_{k,t}) P^{g}_{k,t} - P^{g}_{k+1,t} &\leq 0.
\end{align}  
\end{subequations}
where $r_{k,t}^{u},r_{k,t}^{d} \in (0,1)$ are the ramping up and down rates, respectively. In summary, multi-period ACOPF for a time-invariant network $(\mathcal{B},\mathcal{E})$ with rational data $(P^{d}_{i,t},Q^{d}_{i,t})_{i \in \mathcal{B},t \in \{0,\dots,T-1\}}$ and $(Y_{km},U_{km})_{\{k,m\} \in \mathcal{E}}$ can be defined by \eqref{AC:firstformulation} and \eqref{rampingconstraints} on variables $$(V_{i,t},\theta_{i,t})_{i \in \mathcal{B}}, \, (P_{km,t},P_{mk,t},Q_{km,t},Q_{mk,t})_{\{k,m\} \in \mathcal{E}}, \, (P^{g}_{k,t},Q^{d}_{k,t})_{k \in \mathcal{G}}$$ for each $t \in \{0,\dots,T-1\}$.

\section{Literature Review}\label{section:literature_review}

First, we will briefly review the extensive literature on convex relaxations for ACOPF. An extensive review is provided in \cite{molzahn+hiskens19}. 

The simplest relaxations use, a starting point, a rectangular formulation of the ACOPF problem (rather than the polar setup described above) yielding a QCQP (quadratically constrained quadratic program) and rely on the well-known McCormick \cite{mccormick76} reformulation to linearize bilinear expressions. This straightforward relaxation has long been known to provide very weak bounds. 

The SOC relaxation introduced in~\cite{jabr06}, which is widely known as the Jabr relaxation (see Section \ref{subsection:jabrsocp}), has had significant impact due to its effectiveness as a lower bounding technique. Briefly, it linearizes the power flow definitions \eqref{AC:def_from_activepower}-\eqref{AC:def_to_reactivepower} using $|\mathcal{B}| + 2 |\mathcal{E}|$ additional variables and adds $|\mathcal{E}|$ valid SOC inequalities (c.f. \ref{subsection:jabrsocp} for a complete derivation). Moreover, in~\cite{jabr06} it is shown that for radial transmission networks, i.e., networks that have a tree structure, the load flow problem can be solved via an SOCP whose conic constraints correspond to the latter SOC inequalities. While on the one hand the SOC relaxation is strong, it also yields formulations that, in the case of large ACOPF instances, are very challenging even for the best solvers.

A wide variety of techniques have been proposed to strengthen the Jabr relaxation. In~\cite{kocuk+etal16} \emph{arctangent constraints} are associated with cycles, with the goal of capturing the relationship between the additional variables in the Jabr relaxation and phase angles -- these are formulated as bilinear constraints, and then linearized via McCormick inequalities. The two other strengthened SOC formulations proposed in~\cite{kocuk+etal16} add polyhedral envelopes for \emph{arctangent functions}, and dynamically generate semi-definite cuts for cycles in the network. \cite{kocuk+etal18} proposes a \emph{minor-based} formulation for ACOPF (which is a reformulation of the rank-one constraints in the QCQP formulation for ACOPF~\cite{molzahn+hiskens19}), which is relaxed to generate cutting-planes improving on the tightness of the Jabr relaxation. \cite{coffrin+etal16a} introduced the Quadratic Convex (QC) relaxation   (an SOC relaxation) which is related to the Jabr relaxation (c.f. \ref{subsection:i2socp}), strengthened with polyhedral envelopes for sine, cosine and bilinear terms appearing in the power flow definitions \eqref{AC:def_from_activepower}-\eqref{AC:def_to_reactivepower}. A novel machine learning-based approach is followed in~\cite{qiu+etal24} to obtain fast dual bounds for the Jabr relaxation for cases with up to 2869 buses.

A semidefinite programming relaxation based on the \emph{Shor relaxation}~\cite{shor87} is presented in  in~\cite{bai+etal08} followed by~\cite{lavaei+low12}.  This formulation is at least as tight as the Jabr relaxation at the expense of higher computational cost~\cite{kocuk+etal18}. In~\cite{lavaei+low12} it is proven that the SDP relaxation achieves zero duality gap under appropriate network assumptions. Experiments with SDP relaxations have been constrained by current SDP technology capabilities.



Next we review linear relaxations for ACOPF. \cite{bienstock+munoz14,bienstock+munoz15} (also see \cite{molzahn+hiskens19}) introduce the so-called active-power loss linear inequalities which impose the fact that on any branch the active power loss is nonnegative. The resulting relaxation, which we term the linear-loss-relaxation, is shown to yield good lower bounds. 
In a similar same vein, \cite{coffrin+etal16b} propose the network flow and the copper plate relaxations. The network flow relaxation amounts to the linear loss-relaxation with additional sparse linear inequalities that lower bound net reactive power losses in appropriate cases\footnote{Note that it is physically possible to have reactive power gains, i.e., negative reactive losses, see~\cite{bergen+vittal99}).}. Moreover, the ``copper plate" relaxation is obtained from the network flow relaxation by neglecting the power flow equations entirely via aggregation of all active and reactive power injections in the network. Along these lines \cite{taylor+hover11} provides a relaxation which enforces a (valid) linear relationship between active and reactive power losses by relaxing linear combinations of \eqref{AC:def_from_activepower}-\eqref{AC:def_to_reactivepower}. 

The technique in ~\cite{bienstock+munoz14}, inspired by Glover~\cite{glover75}, $\epsilon$-approximates the products of continuous variables (arising from the rectangular formulation of ACOPF \cite{molzahn+hiskens19}), to arbitrary precision, using binary expansions and McCormick inequalities. This process yields a mixed integer linear $\epsilon$-approximation for ACOPF. Another linear $\epsilon$-approximation, which is based on the Jabr relaxation, is used in~\cite{mhanna+chapman16}. Their main contribution is using the SOC linear approximation developed in~\cite{bental+nemirovski01} which requires $O(k_{\ell} \log(1/\epsilon) )$ linear constraints and additional variables to $\epsilon$-approximate a conic constraint of row size $k_\ell \in \N$. One of the algorithms in~\cite{mhanna+mancarella22} is an SLP method applied to the Jabr relaxation, and thus yielding a linear relaxation for ACOPF.


The well-known, and widely used in practice, Direct Current Optimal Power Flow (DCOPF) model is a linear approximation for ACOPF~\cite{stott+etal09,cain+etal12}. It is based on a number of simplifications that are approximately valid under normal system operations; under this formulation active power losses are zero and reactive power is completely ignored~\cite{bergen+vittal99}. A main drawback of this formulation is that AC feasible solutions might not even be feasible for DCOPF~\cite{baker20} -- losses are the main cause of infeasibility -- in contrast to the previously mentioned linear $\epsilon$-approximations. In any case, this linear model is the canonical approximation of ACOPF for many extensions and applications such as welfare maximization and pricing in energy markets, day-ahead security-constrained unit commitment (SCUC), real-time security-constrained
dispatch (SCD) and transmission switching (TS) among many others.

We refer the reader to the surveys~\cite{molzahn+hiskens19,bienstock+etal20,zohrizadeh+etal20} for additional material on convex relaxations for single-period ACOPF.\\

Single-period convex relaxations can naturally be extended to multi-period ACOPF. To the best of our knowledge, all of the computational experiments address small to medium-sized cases -- which is not surprising given that the single-period (nonlinear) convex relaxations are already challenging enough (see e.g. \ref{subsubsection:performance_socps}). It is worth mentioning that most of the research in the multi-period setting has focused on obtaining AC primal bounds for numerous multi-period ACOPF features. Multi-period Jabr SOCPs are solved in:~\cite{lorca+sun17} to find approximate solutions to a robust multi-period ACOPF;~\cite{constante+etal22} and~\cite{bai+etal15} as subproblems of Benders decomposition algorithms for AC-constrained UC;~\cite{quan+etal14} for SOC-constrained UC; and a multi-period mixed integer Jabr SOCP is solved~\cite{liu+etal18} for AC-constrained UC.
Moreover, \cite{zohrizadeh+etal19} solves a sequence of Jabr SOCP relaxations using penalty terms to enforce AC feasibility. 
On the other hand, notable linear approximations for multi-period ACOPF are the copper-plate linear, DCOPF, and strengthened versions of DCOPF~\cite{yang+etal2018} which account for reactive power and voltage magnitudes. These approximations are key ingredients for solving scalable UC or Security Constrained ACOPF (SC-ACOPF). In~\cite{alizadeh+capitanescu23} a linearized multi-period Security Constrained Stochastic ACOPF based on the strengthened DC approximation~\cite{yang+etal2018} is presented. \\




Next, we succinctly review literature on heuristics -- mostly based on a convex relaxations or linear appromixations -- used to find AC feasible solutions. In the single-period setting, the linear programming Hot-Start and Warm-Start approximation models are proposed in~\cite{coffrin+vanhentenryck14} for finding locally optimal single-period AC solutions. The Hot-Start model assumes there is a solved AC base-point solution available, and it uses to fix voltage magnitudes in the power AC power flow definitions. Morevoer, they approximate $\sin(\theta)$ by $\theta$ and use a linear convex approximation of $\cos(\theta)$ (under the assumption $\theta \in (-\pi/2, \pi/2)$). Since fixing the voltages might be too restrictive in some cases, their Warm-Start model chooses takes a subset of the buses as target buses, and linearizes deviations from the targeted voltage magnitudes. In~\cite{castillo+etal16} and~\cite{mhanna+mancarella22} successive linear programming (SLP) algorithms are proposed for finding locally optimal single-period AC solutions. 

With respect to the multi-period setting,~\cite{castillo+etal16b} presents an outer-approximation iterative algorithm for the UC problem with AC network constraints (i.e, an multi-period ACOPF is being solved) -- this method solves a finite sequence of MILP master problems and nonlinear subproblems (ACOPF with fixed binary variables). The nonlinear subproblems are tackled using the SLP (sequential linear programming algorithm) presented in~\cite{castillo+etal16a}. Moreover, in~\cite{constante+etal22} an SLP is used as a penalty method to recover a UC problem with AC-constraints feasible solution from the multi-period Jabr relaxation. This is different from the technique used in~\cite{zohrizadeh+etal19} where a multiobjective SOCP relaxation is used to obtain an AC feasible solution. In~\cite{alguacil+conejo00} a Benders decomposition algorithm is used for AC-constrained UC -- the subproblem is a multi-period ACOPF on a small network. Note that all of the previously mentioned references on convex relaxations of some multi-period ACOPF variant are used for AC feasibility. 

Two general purpose nonlinear solvers which can be used to find solutions for single and multi-period ACOPF are Knitro~\cite{knitro} and IPOPT~\cite{ipopt}.



\section{Convex Relaxations and Power Losses}\label{section:socps}

\subsection{Jabr relaxation}\label{subsection:jabrsocp}

A well-known convex relaxation of ACOPF is the \emph{Jabr relaxation}~\cite{jabr06}. It linearizes the power flow definitions \eqref{AC:def_from_activepower}-\eqref{AC:def_to_reactivepower} using $|\mathcal{B}| + 2 |\mathcal{E}|$ additional variables and adds $|\mathcal{E}|$ rotated-cone inequalities. A simple derivation is as follows: For any line $\{k,m\} \in \mathcal{E}$, we define 
\begin{equation}\label{jabrimplicit}
v_{k}^{(2)} := |V_{k}|^2, \, c_{km} := |V_{k}| |V_{m}| \cos(\theta_{k} - \theta_{m}), \, s_{km} := |V_{k}| |V_{m}| \sin(\theta_{k} - \theta_{m}).
\end{equation}
Clearly we have the following valid non-convex quadratic relation
\begin{equation}
    c_{km}^{2} + s_{km}^{2} = v_{k}^{(2)} v_{m}^{(2)},
\end{equation}
which in Jabr \cite{jabr06} is relaxed into the (convex) inequality
\begin{equation}
    c_{km}^{2} + s_{km}^{2} \leq v_{k}^{(2)} v_{m}^{(2)}. \label{SOCP:jabr}
\end{equation}
This is a rotated-cone inequality hence it can be represented as a second-order cone constraint. Note that~\eqref{jabrimplicit} can be used to represent the power flow equations in \eqref{AC:def_from_activepower}-\eqref{AC:def_to_reactivepower} as, $\forall \{k,m\} \in \mathcal{E}$,
\begin{subequations}\label{JABR:flows}
\begin{align}
    P_{km} &= G_{kk}v_{k}^{(2)} + G_{km}c_{km} +  B_{km} s_{km} \label{SOCP:def_to_activepower}\\
    P_{mk} &= G_{mm}v_{m}^{(2)} + G_{mk}c_{km} -  B_{mk} s_{km} \label{SOCP:def_from_activepower} \\
    Q_{km} &= - B_{kk}v_{k}^{(2)} + B_{km}c_{km} -  G_{km} s_{km} \label{SOCP:def_to_reactivepower}\\
    Q_{mk} &= - B_{mm}v_{m}^{(2)} + B_{mk} c_{km} +  G_{mk} s_{km}. \label{SOCP:def_from_reactivepower}
\end{align}
\end{subequations} 

\noindent In summary, the Jabr relaxation can be obtained from the formulation \eqref{AC:firstformulation} by: (i) adding the $c_{km}, s_{km}, v^{(2)}_k$ variables;
(ii) replacing \eqref{AC:def_from_activepower}-\eqref{AC:def_to_activepower} with \eqref{JABR:flows}; and (iii) adding the constraints \eqref{SOCP:jabr}.\footnote{We stress that \eqref{jabrimplicit} is \textit{not} added.}


\subsection{An alternative SOCP}\label{subsection:i2socp}

Recall that complex power injected into branch $\{k,m\} \in \mathcal{E}$ at bus $k \in \mathcal{B}$ is defined by ${S_{km} := V_{k} I^{*}_{km}}$, hence
\begin{equation}\label{eq:power2def}
    |S_{km}|^{2} = |V_{k}|^{2} | I_{km} |^{2}
\end{equation}
 holds. Moreover, since complex power can be decomposed into active and reactive power as $S_{km} = P_{km} + j Q_{km}$, and recalling that $v_{k}^{(2)} := |V_{k}|^{2}$ while denoting $i_{km}^{(2)} := |I_{km}|^{2}$, we have
\begin{equation}\label{fundamentalpowereq}
    P_{km}^{2} + Q_{km}^{2} = v_{k}^{(2)} i_{km}^{(2)}.
\end{equation}
By relaxing the equality \eqref{fundamentalpowereq} we obtain the \emph{i2 rotated-cone inequality}~\cite{coffrin+etal16a,farivar+etal11}
\begin{equation}\label{i2ineq}
    P_{km}^{2} + Q_{km}^{2} \leq v_{k}^{(2)} i^{(2)}_{km}.
\end{equation}
Since the variable $i_{km}^{(2)}$ can be defined by
\begin{equation}\label{quicki2def}
     i_{km}^{(2)} = \alpha_{km} v_{k}^{(2)} + \beta_{km} v_{m}^{(2)} + \gamma_{km} c_{km} + \zeta_{km} s_{km}
\end{equation}
for appropriate constants  $\alpha_{km}, \beta_{km}, \gamma_{km}$ and $\zeta_{km}$ (c.f.~\ref{appendix:i2def}), we obtain an alternative SOC relaxation. This formulation, which we call the \emph{i2 relaxation}, is comprised by \eqref{AC:theobjective}-\eqref{AC:reactivepowerbal},~\eqref{AC:genactivepowerlimit}-\eqref{AC:capacity}, the definition of $i_{km}^{(2)}$~\eqref{quicki2def}, and the rotated-cone inequalities~\eqref{i2ineq}. 

It is known~\cite{subhonmesh+etal12,coffrin+etal16a,coffrin+etal15a} that for each branch $\{k,m\}$ the system defined by the linearized power flows \eqref{JABR:flows} plus the  Jabr inequality~\eqref{SOCP:jabr}, and the system defined by the linearized power flows \eqref{JABR:flows}, the rotated-cone inequality~\eqref{i2ineq} and the linear definition of $i^{(2)}_{km}$~\eqref{appendix:i2def_eq}, are equivalent. It is to be noted that by this we mean that for each feasible solution to one system there is a feasible solution to the other one. However,  equivalence may fail to hold if $i^{(2)}$ is upper bounded.

\begin{proposition}\label{proposition:i2_stronger_than_jabr}
    The Jabr and the $i2$ relaxations are equivalent if the $i^{(2)}$ variables are not upper bounded, and otherwise the $i2$ relaxation can be strictly stronger.
\end{proposition}
\begin{proof}
    Sufficiency was proven in~\cite{coffrin+etal16a}. For an example where the i2 relaxation is strictly stronger than the Jabr relaxation see~\ref{appendix:i2strongerjabr}.
\end{proof}

Our computational experiments corroborate this fact; we have found that linear outer-approximation cuts for the rotated-cone inequalities \eqref{SOCP:jabr} and \eqref{i2ineq} have significantly different impact in lower bounding ACOPF (c.f. Sect.~\ref{section:experiments}).

\subsection{Losses}\label{subsection:losses}

Transmission line losses are a structural feature of power systems~\cite{bergen+vittal99}. It is a physical phenomenon in which the (complex) power sent from bus $k$ to bus $m$ along branch $km$ will not necessarily be equal to the power received at bus $m$. This feature of power grids is captured (in steady state) by the power flow equations~\eqref{AC:def_from_activepower}-\eqref{AC:def_to_reactivepower}. Indeed, active-power loss on line $\{k,m\}$ can be defined as the quantity 
\begin{equation}\label{loss:definition}
    \ell_{km} := P_{km} + P_{mk}.
\end{equation}

First, we will focus on active-power losses. It is well-known (see, e.g., ~\cite{farivar+etal11,bienstock+munoz14}) that for a simple transmission line, i.e., a non-transformer branch without shunts, active-power loss equals
\begin{equation}\label{loss:2}
    \ell_{km} = g_{km} | V_{k} - V_{m} |^{2}
\end{equation}
where $g_{km}$ denotes line conductance (for simple lines we have $G_{kk} = G_{mm} =: g_{km})$. Thus, if $g_{km} \geq 0$ (i.e., the standard case ~\cite{bergen+vittal99,zimmerman+etal11}), the linear inequality $P_{km} + P_{mk} \geq 0$ is valid for ACOPF. This inequality still holds under (traditional) branch shunts and transformers provided that $g_{km} \geq 0$. In an arbitrary relaxation to ACOPF we might have $\ell_{km} < 0$, and, as we demonstrate below; this feature will make the relaxation weak. Moreover, critically, Proposition \ref{prop:jabrouter} below shows that the Jabr inequalities imply nonnegative active power losses as outer-envelope inequalities.\\

Additionally, the following theorem in~\cite{bienstock15} sheds light on the relationship between active-power losses and (global) active power balance.  Consider an ACOPF instance on an undirected network $\mathcal{N} = (\mathcal{B},\mathcal{E})$. Let us subdivide each branch $\{k,m\}$ by introducing a new \textit{node}, denoted $n_{km}$.  To properly place this result, we consider a solution to any relaxation for ACOPF. \textit{Any} relaxation is of interest (even an exact relaxation) so long as variables $P_{km}$ and $P_{mk}$ are present, and active power flow balance constraints~\eqref{AC:activepowerbal} are given (or implied). Now, given a feasible solution to the relaxation we assign real flow values to the edges of the subdivided network, and orient those edges, as follows. Consider an arbitrary branch $\{k,m\}$. 
\begin{itemize}
    \item [(1)] If $P_{km} \ge 0$ the edge between $k$ and $n_{km}$ is oriented from $k$ to $n_{km}$ and assigned flow $P_{km}$. If $P_{km} < 0$ the edge between $k$ and $n_{km}$ is oriented from $n_{km}$ to $k$ and assigned flow $-P_{km}$.
    \item [(2)] Similarly, if $P_{mk} \ge 0$  the edge between $m$ and $n_{km}$ is oriented from $m$ to $n_{km}$ and assigned flow $P_{mk}$. If $P_{mk} < 0$ the edge between $m$ and $n_{km}$ is oriented from $n_{km}$ to $m$ and assigned flow $-P_{mk}$.
\end{itemize}
Thus, each edge in the subdivided network is oriented in the direction of the flow value that was assigned.  Using the well-known network flow-decomposition theorem~\cite{ahuja+etal93} yields that the flows in the subdivided network can be decomposed into a sum of \textit{flow-carrying paths}, i.e., directed paths where
\begin{itemize}
    \item Each path starts from a \textit{source} (a bus with positive net generation or a node on a branch with negative loss) and ends at a \textit{sink} (a bus with negative net generation or a node on a branch with positive loss),
    \item Each path carries a fixed positive flow amount.
\end{itemize} 
In summary, 
 
\begin{theorem}[Theorem 1.2.11~\cite{bienstock15}]
    Consider any relaxation to an ACOPF instance that includes, or implies, active power flow balance constraints \eqref{AC:activepowerbal}. Then
\begin{itemize}
    \item[(a)]The sum of active-power generation minus negative active power losses amounts equals the sum of active-power loads plus positive active-power losses. Furthermore, there exists a family of flow-carrying paths that accounts for all the flows in the subdivided network.   
    \item[(b)] There is a similar statement regarding reactive power flows (paraphrased here for brevity). The sum of reactive-power generation amounts plus the sum of reactive-power generated by line shunt elements, equals the sum of reactive-power loads and losses. Furthermore, there exist a family of paths that account for all reactive-power generation, loads, and losses. Each path has as origin a generator (with positive reactive-power generation) or a shunt element, and as destination a bus with a positive reactive-power load or a line loss. If all shunt susceptances are zero, then each path has a bus with positive reactive-power generation as origin. 
\end{itemize}
\end{theorem}
The relevance of (a) is that when negative losses are present, then total generation may be smaller than total loads -- effectively, the negative losses amount to a source of free generation and directly contribute to a lower objective value for ACOPF than AC feasible. While this point is also made in~\cite{coffrin+etal16a} regarding a particular relaxation, the above theorem provides a clear explanation for this fact that applies to \textit{any} relaxation. 

As an experiment, suppose we remove, from the Jabr relaxation, just one of the SOC constraints, say for branch $\{\hat k, \hat m\}$, thus obtaining a weaker relaxation; and suppose that, furthermore, the branch limits $U_{km}$ are large (the usual case).  Then, as we substantiate next, it is quite likely that the loss on $\{\hat k, \hat m\}$ will be negative in the weakened relaxation and the value of the relaxation will be (much) weaker -- precisely because that phenomenon would allow a solution with less generation, and thus, lower cost, than otherwise possible.  

Table~\ref{table:experiments_losses_1} provides empirical verification; it reports on 100 experiments in which exactly one Jabr constraint is removed from the formulation. ``Avg Loss" represents the average loss across all branches; ``Avg (br)" and "Min (br)" are the average and minimum (resp.) loss at the branch whose SOC constraint was bypassed; and ``Jabr" denotes the value of the Jabr relaxation, while ``weak Jabr" is the average value of the 100 weakened Jabr SOCs. For reference, ``AC-L" represents the total loss of an AC locally optimal solution, while ``SOC-L" is the total loss of an optimal solution to the Jabr SOC.

\begin{table}[htbp]
    \centering
    \setlength{\tabcolsep}{6pt}
    \caption{Average and minimum losses 100 repetitions of removing a randomly selected SOC constraint.}
    \resizebox{\textwidth}{!}{
    \begin{tabular}{l r r r r r r r}
        \toprule
        {Case} & \multicolumn{1}{c}{Avg Loss} & \multicolumn{1}{c}{Avg (br)} & \multicolumn{1}{c}{Min (br)} & \multicolumn{1}{c}{Jabr} & \multicolumn{1}{c}{weak Jabr} & \multicolumn{1}{c}{AC-L} & \multicolumn{1}{c}{SOC-L} \\
        \midrule
        case14  & -0.3808  & -0.4906 & -1.7443 & 8075.12 & 6292.78 & 0.0929 & 0.0918 \\ 
        case118  & 0.1084 & -0.7046 & -5.1803 & 129340.00 & 126982.72 & 0.7740 & 0.7125 \\ 
        case300 & 1.8485 & -1.1652 & -6.1421 & 718654.00 & 714858.26 & 3.0274 & 2.8064  \\ 
        \bottomrule
    \end{tabular}
    }
    \label{table:experiments_losses_1}
\end{table}

In Table~~\ref{table:experiments_losses_2} we provide additional evidence in the same direction. We report on $100$ experiments, for larger cases, where  we remove $100$ Jabr constraints randomly selected using the uniform distribution. To highlight the relative orders of magnitude of the losses, we exhibit the total load for each case. The column ``Sample" is an average over all the samples of the $100$ repetitions.

\begin{table}[htbp]
    \centering
    \setlength{\tabcolsep}{6pt}
    \caption{Average and minimum losses 100 repetitions of removing 100 randomly selected SOC constraints.}
    \resizebox{\textwidth}{!}{
    \begin{tabular}{l r r r r r r r r}
        \toprule
        & & \multicolumn{4}{c}{Weak Jabr} & \multicolumn{2}{c}{Jabr} & \multicolumn{1}{c}{AC} \\
        \cmidrule(l{0.5em}r{0.40em}){3-6} \cmidrule(l{0.5em}r{0.40em}){7-8} \cmidrule(l{0.5em}r{0.40em}){9-9}
        & & \multicolumn{2}{c}{Avg Losses} & \multicolumn{1}{c}{Min Losses} & &  & \\
        \cmidrule(l{0.5em}r{0.40em}){3-4} \cmidrule(l{0.5em}r{0.40em}){5-5} 
        {Case} & Load & \multicolumn{1}{c}{Total} & \multicolumn{1}{c}{Sample} & \multicolumn{1}{c}{Sample} & \multicolumn{1}{c}{Objective} & \multicolumn{1}{c}{Objective} & \multicolumn{1}{c}{Losses} & \multicolumn{1}{c}{Losses} \\
        \midrule
        1354pegase & 730.60 & -150.38 & -379.60 & -764.52 & 58021.55 & 74008.58 &  9.49 & 10.09 \\ 
        2869pegase & 1324.37 & -197.83 & -402.59 & -1119.85 & 112663.19 & 133872.69 & 14.24 & 15.51  \\
        3375wp & 483.63 & -65.50 & -117.11 & -343.72 & 6467352.82 & 7385372.70 & 6.90 & 8.30  \\
        \bottomrule
    \end{tabular}
    }
    \label{table:experiments_losses_2}
\end{table}

\color{black}

\noindent The following example sheds light on the importance of imposing the non-negative active-power loss requirement. 

\begin{remark}
Solving the McCormick relaxation of the rectangular QCQP formulation~\cite{molzahn+hiskens19} of case3 from MATPOWER, we obtain an optimal solution with zero active-power generation. In particular\footnote{Since total active-power demand is $1.00$ and there is a bus shunt at bus $2$ with shunt conductance $0.1$, and the magnitude of voltage at bus 2 squared is $0.8211$ in our solution, we have that generation equals losses plus the sum of the loads.}, the active-power losses are
\begin{align*}
    \ell_{12} = 0.0789, \quad \ell_{13} = -1.1941, \quad \ell_{23} = 0.0331,
\end{align*}
i.e., total active-power losses equals $-1.0821$.


\begin{figure}[H]
\centering
\includegraphics[scale=1]{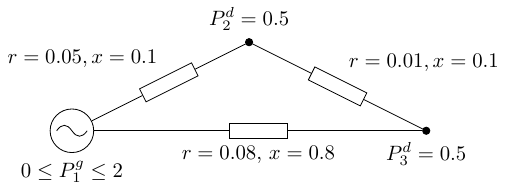} 
\caption{Circuit representation of case3 from MATPOWER.}
\end{figure}


\end{remark}

Given that we obtain active power loss as a linear combination of the power flow equations it is natural to wonder if the SOC relaxations to ACOPF (or outer approximations thereof) accurately approximate active-power losses. In Table~\ref{table:losses} we provide empirical evidence that the quality of the approximations to total losses provided by the SOC relaxations is still very weak -- even though the objective lower bounds proved by the SOC relaxations are quite close to ACOPF primal bounds. 
\begin{table}[htbp]
    \setlength{\tabcolsep}{4pt}
    \centering
    \footnotesize
    \caption{Active-power losses and optimality gaps.}
    \begin{tabular}{ @{} l r r r r r r r r r@{} }
    \toprule
    & \multicolumn{2}{c}{AC} & \multicolumn{2}{c}{Jabr SOC} & & & \\ 
    \cmidrule(l{0.5em}r{0.40em}){2-3} \cmidrule(l{0.5em}r{0.40em}){4-5}
    \multicolumn{1}{l}{Case} & Objective & Losses & Objective & Losses & Opt Gap & Loss Gap \\
    \midrule
        1354pegase & 74069.35 & 10.09 & 74008.58 & 9.48 & 0.08\% & 6.04\% & \\ 
        ACTIVSg10k & 2485898.75 & 23.01 & 2466666.10 & 13.88 & 0.77\% & 39.68\% & \\
        ACTIVSg70k & 16439499.83 & 156.83 & 16217263.66 & 121.31 & 1.35\% & 22.65\% & \\
    \bottomrule
    \end{tabular}

    \label{table:losses}
\end{table}
\color{black}

It must be noted,  however, that from a theoretical perspective the existence of active-power losses is not what makes ACOPF intractable; the NP-hardness proof in~\cite{verma09,bienstock+verma19} relies on a lossless system.

\subsection{Losses and cutting-planes}

Next we provide theoretical justification for using an outer-approximation cutting-plane framework on the Jabr and $i2$ relaxations. We show that for transmission lines with $G_{kk} > 0 > G_{km} = G_{mk} \geq - G_{kk}$ and $B_{km} = B_{mk}$, in particular lines with no transformer or shunt elements, active-power loss inequalities are implied by the Jabr inequalities, and also by the definition of the $i^{(2)}$ variable. We begin with two simple technical observations.

First, consider a (generic) rotated cone inequality 
\begin{align}\label{eq:genericrotated}
& x^{2} + y^{2} \leq wz
\end{align}
where $x,y \in \R$ and $w,z \in \R_{+}$. Then
\begin{subequations}
    \begin{align}
        x^{2} + y^{2} \leq wz &\iff (2x)^{2} + (2y)^{2} \leq (w+z)^{2} - (w-z)^{2} \\
        &\iff ||(2x,2y,w-z)^{\top}||_{2} \leq w+z. \label{eq:rotatedrewrite}
    \end{align}
\end{subequations}
Next, let $\lambda \in \R^3$ satisfy $|| \lambda||_2 = 1$.  Then, by \eqref{eq:rotatedrewrite},
\begin{align}
& 2 \lambda_1 x + 2 \lambda_2 y + \lambda_3(w - z) \le || \lambda||_2 \,
||(2x,2y,w-z)^{\top}||_{2} \leq w+z. \label{eq:genericouter}
\end{align}
Inequality \eqref{eq:genericouter} provides a generic recipe to obtain outer-envelope inequalities for the rotated cone \eqref{eq:genericrotated}. As a result of these developments, we have:
\begin{proposition}\label{prop:jabrouter}
    For a transmission line $\{k,m\} \in \mathcal{E}$ with $G_{kk} > 0 > G_{km} = G_{mk} \geq - G_{kk}$ and $B_{km} = B_{mk}$, the Jabr inequality $c_{km}^{2} + s_{km}^{2} \leq v_{k}^{(2)} v_{m}^{(2)}$ implies, as an outer envelope approximation inequality, that $\ell_{km} \geq 0$.
\end{proposition}
\begin{proof}
See Appendix~\ref{appendix:jabr_implies_lossineq}.
\end{proof}

Moreover, for simple transmission lines the definition of the variable $i^{(2)}$ implies the active-power loss inequalities~\cite{subhonmesh+etal12,coffrin+etal15b}. Actually, if a line $\{k,m\}$ has no shunts, then it is folklore (and follows from basic physics) that $i^{(2)}_{km} \geq 0$ implies $\ell_{km} \geq 0$.


\begin{proposition}\label{proposition:i2_implies_lossineq}
For any line transmission line $\{k,m\} \in \mathcal{E}$ with $y_{km}^{sh} = 0$, the definition of the variable $i^{(2)}_{km}$ implies the active-power loss inequality associated with that line.
\end{proposition}
\begin{proof}
See Appendix~\ref{appendix:i2_implies_lossineq}.
\end{proof}

\subsection{A numerically better-behaved i2 relaxation}\label{subsection:newrelaxation}


The i2 relaxation is at least as strong as the Jabr relaxation, and in congested networks it can be strictly stronger (c.f. Proposition~\ref{proposition:i2_stronger_than_jabr}). Thus, the i2 relaxation is a natural candidate for outer-approximation via a cutting-plane framework (c.f.~\ref{section:cutplane});  we numerically verify the strength of a linear outer-approximation of the i2 relaxation in~\ref{subsection:cutcomputations}. The strength of this relaxation comes at a price: numerical instability. To be precise, the potentially large coefficients in the definition of the $i^{(2)}$ variables can result in a ill-conditioned SOC instance (see, e.g.,~\cite{klotz14}). Next we illustrate this issue with a simple case to show how the Jabr outer-envelope inequalities complement the strength of the $i^{(2)}$ definitions. We conclude this subsection with a new, numerically stable modified i2 relaxation, which we term i2+($\rho$), where $\rho > 0$. 

Let $\{k,m\}$ be a transmission line with a limit $U_{km} < +\infty$, and without any transformers or shunt elements. By~\eqref{appendix:i2def_eq} we have
        \begin{equation}\label{i2simpleline}
            i_{km}^{(2)} = \left( \frac{1}{r_{km}^{2} + x_{km}^{2}} \right) \left( v_{k}^{(2)} + v_{m}^{(2)} - 2  c_{km}  \right)
        \end{equation}
where $r_{km}$ and $x_{km}$ denote the line's  resistance and reactance. Moreover, equation~\eqref{eq:power2def} implies that $i_{km}^{(2)}$ can be upper-bounded by the constant $H_{km}: = U_{km} / (V_{k}^{\min})^{2}$, where $V_{k}^{\min}$ corresponds to the lower limit for voltage magnitude at bus $k$ (both voltage limits are close to $1$ under standard data conditions).

Suppose that this simple line $\{k,m\}$ has small resistance, e.g., on the order of $10^{-5}$. Then $1/(r_{km}^{2} + x_{km}^{2})$ can be on the order of $10^{8}$. In addition, assume that the limit $U_{km}$ is \emph{small}. Thus,  $(r_{km}^{2} + x_{km}^{2}) H_{km}$ can be fairly small, which yields
\begin{equation}\label{surfacejabrcone}
    v_{k}^{(2)} + v_{m}^{(2)} - 2 c_{km} \leq (r_{km}^{2} + x_{km}^{2}) H_{km} \approx 0.
\end{equation}
Since $v_{k}^{(2)} +  v_{m}^{(2)}  - 2 c_{km} \geq 0$ is a Jabr outer-envelope cut (c.f. proof Proposition~\ref{prop:jabrouter}), inequality~\eqref{surfacejabrcone} is forcing our solutions to lie near the surface of the rotated-cone $c_{km}^{2} + s_{km}^{2} \leq v_{k}^{(2)} v_{m}^{(2)}$, i.e., it is cutting-off solutions which are in the interior of the Jabr cone for line $\{k,m\}$. Notice that the more capacitated the transmission line is, the stronger the effect of the $i^{(2)}$ variables (c.f. Figure~\ref{fig:badi2s}).

This key observation motivates a numerically more stable relaxation. Consider the definition of $i^{(2)}_{km}$ in~\ref{appendix:i2def} for appropriate constants $\alpha_{km}, \beta_{km}, \gamma_{km}$ and $\zeta_{km}$:
\begin{equation}\label{goodi2s}
    i_{km}^{(2)} = \alpha_{km} v_{k}^{(2)} + \beta_{km} v_{m}^{(2)} + \gamma_{km} c_{km} + \zeta_{km} s_{km}.
\end{equation}
Let $\rho > 0$ be some fixed parameter, and consider the following heuristic: for every transmission line $\{k,m\} \in \mathcal{E}$,
\begin{enumerate}
    \item[(i)] if $\alpha_{km} > \rho$, we say that the branch induces a \emph{bad-i2}, and we define the following inequalities (which are valid by the previous discussion)
    \begin{subequations}\label{badi2s}
    \begin{align}
        v_{k}^{(2)} + \frac{\beta_{km}}{\alpha_{km}} v_{m}^{(2)} + \frac{\gamma_{km}}{\alpha_{km} } c_{km} + \frac{\zeta_{km}}{\alpha_{km} } s_{km} &\geq 0, \label{badi2s_1}\\
        v_{k}^{(2)} + \frac{\beta_{km}}{\alpha_{km}} v_{m}^{(2)} + \frac{\gamma_{km}}{\alpha_{km} } c_{km} + \frac{\zeta_{km}}{\alpha_{km} } s_{km} &\leq \frac{H_{km}}{\alpha_{km}}; \label{badi2s_2}
    \end{align}
    \end{subequations}
    \item[(ii)] otherwise, we say that the branch induces a \emph{good-i2}, and we define $i_{km}^{(2)} = \alpha_{km} v_{k}^{(2)} + \beta_{km} v_{m}^{(2)} + \gamma_{km} c_{km} + \zeta_{km} s_{km}$.
\end{enumerate}

\begin{figure}[H]
\centering
\includegraphics[scale=1]{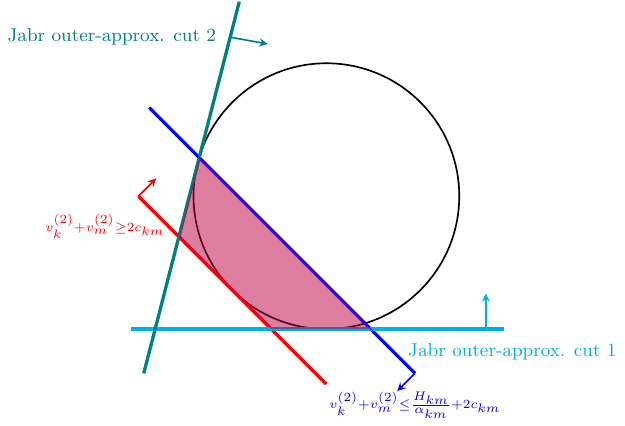} 
\caption{The black circle represents a horizontal cross-section of a projected Jabr cone~\eqref{SOCP:jabr} onto the $(v_{k}^{(2)},v_{m}^{(2)},c_{km})$-space. The red and blue inequalities represent~\eqref{badi2s} for a simple transmission line, while the teal and cyan inequalities represent two Jabr outer-approximation cuts.}\label{fig:badi2s}
\end{figure}

Empirically, we see in Table~\ref{table:coeffs_badi2s} that the coefficients of the inequalities~\eqref{badi2s} are relatively small; clearly closer to the range of coefficients of the linear constraints of the Jabr SOCP, and to the coefficients\footnote{The size of coefficients of these cuts are linear in the values of variables $c,s$ and $v$ in any solution to a linear relaxation of the Jabr SOCP; see Proposition~\eqref{proposition:projection}.} of the Jabr outer-envelope cuts which will be deployed in our cutting-plane algorithm (c.f. the linearly constrained relaxation $M$ in~\ref{subsubsection:basicalgorithm_multi}).


\begin{table}[htbp]
    \setlength{\tabcolsep}{3pt}
    \centering
    \footnotesize
    \caption{Statistics of the coefficient $\alpha$ in~\eqref{goodi2s} and of the coefficients in~\eqref{badi2s}.}
    \resizebox{\textwidth}{!}{
    \begin{tabular}{ @{} l r r | r r r r r r r r r r@{} }
    \toprule
    & \multicolumn{2}{c}{$\alpha$} & \multicolumn{2}{c}{$\frac{\beta}{\alpha}$} & \multicolumn{2}{c}{$\frac{\gamma}{\alpha}$} & \multicolumn{2}{c}{$\frac{\zeta}{\alpha}$} & \multicolumn{3}{c}{$\frac{H}{\alpha}$} & \\
    \cmidrule(l{0.5em}r{0.40em}){2-3} \cmidrule(l{0.5em}r{0.40em}){4-5} \cmidrule(l{0.5em}r{0.40em}){6-7} \cmidrule(l{0.5em}r{0.40em}){8-9} \cmidrule(l{0.5em}r{0.40em}){10-12}
    \multicolumn{1}{l}{Case} & Max & Avg & Max & Avg & Max & Avg & Max & Avg & Max & Min & Avg & \\
    \midrule
        ACTIVSg10k & 1e10 & 1092402.44 & 1.47 & 1.01 & -1.80 & -2.01 & 0.88 & 0.00 & 18.52 & 2.57e-09 & 3.06 \\
        10000goc-sad & 2233637.70 & 241342.40 & 1.21 & 1.00 & -1.90 & -2.00 & 0.01 & 0.00 & 1.89 & 4.64e-06 & 0.48  \\
        30000goc-api & 24134013.85 & 366208.30 & 1.31 & 1.00 & -1.94 & -2.00 & 0.01 & 5.90e-06 & 1.89 & 2.16e-06 & 0.71  \\
    \bottomrule
    \end{tabular}
    }
    \label{table:coeffs_badi2s}
\end{table}

\noindent By making reasonable assumptions about transmission line parameters we can derive the following upper bounds for the coefficients in inequalities~\eqref{badi2s}.

\begin{proposition}\label{proposition:boundedcoeffs}
    Let $\{k,m\} \in \mathcal{E}$ be a transmission line that satisfies $| b | > b^{sh}$, $g^{sh} = 0$ and $|b| > g$. Then, the coefficients $\frac{\beta_{km}}{\alpha_{km} }$, $\frac{\gamma_{km}}{\alpha_{km} }$, and $\frac{\zeta_{km}}{\alpha_{km} }$ in~\eqref{badi2s} are bounded by $\max \{ \tau_{km}^{2}, 3\tau_{km}\}$, where $\tau_{km}$ is the transformer tap ratio located at bus $k$.\footnote{If the line does not have a transformer, then $\tau_{km} = 1$; generally, tap ratios do not take extreme values.}
\end{proposition}
\begin{proof}
See Appendix~\ref{appendix:boundedcoeffs}.
\end{proof}

In consequence, given $\rho > 0$, we define our numerically better-behaved i2+($\rho$) relaxation by \eqref{AC:theobjective}-\eqref{AC:reactivepowerbal},~\eqref{AC:genactivepowerlimit}-\eqref{AC:capacity}, the definition of $i_{km}^{(2)}$~\eqref{appendix:i2def_eq} and the rotated-cone inequalities~\eqref{i2ineq} for every good-i2 branch $\{k,m\}$, and inequalities~\eqref{badi2s} for every bad-i2 branch $\{k,m\}$. We use the i2+($\rho$) relaxation for our multi-period experiments (see discussion and results in~\ref{subsection:experiments_multi}).

\section{Cutting-plane Framework: Cut Separation and Management}\label{section:cutplane}

In this paper we use a dynamically generated linearly-constrained relaxation as a lower bounding procedure for ACOPF; relying on a perspective sometimes associated with integer programming.

Given a set $X$ in $\mathbb{R}^{n}$, we say that a convex inequality $g(x) \leq d$ is \emph{valid} for $X$ if $g(x) \leq d$ holds for all $x \in X$. Given $X$ and $\overline{x} \notin X$, then we say that $c^{\top} x \leq d$ is a (linear) \emph{cut} for $X$ if the inequality is valid for $X$ but it is not satisfied by $\overline{x}$. A (linear) cutting-plane algorithm~\cite{schrijver98} for a set $X$ is an iterative procedure in which, starting from an initial (linear) relaxation, in every round (linear) cuts are added to approximate $X$. Typically, these cuts are computed in iterative fashion; at each round an optimal solution $\overline{x} \notin X$ to the current relaxation is separated using a computed cut.


In the case of single-period experiments our target set $X$ is the i2 relaxation of ACOPF, while we use the i2+ relaxation for the multi-period setting. The reason is simple: the multi-period i2 relaxation for medium to large networks is not numerically tractable. 

In support of this statement, we note that direct solution of the Jabr and i2 relaxations of ACOPF, for large instances, is computationally prohibitive and often results in non-convergence (c.f. tables~\ref{table:jabrsocps},~\ref{table:i2socps},~\ref{table:ws_loads_1_1},~\ref{table:ws_onelineoff},~\ref{table:ws_loads_T4_i2socp+},~\ref{table:ws_loads_T12_i2socp+} and~\ref{table:ws_loads_T24_i2socp+}). Our empirical evidence further shows that outer-approximation of the rotated-cone inequalities (in either case) requires a large number of cuts in order to achieve a tight relaxation value. Moreover, employing such large families of cuts yields a relaxation that, while linearly constrained, still proves challenging -- both from the perspective of running time and numerical tractability. Nonetheless, a characteristic feature of our iterative procedure is its robustness to potentially suboptimal termination of the oracle used to solve the LPs or convex QPs; \emph{independent} of the quality of the primal solution obtained, our linear cuts will always be valid.

However, as we show, adequate cut management proves successful, yielding a procedure that is (a) rapid, (b) numerically stable, and (c) constitutes a very tight relaxation (c.f. tables~\ref{table:ws_loads_1_1} and~\ref{table:ws_onelineoff}).  The critical ingredients in this procedure are: (1) fast cut separation; (2) appropriate violated cut selection; and (3) effective dynamic cut management, including rejection of \textit{nearly-parallel} cuts and removal of \textit{expired} cuts, i.e., previously added cuts that are slack (cf.~\ref{subsection:basicalgorithm}).

Our procedure possesses efficient warm-starting capabilities -- this is a central goal of our work. Previously computed cuts, for some given instance, can be re utilized and loaded into new runs of \emph{related} instances, hence leveraging previous computational effort. It is worth noting that this reoptimization feature stands in sharp contrast to what is possible using nonlinear (convex) solvers. In~\ref{subsection:warmstarts} we justify the validity of this feature and Tables~\ref{table:ws_loads_1_1} and~\ref{table:ws_onelineoff} summarize extensive numerical evidence on its performance relative to solving the SOCPs `from scratch'. We remark that adequate cut management is what enables this capability in the case of large single and multi-period ACOPF instances.

\subsection{Two Simple Cut Procedures}\label{subsection:cuts}
The following (routine) results provide a fast procedure for separating over the rotated-cone inequalities
\begin{equation}\label{rotatedconeinequalities}
    c_{km}^{2} + s_{km}^{2} \leq v_{k}^{(2)} v_{m}^{(2)}, \hspace{2em} P_{km}^{2} + Q_{km}^{2} \leq v_{k}^{(2)} i_{km}^{(2)}.
\end{equation}

%

\begin{proposition}\label{proposition:projection}
    Consider the second-order cone $C := \{(x,s) \in \R^{n} \times \R_{+} \, : \, ||x||_{2} \leq s\}$. Suppose $(x',s') \notin C$ with $s' \geq 0$. Then the projection of $(x',s')$ onto $C$ is given by
    \begin{equation*}
        x_{0} := s_{0} \frac{x'}{||x'||}, \hspace{1em} s_{0} := \frac{||x'|| + s'}{2}.
    \end{equation*}
    Moreover, the hyperplane which achieves the maximum distance from $(x',s')$ to any hyperplane separating $C$ and $(x',s')$ is given by
    \begin{equation*}
        (x')^{t} x  \leq s || x' ||.
    \end{equation*}
\end{proposition}
\begin{proof}
    See Appendix~\ref{appendix:proofprojection}.
\end{proof}

\begin{corollary}\label{cor:cuts}
    Let $C := \{ (x,y,w,z) \in \R^{2} \times \R^{2}_{+} \, : \, x^{2} + y^{2} \leq wz \} \subseteq \R^{4}$ and suppose that $(x',y',w',z') \notin C$ where $w'+z'>0$. Then the hyperplane separating $C$ from $(x',y',w',z')$ and at maximum distance from this point is given by
    \begin{equation}
        (4x')x + (4y')y + ((w'-z')-n_{0}) w + (-(w'-z')-n_{0})z \leq 0,
    \end{equation}
    where $n_{0} := ||(2x',2y',w'-z')^{\top}||_{2}$.
\end{corollary}
\begin{proof}
     Rewriting the rotated-cone inequality as \eqref{eq:rotatedrewrite} and a direct application of Proposition~\ref{proposition:projection} gives us the desired separating hyperplane.
\end{proof}

The following proposition gives us a simple procedure for computing linear cuts for violated thermal-limit inequalities 
\begin{equation}\label{thermallimit}
    P_{km}^{2} + Q_{km}^{2} \leq U_{km}.
\end{equation}

\begin{proposition}\label{proposition:thermalcuts}
    Consider the Euclidean ball in $\R^{2}$ of radius $r$, $S_{r}:= \{ (x,y) \in \R^{2} \, : \, x^{2} + y^{2} \leq r^{2}\}$, and let $(x',y') \notin S_{r}$. Then the separating hyperplane that is at maximum distance from $(x',y')$ is given by
    \begin{equation}
        (x')x + (y')y \leq r ||(x',y')||_{2}.
    \end{equation}
\end{proposition}
\begin{proof}
    See Appendix~\ref{appendix:proofthermalcuts}.
\end{proof}

\subsection{Basic Cutting-Plane Algorithm}\label{section:thealgorithm}

\subsubsection{Single-period}\label{subsection:basicalgorithm}

In what follows we describe our cutting-plane algorithm for the single-period setting. First, we define a linearly constrained base model $M_{0}$ as follows:
\begin{subequations}\label{basemodel}
\begin{align}
[M_{0}]: &\hspace{2em}\min \hspace{2em} \sum_{k \in \mathcal{G}} F_{k}(P_{k}^{g}) \\
    \text{subject to:}& \nonumber \\
    &\text{constraints } \eqref{AC:activepowerbal},\eqref{AC:reactivepowerbal},\eqref{JABR:flows}, \eqref{AC:genactivepowerlimit}-\eqref{AC:genreactivepowerlimit}, \eqref{AC:voltlimit}
\end{align}
\end{subequations}
In other words, we consider the linearized power flow equations of the Jabr SOCP and all the linear constraints in~\eqref{AC:firstformulation}. Recall that for single-period ACOPF, our target set $X$ is the i2 relaxation.

In every round of our iterative procedure, linear constraints will be added to and removed from the current relaxation, which is initialized as $M_{0}$. We will denote by $M$ the relaxation at a generic iteration of our cutting-plane algorithm.

Given a feasible solution $\bar{x}$ to $M$, and letting $f_{km}(x) \leq 0$ be some valid convex inequality~\eqref{rotatedconeinequalities} or~\eqref{thermallimit}, our measure of \emph{cut-quality} is ${\max \{ f_{km}(\overline{x}), 0 \}}$, i.e.,  the amount by which $\overline{x}$ violates the inequality. Let $\epsilon > 0$. For each type $\tau \in \{ \text{Jabr}, \text{i2}, \text{limit} \}$ of inequality we select the top $p_{\tau}$ percent --a fixed parameter-- most highly violated inequalities of type $\tau$ whose violation is greater than $\epsilon$ -- these are the $\tau$-candidates (Line 7 of the pseudocode). 

For each list of $\tau$-candidates, we compute cuts for the corresponding branches using the cut procedures described in~\ref{subsection:cuts}. Candidate cuts will be rejected if they are \emph{too parallel} to incumbent cuts in $M$ -- near-linear dependency is a common source of ill-conditioning in optimization models~\cite{higham96,klotz14}. To be precise, given $\epsilon_{par} > 0$, we say that two linear inequalities $c^{t}x \leq 0$ and $d^{t} x \leq 0$ are \emph{$\epsilon_{par}$-parallel} if the cosine of the angle between their normal vectors $c$ and $d$ is strictly more that $1 - \epsilon_{par}$. 

Finally, we describe a heuristic for \emph{cleaning-up} our formulation. For each added cut, we keep track of its \emph{cut-age}, i.e., the count of rounds since it was added. Then, in every iteration, if a cut $c^{\top}x \leq d$ has age greater or equal than a fixed parameter $T_{age} \in \mathbb{N}$, and it is \emph{$\epsilon$-slack}, i.e., $d - c^{\top} \overline{x} > \epsilon$, then it is dropped from $M$.
\begin{algorithm}
\caption{Cutting-Plane Algorithm}\label{thealgorithm}
\begin{algorithmic}[1]
\Procedure {Cutplane}{}
\State Initialize $r \gets 0$, $M \gets M_{0}$, $z_{0} \gets + \infty$
\While{$t < T$ and $r < T_{ftol}$}
\State $z \gets \min M$ and $\bar{x} \gets \text{argmin} \, M$
\State Check for violated inequalities by solution $\overline{x}$.
\State Sort inequalities by violation.
\State Compute cuts for the most violated inequalities.
\State Add computed cuts if they are not $\epsilon$-parallel to existing cuts in $M$.
\State Drop cuts of age $\geq T_{age}$ whose slack is $\geq$ $\epsilon_{j}$.
\If{$z - z_{0} < z_{0} \cdot \epsilon_{ftol}$}
\State $r \gets r+1$
\Else
\State $r \gets 0$
\EndIf
\State $z_{0} \gets z$
\EndWhile
\EndProcedure
\end{algorithmic}
\end{algorithm}

In addition to $M_{0}$ and the parameters $p_{\tau}, \epsilon, \epsilon_{par}, T_{age}$, other inputs for our procedure are: a time limit $T>0$; the number of admissible iterations without sufficient objective improvement $T_{ftol} \in \mathbb{N}$; and a threshold for objective relative improvement $\epsilon_{ftol} > 0$.



\subsubsection{Multi-period}\label{subsubsection:basicalgorithm_multi}

We outer-approximate the i2+ relaxation in the multi-period setting. Thus, the linearly constrained base model $M_{0}^{T}$ is given by
\begin{subequations}\label{basemodel_multiperiod}
\begin{align}
[M_{0}^{T}]: &\hspace{2em}\min \hspace{2em} \sum_{t=0}^{T-1} \sum_{k \in \mathcal{G}} F_{k}(P_{k,t}^{g}) \\
    \text{subject to:}& \nonumber \\
    &\text{constraints } \eqref{AC:activepowerbal},\eqref{AC:reactivepowerbal},\eqref{JABR:flows}, \eqref{AC:genactivepowerlimit}-\eqref{AC:genreactivepowerlimit}, \eqref{AC:voltlimit} \\
    &\hspace{5em} \text{for every good-i2 branch $\{k,m\}$},~\eqref{appendix:i2def_eq},~\eqref{i2ineq} \\
    &\hspace{5em} \text{for every bad-i2 branch $\{k,m\}$},~\eqref{badi2s}.
\end{align}
\end{subequations}
In other words, we consider all the linear constraints of our i2+ relaxation.

For each time-period $t$, we sort violated inequalities $\tau \in \{\text{Jabr}, \text{i2}, \text{limit}\}$ by violation, as in the single-period setting, and pick as $(\tau,t)$-candidates branches the top $p_{\tau}$ percentage of the most violated branches per period $t$. The rest follows as in Algorithm~\ref{thealgorithm}.

Our experiments show that our cut selection is robust, and that Gurobi version 10.0.1~\cite{gurobi} is able to handle surprisingly well large volume of cuts, see tables~\ref{table:cp_scratch_T4}~\ref{table:ws_loads_T4_jabrsocp},~\ref{table:ws_loads_T4_i2socp+}.

\section{On the Accuracy of Lower Bounds}\label{section:accuracy}


In this section we address an issue of fundamental importance in numerical optimization, and which, in our opinion, has not received adequate attention in the ACOPF literature.



Most optimization solvers, commercial and academic, work with finite precision, floating-point arithmetic (some notable exceptions are the LP solvers~\cite{applegate+etal07} and~\cite{gleixner+etal16}) and are subject to roundoff errors, usually small. To be precise, whenever an optimization instance is ``solved" and declared ``optimal" by a solver, it is highly likely that the provided solution will have small infeasibilities, i.e., the solution will violate the instance's constraints up to some additive or multiplicative quantity. Usually, this \emph{infeasibility tolerance} can be controlled by the user. 

A natural question then is: given an $\epsilon$-feasible $\tilde{x}$ solution to a convex optimization problem $[P]$ with optimal solution $\overline{x}$ (likely unknown by the user), how \emph{superoptimal} can the approximate solution be? In other words, are there any general guarantees which permit us bound the superoptimality of approximate solutions? For example, if $f$ denotes the objective function for $[P]$, we seek a guarantee of the form  
\begin{equation*}
    f(\tilde{x}) \geq f(\overline{x}) - h_{P}(\epsilon)
\end{equation*}
for a certain function $h_P$ which only depends on the instance's data. Ideally, we would like for $h_{P}$ to be monotonic on $\epsilon$ and polynomial on the size (bit-length) of the problem $[P]$ data. If $[P]$ is a linear program then $h_{P}$ is linear function whose slope coefficient has bit-length polynomial on the problem's bit encoding~\cite{schrijver98}. Moreover, if $[P]$ is a convex QP we also have a good guarantee as in the LP case. Our proof leverages an argument used to show that QP is in NP~\cite{vavasis90}.


\begin{proposition}\label{proposition:superoptimality}
    Let $[P]$ be a convex QP with objective given by $f(x):= x^{\top} H x + c^{\top} x$, where $H \in \Q^{n \times n}$ is a positive definite matrix and $c \in \Q^{n}$, and feasible set $\mathcal{X} := \{ x \in \R^{n} \, : \, Ax \geq b\}$, where $A \in \Q^{m \times n}$ and $b \in \Q^{m}$. Let $\overline{x} \in \Q^{n}$ be an optimal solution to this convex QP. Suppose there is a point $\tilde{x} \in \Q^{n}$ that is $\epsilon$-feasible for $[P]$. Then,
    \begin{equation}\label{superopt_ineq}
        f(\tilde{x}) \geq f(\overline{x}) - || \overline{\lambda} ||_{1}\epsilon.
    \end{equation}
    where $\overline{\lambda}$ is an optimal solution to the dual of $[P]$. 
    
    Moreover, $|| \overline{\lambda} ||_{1}$ can be bounded by a constant $g(A,b,H,c)$ whose bit-length is polynomial on the bit-complexity of the input $(A,b,H,c)$.
\end{proposition}
\begin{proof}
    See Appendix~\ref{appendix:superoptimality}.
\end{proof}

\begin{remark}
    This proposition can be readily applied to the dual of the convex QP which is again a convex QP (the quadratic term in the objective is negative definite). Indeed, if we denote by $q$ the dual objective and by $\tilde{\lambda}$ an $\epsilon$-feasible dual solution, we can apply our result to the convex QP with objective $q':= -q$ which yields
    \begin{equation*}
        q'(\tilde{\lambda}) \geq q'(\overline{\lambda}) - || \overline{x} ||_{1} \epsilon \iff q(\tilde{\lambda}) \leq q(\overline{\lambda}) + || \overline{x} ||_{1} \epsilon.
    \end{equation*}
\end{remark}

Next, we elaborate on the relevance of Proposition~\ref{proposition:superoptimality}. Suppose we are given:
\begin{itemize}
\item A (not necessarily convex) optimization problem $[Z]$ whose optimal objective value is denoted by $\overline{z}$;
\item A convex relaxation $[P]$ of $[Z]$, with objective function $f$, and its dual $[D]$, with objective function $q$;
\item An $\epsilon$-feasible primal-dual pair $(\tilde{x},\tilde{\lambda})$ and an optimal primal-dual pair $(\overline{x},\overline{\lambda})$ for $[P]$-$[D]$.
\end{itemize}
Assume that strong duality holds\footnote{An analogous analysis can be done is there exists a duality gap; we assumed strong duality for clarity of explanation.} for the pair of convex problems $[P]$-$[D]$. Figure~\ref{fig:superopt} depicts two inherent risks when dealing with $\epsilon$-feasible solutions: incurring in (1) invalid lower bounds, and (2) poor primal bounds for $[Z]$. Indeed, (1) an \textit{infeasible} (even if $\epsilon$-feasible) dual solution could provide an invalid lower bound for $[Z]$, i.e., $\overline{z} < q(\tilde{\lambda})$. On the other hand, (2) if a primal $\epsilon$-solution $\tilde{x}$ to the relaxation $[P]$ is used as an approximate or even an exact solution\footnote{For instance, for radial networks the Jabr SOCP relaxation is exact.} to $[Z]$ and we do not have guarantees such as those provided by Proposition~\ref{proposition:superoptimality}, then the value $f(\tilde{x})$ may be significantly off (far from $\overline{z}$). 

\begin{figure}[H]
\centering
\includegraphics[scale=0.85]{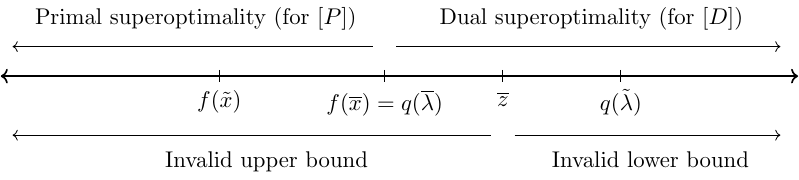} 
\caption{Validity of bounds and superoptimality of $\epsilon$-feasible solutions.}\label{fig:superopt}
\end{figure}

Results such as inequality~\eqref{superopt_ineq} are appealing since the bound has a linear dependence on $\epsilon$. Unfortunately, small dual violations are not the only ingredient for accurate bounds; problem structure is key. In~\ref{appendix:superopt_socps} we show that if $[P]$ and $[D]$ are a primal-dual pair of SOCPs for which strong duality holds, then any $\epsilon$-feasible solution $\tilde{x}$ satisfies $c^{\top} \tilde{x} \geq c^{\top}\overline{x} - ( ||\overline{\lambda}||_{1} + ||\overline{u} ||_{1}) \epsilon$ where $c$ denotes the primal objective and $(\overline{x},(\overline{\lambda},\overline{u}))$ a primal-dual optimal pair. Thus, a natural question is: can we always provide a 'reasonable' bound on $||\overline{\lambda}||_{1} + ||\overline{u} ||_{1}$? From a fundamental perspective, such a bound should be polynomial on the bit-length of the problem data -- as in the case for LPs or convex QPs\footnote{The existence of SOCPs whose unique solution is irrational are folklore results, see~\cite{bienstock+etal23} for examples.}. In what follows we provide a negative answer to the latter question even for the case of the Jabr SOCP for the simplest possible network -- two buses and one transmission line.

\begin{proposition}\label{proposition:nastyACOPF}
    There exists an ACOPF instance on a two-bus network with standard line parameters and rational data, where the only two AC feasible solutions are irrational. Additionally, the unique minimizer of the Jabr SOCP in this instance is also irrational. Furthermore, the dual problem to this Jabr SOCP has a unique maximizer, which is irrational as well.
\end{proposition}

\begin{proof}
    See Appendix~\ref{appendix:nastyACOPF}.
\end{proof}

\begin{remark}
Interestingly, this proposition implies that the feasible region of an ACOPF instance can be disconnected even for the simplest possible non-trivial network and that its convex hull is described by the feasible region of the Jabr SOCP (since the network is radial~\cite{jabr06}). Also see the classical example in \cite{hiskens+davy01}. Moreover, this example shows that to answer whether AC-Feasibility~\cite{lehmann+etal16,bienstock+verma19} is in NP or not one cannot rely on an AC rational feasible solution as a polynomial-size certificate.
\end{remark}

In summary, LP and convex QP relaxations possess robust theoretical guarantees for bounding challenging optimization problems, in sharp contrast to what nonlinear relaxations such as SOCPs can offer. Given that this paper focuses on lower bounds for ACOPF, we are particularly interested in dual infeasibility of primal-dual solutions provided by the solvers (the primal being a convex relaxation of ACOPF). 
In~\ref{subsection:experiments_multi} we report on dual infeasibilities attained by solvers run on our linearly constrained models.

\section{Computational Experiments}\label{section:experiments}

We ran all of our experiments on an Intel(R) Xeon(R) 64-bit Linux machine, with 2 E5-2687W v3 $3.10$GHz CPUs, $20$ physical cores, $40$ logical processors, and $256$ GB RAM. We used three commercial solvers: Gurobi version 10.0.1~\cite{gurobi}, Artelys Knitro versions 13.2.0 and 14.1.0 beta~\cite{knitro}, and Mosek 10.0.43~\cite{mosek}. As a common interface for all SOCP and ACOPF models we used AMPL~\cite{ampl} via Python. We note that unlike Gurobi and Knitro, Mosek does not accept a constraint of the form $x^2 + y^2 \leq z^2$ or $x^2 + y^2 \leq wz$ as a conic constraint; thus we reformulated the SOCPs using a format that Mosek-AMPL was able to read. 

We report extensive numerical experiments on instances from the following data sets:
medium and large-sized standard IEEE instances available from MATPOWER~\cite{zimmerman+etal11},
the Pan European Grid Advanced Simulation and State Estimation (PEGASE) project~\cite{josz+etal16}~\cite{fliscounakis+13}, ACTIVSg synthetic cases developed as part of the US ARPA-E GRID DATA research project ~\cite{birchfield+etal17}, ~\cite{birchfield+etal18}, and the largest instances from the Power Grid Library for Benchmarking AC Optimal Power Flow Algorithms~\cite{babaeinejadsarookolaee+etal21} (PGLIB). Our main focus are cases with 9000 buses or more.

Our cutting-plane algorithm is implemented in Python 3 and uses Gurobi 10.0.1 to solve LPs and convex QPs. All of our reported single-period experiments were obtained with the following parameter setup (c.f.~\ref{section:thealgorithm}): barrier convergence tolerance and absolute feasibility and optimality tolerances equal to $10^{-6}$, $\epsilon = 10^{-5}$, $p_{Jabr} = 0.55$, $p_{i2} = 0.15$, $p_{limit} = 1$, $T_{age} = 5$, $\epsilon_{par} = 10^{-5} /2$, $\epsilon_{ftol} = 10^{-5}$, and $T_{ftol} = 5$. For our multi-period experiments we used the numerically better-behaved i2+($\rho)$ relaxation with $\rho = 10^{2}$, and, as a result, we were able to increase solution accuracy: we set absolute feasibility and optimality tolerances equal to $10^{-8}$, and we relaxed $T_{ftol}$ from $5$ to $3$ for $T=12$ and to $2$ for $T=24$. All of our code, AMPL model files, and solution files can be downloaded from \url{www.github.com/matias-vm}.

Next we describe the symbols used in our tables. The character $``-"$ denotes that the solver did not converge, while  $``\text{TLim}"$ means that the solver did not converge within our time limit. By convergence we mean that the solver \emph{declares} to have obtained an \emph{optimal} solution, within the previously defined tolerances. 
Further, $``\text{INF}"$ means that the instance was declared infeasible by the solver, while $``\text{LOC INF}"$, used by KNITRO, means that the solver converged to an infeasible point. Moreover, if Gurobi declares 
\textit{numerical trouble} while solving our LPs or convex QPs at some iteration of our cutting-plane algorithm, we report the objective value of the previous iteration followed by the character $``*"$. The objective values and running times are reported with 2 decimal places.

We remark that, to the best of our knowledge, this is the first computational study which compares the performance of three leading commercial solvers on the Jabr SOCP using a common framework (AMPL). We evaluate the \textit{solvers} on the Jabr SOCP, and compare their performance with that of our warm-started i2+ formulations. For the solvers we use Jabr instead of the i2 SOCP because the former is numerically better behaved from the solvers' perspective (c.f.~\ref{subsection:newrelaxation}). We report on these numerical issues in~\ref{subsubsection:performance_socps}; also compare tables~\ref{table:ws_loads_T4_jabrsocp} and~\ref{table:ws_loads_T4_i2socp+}.

\subsection{Single-Period Experiments} 

\paragraph{SOCPs: Solver parameters} We set a time limit of $1,000$ seconds for all of our SOCP experiments\footnote{One iteration of our algorithm requires approximately $ \, 60-100$ seconds.}. Next we describe additional parameter specifications for each solver.
\begin{itemize}
    \item We use Gurobi's default homogeneous self-dual embedding interior-point algorithm (barrier method without \emph{Crossover}, and \emph{Bar Homogeneous} set to $1)$, and we set the parameter \emph{Numeric Focus} equal to $1$. Barrier convergence tolerance and absolute feasibility and optimality tolerances were set to $10^{-6}$. Gurobi was allowed 20 threads.
    \item We use Knitro's default Interior-Point/Barrier Direct Algorithm, with absolute feasibility and optimality tolerances equal to $10^{-6}$. We used the HSL MA57 sparse symmetric indefinite linear solver, and the Intel Math Kernel Library (MKL) functions for Basic Linear Algebra Subroutines (BLAS), i.e., for basic vector and matrix computations. Moreover, $20$ threads were used with Knitro. SOCPs were explicitly input to Knitro  as convex problems. When computing primal bounds, we employed the linear solver HSL MA97 whenever Knitro under MA57 was not converging.
    \item We use Mosek's default homogeneous and self-dual interior-point algorithm for conic optimization. We set the relative termination tolerance, as well as primal and dual feasibility tolerances to $10^{-6}$. We used $20$ threads with Mosek.
\end{itemize}

\begin{remark}
    We did not set a time limit for computing single-period ACOPF primal bounds.
\end{remark}

\subsubsection{Cut Computations}\label{subsection:cutcomputations}

Table~\ref{table:cuts} summarizes outcomes on cut computations for a substantial number of instances from the libraries described above. Under ``Cutting-Plane," ``Objective" reports the objective of the last iteration of our algorithm,   ``Time" reports the total running time (in seconds) of our method;  ``Computed" reports the number of cuts computed throughout the entire procedure;  ``Added" exhibits the total number of cuts in our linearly constrained relaxation at the last round; 
and ``Rnd" is the number of rounds or our algorithm. Finally,  ``Primal bound" reports the objective value of a feasible solution to ACOPF and the amount of time (in seconds) required by Knitro for this purpose.

\begingroup
\setlength{\tabcolsep}{3pt}
\begin{table}
\caption{Performance of Cutting-Plane Algorithm (Not Warm-Started)}
\centering
\resizebox{\textwidth}{!}{
\begin{tabular}{ @{} l r r r r r r r r@{} }
\toprule
& \multicolumn{5}{c}{Cutting-Plane Algorithm} & \multicolumn{2}{c}{Primal bound} & \\
\cmidrule(l{0.5em}r{0.40em}){2-6} \cmidrule(l{0.5em}r{0.40em}){7-8}
\multicolumn{1}{l}{Case} & Objective & Time & Computed & Added & Rounds & Objective & Time & \\
\midrule
9241pegase & 309221.81 & 378.82 & 135599 & 29875 & 23 & 315911.56 & 96.74 & \\
9241pegase-api & 6924650.57 & 277.32 & 128316 & 30230 & 21 & 7068721.98 & 73.85 & \\
9241pegase-sad & 6141202.28 & 386.51 & 113686 & 27273 & 21 & 6318468.57 & 33.92 & \\
9591goc-api & 1346373.10 & 187.26 & 87812 & 22469 & 22 & 1570263.74 & 42.85 & \\
9591goc-sad & 1055493.25 & 246.87 & 90153 & 20514 & 27 & 1167400.79 & 28.15 & \\
ACTIVSg10k & 2476851.62 & 132.16 & 60803 & 18183 & 19 & 2485898.75 & 76.71 & \\
10000goc-api & 2502026.03 & 147.12 & 73084 & 19666 & 24 & 2678659.51 & 23.46 & \\
10000goc-sad & 1387303.02 & 114.97 & 58984 & 18528 & 17 & 1490209.66 & 103.06 & \\
10192epigrids-api & 1849488.30 & 152.87 & 97921 & 24882 & 22 & 1977687.11 & 117.15 & \\
10192epigrids-sad & 1672819.53 & 185.02 & 95740 & 23726 & 23 & 1720194.13 & 23.74 & \\
10480goc-api & 2708819.18 & 200.48 & 114967 & 29805 & 21 & 2863484.4 & 38.71 & \\
10480goc-sad & 2287314.69 & 270.38 & 118122 & 28004 & 24 & 2314712.14 & 27.93 & \\
13659pegase & 379084.55 & 841.83 & 176962 & 37297 & 22 & 386108.81 & 1184.15 & \\
13659pegase-api & 9270988.77 & 326.57 & 147479 & 34390 & 19 & 9385711.45 & 44.43 & \\
13659pegase-sad & 8868216.24 & 301.87 & 130682 & 32662 & 19 & 9042198.49 & 42.08 & \\
19402goc-api & 2448812.41 & 440.67 & 213564 & 52388 & 22 & 2583627.35 & 87.33 & \\
19402goc-sad & 1954047.79 & 488.33 & 218291 & 49749 & 25 & 1983807.59 & 64.01 & \\
20758epigrids-api & 3042956.88 & 464.17 & 189436 & 46124 & 25 & 3126508.30 & 61.39 & \\
20758epigrids-sad & 2612551.03 & 379.36 & 180790 & 44624 & 24 & 2638200.23 & 58.11 & \\
24464goc-api & 2560407.12 & 471.14 & 226595 & 57162 & 22 & 2683961.9 & 533.03 & \\
24464goc-sad & 2605128.51 & 506.39 & 222908 & 55242 & 23 & 2653957.66 & 73.87 & \\
ACTIVSg25k & 5993266.85 & 592.39 & 156285 & 43851 & 28 & 6017830.61 & 56.69 & \\
30000goc-api & 1531110.84 & 464.16 & 142385 & 41840 & 24 & 1777930.63 & 134.71 & \\
30000goc-sad & 1130733.51* & 147.74 & 76546 & 76546 & 6 & 1317280.55 & 565.05 & \\
ACTIVSg70k & 16326225.66 & 1065.76 & 350572 & 123431 & 13 & 16439499.83 & 240.55 & \\
78484epigrids-api & 15877674.54 & 1007.99 & 556893 & 240576 & 10 & 16140427.68 & 1079.03 & \\
78484epigrids-sad & 15175077.19 & 1062.55 & 501202 & 313587 & 8 & 15315885.86 & 343.45 & \\
\bottomrule
\end{tabular}
}
\label{table:cuts}
\end{table}
\endgroup

Overall,  our cut management heuristics yield very tight linearly constrained relaxations using a relatively small number of cuts - note that we could potentially add $3|\mathcal{E}|$ cuts per round (for each branch $\{k,m\}$ there are three inequalities~\eqref{rotatedconeinequalities} and \eqref{thermallimit} that might be violated). For instance, case ACTIVSg70k has $88207$ branches and after 10 rounds of cuts we {only keep $123431$ out of the $350572$ linear cuts computed throughout the course of our algorithm. Thus, fewer than $1.5$ linear cuts per branch give us a relaxation with \emph{optimality gap}\footnote{Given a primal bound of a minimization problem, we define the optimality gap of a relaxation of the given problem as $\frac{z_{p}-z_{r}}{z_{p}}$, where $z_{p}$ denotes the objective value of the primal bound and $z_{r}$ denotes the objective value of the relaxation.} at most $0.69 \%$.

We remark that for some instances the objective value of our procedure can be higher than the objective value of the Jabr SOCP since our algorithm is outer-approximating the feasible region of the i2 SOCP (c.f. Proposition~\ref{proposition:i2_stronger_than_jabr}).
 
We also note that 30000goc-sad was the only instance for which Gurobi experienced numerical difficulties while solving the linearly constrained relaxation (indicated by the character $``*"$ next to the objective value). In this case the reported objective value corresponds to the previous iteration -- setting a more aggressive cut management heuristic, for instance decreasing $T_{age}$ from $4$ to $5$, gave us numerically more stable cuts and a better bound.

Finally, we obtained our primal bounds by running Knitro with a \emph{flat-start}, i.e., we provided as initial point voltage magnitudes set to $1$ and all $\theta_{km} = 0$.


\subsubsection{Performance on SOCPs}\label{subsubsection:performance_socps}


In Table~\ref{table:jabrsocps}, we observe that for the cases in which at least two solvers converge, the reported bounds for the Jabr SOCP agree on the first 2 to 3 most significant digits. These differences in bounds reflect how numerically challenging the given instances are. We remind the readers of the parameter choices that we made in order for the solvers to achieve termination -- the solvers otherwise would often fail.

\begingroup
\setlength{\tabcolsep}{4pt}
\begin{table}
\caption{Solvers' Performance on Jabr SOCP}
\centering
\begin{tabular}{ @{} l r r r r r r r @{} }
\toprule
& \multicolumn{3}{c}{Objective} & \multicolumn{3}{c}{Time} & \\
\cmidrule(l{0.5em}r{0.40em}){2-4} \cmidrule(l{0.5em}r{0.40em}){5-7}
\multicolumn{1}{l}{Case} & Gurobi & Knitro & Mosek & Gurobi & Knitro & Mosek & \\
\midrule
9241pegase & - & 309234.16 & - & 82.11 & 34.68 & 31.11 & \\
9241pegase-api & - & 6840612.84 & - & 116.32 & 23.39 & 72.29 & \\
9241pegase-sad & - & 6083747.85 & - & 111.05 & 26.01 & 75.99 & \\
9591goc-api & 1346480.71 & 1348107.89 & 1345869.72 & 38.25 & 23.74 & 36.60 & \\
9591goc-sad & 1055698.54 & 1058606.56 & 1054379.58 & 49.29 & 32.83 & 37.61 & \\
ACTIVSg10k & - & 2468172.93 & 2466666.10 & 40.18 & 21.48 & 26.08 & \\
10000goc-api & - & 2507034.94 & 2498948.00 & 48.63 & 35.19 & 30.13 & \\
10000goc-sad & 1387288.49 & 1388679.63 & 1386041.07 & 23.58 & 26.27 & 23.68 & \\
10192epigrids-api & - & 1849684.14 & 1848873.47 & 75.82 & 42.69 & 29.09 & \\
10192epigrids-sad & - & 1672989.96 & 1672534.72 & 83.85 & 28.33 & 28.63 & \\
10480goc-api & - & 2708973.58 & 2707828.26 & 75.94 & 27.21 & 56.82 & \\
10480goc-sad & - & 2286454.3 & 2285547.23 & 149.93 & 38.17 & 59.48 & \\
13659pegase & 379135.73 & 379144.11 & - & 33.61 & 43.26 & 34.92 & \\
13659pegase-api & - & 9198542.14 & - & 162.21 & 30.64 & 105.11 & \\
13659pegase-sad & 8826902.31 & 8826958.23 & 8787429.86 & 83.75 & 31.84 & 108.74 & \\
19402goc-api & - & 2449020.25 & 2447799.72 & 158.12 & 152.89 & 103.04 & \\
19402goc-sad & - & 1954331.70 & 1952550.06 & 203.56 & 155.89 & 104.88 & \\
20758epigrids-api & - & - & 3040421.02 & 143.99 & TLim & 93.46 & \\
20758epigrids-sad & - & - & 2610196.94 & 98.30 & TLim & 75.88 & \\
24464goc-api & 2548335.96 & - & 2558631.63 & 603.95 & TLim & 129.90 & \\
24464goc-sad & - & - & 2603525.46 & 333.50 & TLim & 128.50 & \\
ACTIVSg25k & 5956787.54 & 5964417.54 & 5955368.56 & 169.66 & 87.14 & 87.18 & \\
30000goc-api & - & 1531256.65 & 1529197.81 & 207.60 & 118.80 & 123.38 & \\
30000goc-sad & - & - & 1130868.71 & 191.22 & TLim & 84.90 & \\
ACTIVSg70k & - & 16221577.73 & 16217263.66 & 553.26 & 320.98 & 232.47 & \\
78484epigrids-api & - & - & - & 756.00 & TLim & 637.48 & \\
78484epigrids-sad & 15180775.21 & - & 15169401.54 & 463.17 & TLim & 601.04 & \\
\bottomrule
\end{tabular}
\label{table:jabrsocps}
\end{table}
\endgroup

As we mentioned at the beginning of this section, the i2 SOCP is numerically even more challenging for the solvers than the Jabr SOCP. Indeed, in Table~\ref{table:i2socps} we can see that the solvers do struggle (see also Table~\ref{table:ws_loads_T4_i2socp+}). We studied in detail some cases where Gurobi AMPL declared optimality --for example case ACTIVSg70k-- while reporting variable bound max scaled violation equal to $8.43$ as well as large primal and dual residuals ($0.0128$ and $3.25$, respectively). Moreover, we noticed inconsistent termination status for cases 10192epigrids-sad, 10480goc-api, 20758epigrids-sad, and 30000goc-sad on Gurobi and Gurobi through AMPL (Gurobi-AMPL) using the same model; Gurobi AMPL declares optimal termination for these instances while Gurobi does not. Cases for which we were able to identify \emph{low quality} solutions or inconsistencies have been denoted with the character ``\dag" next to their reported objective value in Table~\ref{table:i2socps}.


\begingroup
\setlength{\tabcolsep}{4pt} 
\begin{table}
\caption{Solvers' Performance on i2 SOCP}
\centering
\begin{tabular}{ @{} l r r r r r r r @{} }
\toprule
& \multicolumn{3}{c}{Objective} & \multicolumn{3}{c}{Time} & \\
\cmidrule(l{0.5em}r{0.40em}){2-4} \cmidrule(l{0.5em}r{0.40em}){5-7}
\multicolumn{1}{l}{Case} & Gurobi & Knitro & Mosek & Gurobi & Knitro & Mosek & \\
\midrule
10192epigrids-api & 1849683.44 & 1849684.2 & - & 37.44 & 19.4 & 30.74 & \\
10192epigrids-sad & 1672998.72\dag & 1672998.73 & - & 23.36 & 21.48 & 24.53 & \\
10480goc-api & 2709110.52\dag & 2709110.71 & - & 41.44 & 27.78 & 35.28 & \\
10480goc-sad & 2287736.73 & 2287715.33 & - & 41.11 & 28.46 & 28.48 & \\
13659pegase & 379142.67 & - & - & 52.12 & TLim & 36.49 & \\
13659pegase-api & 9287242.7 & 9287244.72 & - & 66.39 & 236.65 & 30.01 & \\
13659pegase-sad & 8878803.69 & - & - & 63.11 & TLim & 30.48 & \\
19402goc-api & 2449100.15\dag & 2449102.05 & - & 79.38 & 55.6 & 54.18 & \\
19402goc-sad & 1954367.11\dag & 1954367.2 & - & 180.97 & 59.46 & 82.55 & \\
20758epigrids-api & - & 3043275.95 & - & 79.03 & 64.17 & 56.02 & \\
20758epigrids-sad & 2612841.71\dag & 2612841.8 & - & 48.32 & 84.07 & 58.87 & \\
24464goc-api & 2560829.65\dag & - & - & 132.34 & TLim & 88.79 & \\
24464goc-sad & 2605532.65\dag & - & - & 74.43 & 916.1 & 65.41 & \\
ACTIVSg25k & 5994727.45 & - & - & 70.61 & TLim & 52.38 & \\
30000goc-api & 1531320.78 & 1531322.2 & - & 96.91 & 593.73 & 71.38 & \\
30000goc-sad & 1132242.88\dag & 1132256.94 & - & 78.0 & 325.11 & 74.61 & \\
ACTIVSg70k & 16333807.38\dag & - & - & 300.98 & TLim & 209.3 & \\
78484epigrids-api & 15882668.49 & 15882668.46 & 15882654.42 & 216.15 & 315.31 & 203.81 & \\
78484epigrids-sad & 15180792.15 & 15180792.0 & 15180763.6 & 250.43 & 376.82 & 222.17 & \\
\bottomrule
\end{tabular}
\label{table:i2socps}
\end{table}
\endgroup

\subsection{Warm-Started Single-Period Instances}\label{subsection:warmstarts}

As described in the introduction, in power engineering practice it is of interest to address problem instances where a limited change in data occurs after solving a previous problem. In the context of this paper, we can therefore assume that we have a \textit{warm-started formulation}, i.e., one where we  leverage previously computed cuts. Here we present this warm-starting feature of our algorithm; we justify its validity and show via numerical experiments its appealing lower bounding capabilities. 

We note that the convex inequalities~\eqref{rotatedconeinequalities}, which we use to develop cuts, do not depend on input data such as loads or operational limits. Any such cut remains valid and can be used if the associated branch remains operational.  This will be our strategy, below.

We created two kinds of perturbed instances: {\bf a)} Instances were the load of each bus was perturbed by a Gaussian $(\mu,\sigma) = (0.01 \cdot P_{d}, 0.01 \cdot P_{d})$, where $P_{d}$ denotes the original load, subject to the newly perturbed load being non negative; and {\bf b)} instances were the transmission line which carries the largest amount of active power in an ACOPF solution is turned off. With regards to cases of type a), we focus on relatively small individual load changes to model re-optimization after a relatively small time change.  In the multi-period setting, below, we consider much larger load changes, to capture hourly models. Cases of type b) do change the structure of the network and in their warm-start we removed any cuts associated with the switched-off branch.   

Tables~\ref{table:ws_loads_1_1} and~\ref{table:ws_onelineoff} summarize our warm-started experiments on perturbed instances from our data set in Table~\ref{table:cuts} and compare our algorithm and solvers' performance on the Jabr SOCP. ``First Round" reports the objective value and running time of the relaxation $M_{0}$ loaded with the cuts computed in Table~\ref{table:cuts}, i.e., our warm-started relaxation. Moreover, under ``Cutting-plane", we report on the objective value (at termination) and the total running time of our cutting-plane procedure (on the warm-started relaxation).  ``Jabr SOCP" and ``Primal bound" report, respectively, on the objective value and running time of the Jabr SOCP for the three solvers, and ACOPF primal solutions.

We stress that after just \textit{one} round, the bound provided by our algorithm is already quite good -- and also point out the comparison between the running time for our first round, and the solvers' running time. 

\paragraph{Loads perturbed -- Gaussian deltas.} 
For most of these cases, our procedure proves very tight lower bounds in less than 25 seconds (``First Round" column). Judging by the time required by Knitro  and by the number of cases in which the solvers converge running the SOCPs, it appears that these instances are, overall, more challenging than their unperturbed counterparts.

Our procedure also stands out in quickly lower bounding the largest cases. For instance, a very sharp bound for case ACTIVSg70k is obtained in 102.25 seconds, taking less than half of the time it takes the fastest SOCP solver to converge. Similar performance is achieved on the largest epigrids cases where our method is 3 to 5 times faster.    

An interesting empirical fact is that our cuts are robust with respect to load perturbations. Indeed, our evidence shows that there is not a considerable improvement from the ``First Round" to the ``Last Round" objectives.
Thus, the previously pre-computed cuts loaded to $M_{0}$ in the first iteration are, already, accurately outer-approximating the SOC relaxations.

Our linearly constrained relaxations are able to prove infeasibility for the modified case 9241pegase-api in 23.10 seconds while none of the three solvers were able to provide a certificate of infeasibility for the Jabr SOCP. Knitro required 1845.42 seconds to declare convergence to a locally infeasible solution. Similar results are obtained for case 24464ogc-api.

The only case were our method fails to provide a valid lower bound is case 30000goc-sad -- our minimization oracle reports numerical trouble and fails to provide a solution to our warm-started relaxation. This is not surprising since difficult numerical behavior was noticed when computing cuts for this case.

\paragraph{Transmission line with largest flow switched off.}
Overall, our method achieves a similar performance on this set of perturbed instances as in a); sharp lower bounds are obtained in about 25 seconds for most of the cases.

For this data set, our method and all of the SOCP solvers are able to prove infeasibility relatively quickly. On the other hand, our method proves a lower bound for ACTIVSg70k relatively quickly in the first round, but fails to converge in the next round due to numerical trouble caused by the newly added cuts.

As when perturbing loads, our warm-started formulation achieves a good performance on the largest epigrid cases --  bounds are sharp with respect to the SOC relaxations and it is at least 3x faster.

The only case were our method fails to provide a lower bound is case 30000goc-sad -- our minimization oracle reports numerical trouble and fails to provide a solution to our warm-started relaxation.

\begingroup
\setlength{\tabcolsep}{4pt}
\begin{sidewaystable}
\footnotesize
\centering 
\vspace*{12cm}
\caption{Warm-Started Relaxations, Loads perturbed Gaussian $(\mu,\sigma) = (0.01 \cdot P_{d}, 0.01 \cdot P_{d})$}
\begin{adjustbox}{valign=m,center}
\resizebox{\textwidth}{!}{
\begin{tabular}{ @{} l r r r r r r r r r r r r r @{} }
\toprule
& \multicolumn{4}{c}{Cutting-Plane} & \multicolumn{6}{c}{Jabr SOCP} &  \\
\cmidrule(l{0.5em}r{0.40em}){2-5} \cmidrule(l{0.5em}r{0.40em}){6-11}
& \multicolumn{2}{c}{First Round} & & & \multicolumn{3}{c}{Objective} & \multicolumn{3}{c}{Time} & \multicolumn{2}{c}{Primal bound} & \\
\cmidrule(l{0.5em}r{0.40em}){2-3} \cmidrule(l{0.5em}r{0.40em}){6-8} \cmidrule(l{0.5em}r{0.40em}){9-11} \cmidrule(l{0.5em}r{0.40em}){12-13} 
\multicolumn{1}{l}{Case} & Objective & Time & Objective & Time & Gurobi & Knitro & Mosek & Gurobi & Knitro & Mosek & Objective & Time &  \\
\midrule
9241pegase    &     309288.32     &     13.78     &     309299.97     &    160.28    &   -   &    309302.67    &  -  &   73.12   &    32.21    &  36.04  & 315979.53 & 101.48 &      \\
9241pegase-api    &    INF      &     23.10     &     INF     &   23.10     &   -   &    -    &  -  &   134.53   &    TLim    &  72.96  & LOC INF & 1845.92 &  \\
9241pegase-sad    &     6153913.91     &    16.18      &    6154117.59      &    136.78    &   -   &    6096743.03    &  -  &   97.51   &    26.07    &  83.43  & 6333763.92 & 43.71 &  \\
9591goc-api    &     1343642.47     &     11.06     &     1343670.62    &    56.36    &   1343767.43   &    1345384.57    &  1343190.29  &   39.36   &    25.36    &  35.30  & 1571582.59 & 54.16 &  \\
9591goc-sad    &     1058124.48     &    12.62      &     1058157.44     &    65.37    &   1058337.76   &    1061275.83    &  1057323.31  &   51.85   &    34.04    &  37.52  & 1178895.53 & 29.53 &  \\
ACTIVSg10k    &    2475041.43      &      9.52    &     2475078.69     &    50.51     &   -   &    2466383.20    &  -  &   42.31   &    21.75    &  29.33  & 2484093.15 & 57.24 &  \\
10000goc-api    &   2502049.28       &    8.51      &     2502098.01    &    36.48     &   2501946.30   &    2507074.78    &  2499373.75  &   31.91   &    43.44    &  32.33  & LOC INF & 1677.21 &  \\
10000goc-sad    &     1388833.86     &    8.70      &     1388859.09     &    44.50     &   1388824.91   &    1390230.41    &  1387588.17  &   25.96   &    29.31    &  23.67  & 1493481.44 & 93.72 &  \\
10192epigrids-api    &     1848085.36     &    10.27      &    1848133.48      &    45.84    &   -   &    1848285.26    &  1847120.93  &   65.38   &    41.17    &  25.99  & LOC INF & 1458.35 &  \\
10192epigrids-sad    &     1672358.89     &    10.33      &     1672398.61     &    53.37    &   -   &    1672533.02    &  1671364.67  &   73.64   &    28.61    &  35.66  & 1717429.36 & 23.89 &  \\
10480goc-api    &      2704157.29    &    12.43      &     2704252.95     &    58.45    &   -   &    2704373.73    &  2703432.85  &   197.17   &    27.57    &  55.92  & 2868495.28 & 36.89 &  \\
10480goc-sad    &    2294908.37      &    12.81      &    2294990.69      &    70.93    &   -   &    2294080.35    &  2292830.56  &   185.22   &    35.90    &  58.31  & 2322198.81 & 27.34 &  \\
13659pegase    &     379742.62     &    60.74      &     379794.51     &    426.88    &   379799.37   &    379804.43    &  -  &   34.21   &    43.17    &  32.75  & 386765.25 & 370.23 &  \\
13659pegase-api    &     9253539.07     &    21.25      &     9253773.43     &    109.20    &   9181205.93   &    9181269.20    &  -  &   97.11   &    30.41    &  118.31  & 9368277.57 & 62.20 &  \\
13659pegase-sad    &     8865733.59     &    21.28      &     8865892.49     &    113.04    &   8824442.20   &    8824486.03    &  -  &   86.49   &    33.19    &  102.59  & 9039904.52 & 40.02 &  \\
19402goc-api    &     2452185.69     &    23.55      &     2452270.83     &    120.10    &   -   &    2452448.33    &  2451708.50  &   146.87   &    120.39    &  103.32  & LOC INF & 4440.99 &  \\
19402goc-sad    &    1956255.19      &    23.28      &     1956313.91     &    113.89    &   -   &    1956570.60    &  1955018.07  &   231.90   &    172.82    &  102.19  & 1986936.95 & 66.02 &  \\
20758epigrids-api    &     3043006.76     &     22.34     &    3043076.56      &    104.06    &   -   &    -    &  3032919.24  &   134.60   &    TLim    &  78.32  & LOC INF & 12425.89 &  \\
20758epigrids-sad    &     2610197.53     &     20.46     &     2610261.88     &    93.09    &   -   &    -    &  2608090.26  &   143.69   &    TLim    &  72.19  & 2635892.81 & 49.25 &  \\
24464goc-api    &     2561680.14     &   26.28      &      INF    &    50.38    &   -   &    LOC INF    &  -  &   223.07   &    573.37    &  118.6  & - & 19444.54 &  \\
24464goc-sad    &     2606391.76     &     26.78     &    2606473.78      &     133.54   &   -   &    -    &  2604708.86  &   423.12   &    TLim    &  128.84  & 2655942.01 & 72.48 &  \\
ACTIVSg25k    &     5988886.18     &    28.24      &     5989016.75     &    198.58    &   5952404.50   &    5960068.30    &  5949381.04  &   138.01   &    73.75    &  109.39  & 6013477.05 & 57.87 &  \\
30000goc-api    &    1527412.96      &   25.35       &   1527487.45   &    151.75    &   -   &    1528338.73    &  1525625.64  &   243.61   &    369.83    &  119.92  & LOC INF & 3407.47 &  \\
30000goc-sad    &    -      &     46.33     &     -     &     46.33     &   -   &    -    &  1132715.53  &   257.94   &    TLim    &  75.20  & 1318389.55 & 620.27 &  \\
ACTIVSg70k    &    16316572.42      &    102.25      &    16317886.35       &    536.51    &   -   &    16210682.53    &  16206290.43  &   498.80   &    309.56    &  229.07  & 16428367.50 & 243.84 &  \\
78484epigrids-api    &     15862318.24     &     115.76     &     15865624.98     &     883.93     &   -   &    -    &  15859950.52  &   757.64   &    TLim    &  642.24  & - & 8113.53 &  \\
78484epigrids-sad    &     15176866.00     &     151.77     &     15180592.27     &    1118.02    &   15182602.75   &    -    &  15174716.43  &   420.56   &    TLim    &  589.46  & 15316872.94 & 353.13 &  \\
\bottomrule
\end{tabular}
}
\end{adjustbox}
\label{table:ws_loads_1_1}
\end{sidewaystable}
\endgroup

\begingroup
\setlength{\tabcolsep}{4pt}
\begin{sidewaystable}
\centering
\vspace*{12cm}
\begin{adjustbox}{valign=m,center}
\rotatebox{180}{
\begin{minipage}{\textwidth}
\caption{Warm-Started Relaxations, Transmission Line with largest flow turned off}
\footnotesize
\resizebox{\textwidth}{!}{
\begin{tabular}{ @{} l r r r r r r r r r r r r r @{} }
\toprule
& \multicolumn{4}{c}{Cutting-Plane} & \multicolumn{6}{c}{Jabr SOCP} &  \\
\cmidrule(l{0.5em}r{0.40em}){2-5} \cmidrule(l{0.5em}r{0.40em}){6-11}
& \multicolumn{2}{c}{First Round} & & & \multicolumn{3}{c}{Objective} & \multicolumn{3}{c}{Time} & \multicolumn{2}{c}{Primal bound} & \\
\cmidrule(l{0.5em}r{0.40em}){2-3} \cmidrule(l{0.5em}r{0.40em}){6-8} \cmidrule(l{0.5em}r{0.40em}){9-11} \cmidrule(l{0.5em}r{0.40em}){12-13} 
\multicolumn{1}{l}{Case} & Objective & Time & Objective & Time & Gurobi & Knitro & Mosek & Gurobi & Knitro & Mosek & Objective & Time &  \\
\midrule
9241pegase  &  INF  &  10.86  &  INF  &  10.86  &  INF  &  INF  &  INF  &  8.36  & 7.55   &  10.22  & INF & 8.37 &  \\
9241pegase-api  &  INF  &  7.55  &  INF  &  7.55  &  INF  &  INF  &  INF  &  7.92  &  8.02  &  10.22  & INF & 8.23 &    \\
9241pegase-sad  &  INF  &  7.33  &  INF  &  7.33  &  INF  &  INF  &  INF  &  8.14  & 8.10 &  10.31  & INF & 8.46 &  \\
9591goc-api  &  1346470.95  &  10.42  &  1346859.06  &  60.76  &  1346969.75  &  1348591.44  &  1346437.99  &  39.30  &  17.98  &  36.89  & 1395829.51 & 28.08 &  \\
9591goc-sad  &  1055823.53  &  11.51  &  1056267.57  & 101.64   &  1056447.48  & 1059382.10 &  1055501.31  &  45.09  & 35.41   & 37.18 & 1199276.44 & 29.90 &  \\
ACTIVSg10k  & 2477043.05   &  9.94  &  2477537.79  &  75.85  &  -  &  2468821.96  &  2466981.35  &  44.81  &  21.60  &  17.52  & LOC INF & 7092.45 &  \\
10000goc-api  &  2506671.15  &  8.06  & 2509971.69 &  46.10  & 2509846.00 &  2514991.16  &  2506236.75  &  31.03  &  33.95  &  32.15  & 2692320.35 & 23.28 &  \\
10000goc-sad  &  1387382.65  &  8.76  &  1387515.89  &  66.14  & 1387480.33   &  1388870.68  &  1386283.75  &  26.53  &  34.24  &  24.14  & 1506187.88 & 108.19 &  \\
10192epigrids-api  &  1849901.82  &  9.47  &  1850621.81  &  68.73  &  -  &  1850788.76  &  1849821.44  &  69.81  &  38.01  & 25.30  & 2021493.05 & 117.18 &  \\
10192epigrids-sad  &  1673575.50  &  11.08  &  1674274.99  &  74.49  &  -  &  1674417.21  &  1673564.57  &  69.91  &  43.91  &  28.54  & 1734014.50 & 24.11 &  \\
10480goc-api  &  2710040.46  &  11.33  &  2711100.23  &  73.85  &  -  &  2711224.27  &  2710520.15  &  95.40  &  27.99  &  56.58  & 2862699.50 & 225.06 &  \\
10480goc-sad  &  2288069.64  &  13.47  & 2288969.47 & 98.88 & - &  2288069.23  &  2286864.08  &  106.54  &  37.65  &  59.43  & 2318279.76 & 26.13 &  \\
13659pegase  &  379102.13  &  53.29  &  379163.58  & 199.96 &  379177.99  &  379182.14  &  -  &  33.18  &  345.83  & 31.90 & 386126.93 & 394.93 &  \\
13659pegase-api  &  INF  &  10.81  &  INF  & 10.81 &  INF  &  INF  &  INF  &  10.94  &  8.90  &  13.79  & INF & 11.55 &  \\
13659pegase-sad  &  INF  & 10.63 &  INF  & 10.63 &  INF  &  INF  & INF &  11.02  &  10.80  &  13.76  & INF & 11.73 &  \\
19402goc-api  &  2450110.09  &  23.93  &  2451621.60  &  171.31 &  -  &  2451793.39  &  2450488.01  &  154.42  &  132.03  &  104.58  & 2587915.50 & 403.20 &  \\
19402goc-sad  & 1954365.39 & 23.70 &  1954881.06  &  191.52  &  -  &  1955116.35  &  1953676.05  &  258.93  &  156.10  &  102.63  & 1985954.83 & 63.38 &  \\
20758epigrids-api  &  3043482.21  &  20.44  &  3044690.46  & 133.92  & -  &  -  &  3041974.74  &  112.78  &  TLim  & 96.19 & 3132571.31 & 52.82 &  \\
20758epigrids-sad  & 2612646.70 & 20.62 & 2612786.37 & 115.89 &  -  &  -  &  2610315.41  &  169.67  &  TLim  &  73.02  & 2638560.64 & 47.87 &  \\
24464goc-api  & 2560669.11 &  25.94  & 2561110.03 & 161.80 &  2550118.22  &  - &  2559240.55  &  440.81  &  TLim  &  118.05  & 2684708.93 & 1663.63 &  \\
24464goc-sad  & 2605179.75 & 26.98 & 2605369.03 & 166.81 &  -  &   2605474.23 &  2603609.34  &  564.27  &  74.69  &  124.19  & 2654344.45 & 76.39 &  \\
ACTIVSg25k  &  6045885.88  &  27.52  &  6048122.86  & 238.67 &  6009656.52  &  6018875.03  &  6009500.57  &  144.28  &  65.42  & 81.41 & LOC INF & 1634.38 &  \\
30000goc-api  &  1531110.55  &  25.02  & 1531159.95 & 130.30 &  -  &  1532013.41  &  1529195.12  &  195.74  &  135.71  &  119.77  & LOC INF & 3203.26 &  \\
30000goc-sad  &  - &  45.60  &  -  &  45.60  &  -  &  -  &  1130917.78  &  218.12  &  TLim  &  74.27  & 1324622.71 & 186.43 &  \\
ACTIVSg70k  &  16426522.74  &  98.51  &  16426522.74*  &  98.51  &  -  &  -  &  -  &  150.77  &  TLim  &  129.60  & LOC INF & 3160.81 &  \\
78484epigrids-api  &  15888353.48  &  104.32  &  15892229.11  & 916.88 &  15894055.55  &  -  &  15880422.78  &  322.22  &  TLim  &  625.74  & 16169740.92 & 2328.79 &  \\
78484epigrids-sad  &  15179882.22  &  149.65  &  15185980.69  &  1151.33  &  15188085.99  &  -  &  15182701.65  &  437.59  & TLim   &  594.84  & 15330674.69 & 272.41 &  \\
\bottomrule
\end{tabular}
}
\label{table:ws_onelineoff}
\end{minipage}
}
\end{adjustbox}
\end{sidewaystable}
\endgroup

\subsection{Multi-Period Results}\label{subsection:experiments_multi}

\paragraph{Data.}
At present, there is a dearth of publicly-available, realistic multi-period AC load and generation data.  One partial source is the set of cases developed for the GO3 Competition \cite{go3}\footnote{The GO3 Competition does not provide explicit load data, since a welfare maximization problem needs to be solved to determine the active and reactive power consumption. We purposely avoided this dependence on the solution of another challenging optimization problem.}. We are also aware that creating \emph{good} ACOPF data is a non-trivial task~\cite{li+etal21}, ~\cite{birchfield+etal17}, but given that there were no publicly available multi-period load time-series to match the standard AC cases we developed our own data sets\footnote{These data sets can be downloaded from www.github.com/matias-vm.}. 

We created and worked with three data sets: (1) $T=4$ with increasing\footnote{Given that the PGLIB library~\cite{babaeinejadsarookolaee+etal21} contains congested instances, we applied decreasing (on expectation) active power loads perturbations to the PGLIB instances.} (on expectation) active power loads, to be precise, at each time-period $t \in \{1,2,3\}$, $P_{k,t}^{d}$ is set to $(1 + 0.01 \cdot t + u_{k,t}) \cdot P_{k}^{d} $, where $u_{k,t}$ is a uniformly distributed random variable with range $[0,0.025]$; $T=12$ where a half-a-day active power load curved is simulated, see Figure~\ref{fig:loadcurve_T12}; and (3) $T=24$ where a 24 hr active power load curve is simulated, see Figure~\ref{fig:loadcurve_T24}. Moreover, we assumed 50\% generator ramp up $r^{u}$ and down $r^{d}$ rates for consecutive time-periods (c.f.~\ref{multiperiod:formulation}) in all of the three settings. 

\begin{figure}[H]
  \centering
  \resizebox{\textwidth}{!}{
  \begin{minipage}[b]{0.48\textwidth}
    \includegraphics[width=\textwidth]{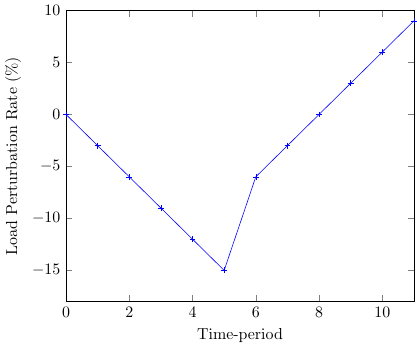}
    \caption{12hr load curve.}\label{fig:loadcurve_T12}
  \end{minipage}
  \hfill
  \begin{minipage}[b]{0.48\textwidth}
    \includegraphics[width=\textwidth]{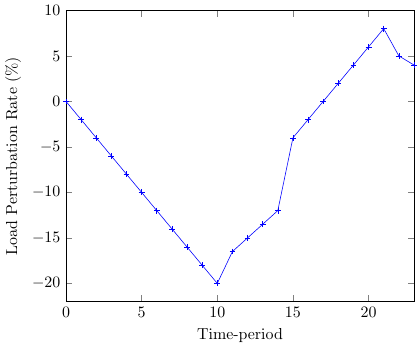}
    \caption{24hr load curve.}\label{fig:loadcurve_T24}
  \end{minipage}
  }
\end{figure}

\paragraph{Upper bound heuristics.}
In order to evaluate how tight a given lower bound is, an upper bound is needed. We relied on two simple heuristics to compute upper bounds. The first heuristic, which we term H1, consists in providing to Knitro the complete multi-period formulation. The parameter setup for H1 is: absolute feasibility tolerance equal to $10^{-6}$, while we set relative optimality tolerance to $10^{-3}$; simple variables bounds are enforced throughout the optimization; and $T_{ftol} = 3$; and time limit of $2400$ seconds for $T=4$, and $7,200$ seconds for $T=12,24$. If H1 fails, then we proceed with heuristic H2 which solves one single-period ACOPF instance at a time, and upon convergence, fixes the upper and lower limits of active power generation of the next instance as a function of generator outputs in the current period. To be precise, for every generator $k \in \mathcal{G}$ and any time-period $t \in \{0,\dots,T-2\}$, and provided that the nonlinear solver converges for the single-period instance at time $t$, then the output of generator $k$ for the time-period $t+1$ is bounded by $P^{min}_{k,t+1} \leq P^{g}_{k,t+1} \leq P^{max}_{k,t+1}$ where $P^{min}_{k,t+1} := (1 - r^{d}_{k,t+1}) P^{g}_{k,t}$ and $P^{max}_{k,t+1} := (1 + r^{u}_{k,t+1}) P^{g}_{k,t}$. If the solver fails to find a solution for some time-period, then we declare failure of H2. Overall, H2 succeeded in most of the instances for which we were able to obtain a primal bound. The parameter setup for H2 was the same as for H1 except that we set the time limit for single-period ACOPFs to $1,200$ seconds.

\paragraph{SOCPs: Solver parameters for multi-period instances.} With respect to our SOCP experiments, we set a time limit of $2,400$ seconds for $T=4$ instances, $4,800$ seconds for $T=12$, while we set $8,000$ seconds for $T=24$ instances. In what follows we describe the parameter specifications for each solver for the SOCP runs.
\begin{itemize}
    \item We use Gurobi's default homogeneous self-dual embedding interior-point algorithm (barrier method without \emph{Crossover}, and \emph{Bar Homogeneous} set to $1)$, and we set the parameter \emph{Numeric Focus} equal to $1$. Barrier convergence tolerance and absolute feasibility and optimality tolerances were set to $10^{-6}$. Gurobi was allowed 20 threads.
    \item We use Knitro's default Interior-Point/Barrier Direct Algorithm, with absolute feasibility and optimality tolerances equal to $10^{-6}$. We used the Intel MKL PARDISO sparse symmetric indefinite solver, and the Intel Math Kernel Library (MKL) functions for Basic Linear Algebra Subroutines (BLAS), i.e., for basic vector and matrix computations. Moreover, we gave Knitro $20$ threads to use for parallel computing features ($10$ threads were given to the linear solver). When solving the SOCPs, we explicitly instructed Knitro to handle the problem as convex, and to enforce simple variable bounds throughout the optimization.
    \item We use Mosek's default homogeneous and self-dual interior-point algorithm for conic optimization. We set the relative termination tolerance, as well as primal and dual absolute feasibility tolerances to $10^{-6}$. We assigned $20$ threads to Mosek.
\end{itemize}

We recall the reader that we use the i2+($\rho$) relaxation, with $\rho = 100$, as target set for our cutting-plane multi-period experiments. For our cutting-plane algorithm, when $T = 4$ a $1,200$ seconds time limit was enforced prior to \emph{starting} a new round of cutting; the time limit was $3,600$ seconds for $T=12$, and $6,000$ seconds for $T=24$.

\subsubsection{Not Warm-Started Cutting-Plane Algorithm: T = 4}\label{subsubsection:cutplane_scratch}

Table~\ref{table:cp_scratch_T4} summarizes our not warm-started cutting-plane experiments on multi-period instances with $T=4$. Columns ``\# Vars", ``\# Cons", and ``FTime" denote the number of variables, number of constraints and formulation time of the base model $M_{0}^{T}$. Under  ``First Iteration" we report on the objective value, running time, and maximum dual infeasibility error (c.f.~\ref{section:accuracy}) of the relaxation $M^{T}_{0}$, i.e., without any cuts. On the other hand, ``Last Iteration" presents the objective value, running time, total number of cuts held by the model $M_{last}^{T}$ on the last iteration of the algorithm, and maximum dual infeasibility error.  ``Total" presents the total running time of our algorithm and the total number of rounds of cuts, and ``ACOPF" describes the number of variables, number of constraints, and the objective value and running time of each ACOPF primal solution.

Overall, we can see that lower bounds are fairly tight; even our base model $M_{0}^{T}$ without any cuts, whose objective value is shown under ``First Iteration", attains quick and sharp lower bounds for these multi-period instances. For instance, an optimality gap of $6.91\%$ is achieved by $M_{0}^{T}$ for case ACTIVSg70k, and, after three rounds of cuts, nearly 800,000 cuts drive the gap down to $1.75\%$.

Our cut management heuristics prove to be empirically successful: dual infeasibility reported by Gurobi is either maintained or slightly worsened (compare ``DInfs" of ``First Iteration" versus that of ``Last Iteration") after adding a significant number of cuts. This means that our cuts are keeping our linearly constrained relaxation numerically stable and that our reported lower bounds are reasonably accurate (c.f.~\ref{section:accuracy}). Moreover, we observe that instances which require a significant number of rounds of cuts, say 10 or more, under the given time limit, notoriously benefit from our cut clean-up heuristic. We see that a small number of cuts (roughly 10\% of the total number of constraints) is good enough to outer-approximate the rotated-cones and attain sharp ACOPF lower bounds (as in the single-period case, c.f.~\ref{subsection:cutcomputations}).

\begingroup
\setlength{\tabcolsep}{4pt}
\begin{sidewaystable}
\centering
\vspace*{12cm}
\begin{adjustbox}{valign=m,center}
\rotatebox{180}{
\begin{minipage}{\textwidth}
\caption{Cutting-Plane Not Warm-Started, $T=4$ and loads perturbed $\text{Unif}(0,0.025) + 0.01 \cdot t$ for $t \in \{1,2,3\}$, ramping rates $50\%$}
\footnotesize
\resizebox{\textwidth}{!}{
\begin{tabular}{ @{} l r r r r r r r r r r r r | r r r r r @{} }
\toprule
& \multicolumn{12}{c}{Cutting-Plane} & \multicolumn{4}{c}{ACOPF} \\
\cmidrule(l{0.5em}r{0.40em}){2-13} \cmidrule(l{0.5em}r{0.40em}){14-17}
& & & & \multicolumn{3}{c}{First Iteration} & \multicolumn{4}{c}{Last Iteration} & \multicolumn{2}{c}{Total} & & & \multicolumn{2}{c}{PBound}  \\
\cmidrule(l{0.5em}r{0.40em}){5-7} \cmidrule(l{0.5em}r{0.40em}){8-11} \cmidrule(l{0.5em}r{0.40em}){12-13} \cmidrule(l{0.5em}r{0.40em}){16-17}
\multicolumn{1}{l}{Case} & \#Vars & \#Cons & FTime & Obj & Time & DInfs & Obj & Time & \#Cuts & DInfs & Time & Rnds & \#Vars & \#Cons & Obj & Time &  \\
\midrule
1354pegase & 66897 & 72564 & 1.35 & 299460.87 & 2.33 & 5.80e-10 & 303439.90 & 2.31 & 11043 & 3.32e-09 & 37.32 & 12 & 53772 & 70220 & 303727.82 & 7.60 &\\
2869pegase & 150431 & 161568 & 3.31 & 543415.33 & 6.16 & 1.23e-07 & 549381.91 & 7.83 & 26938 & 6.00e-10 & 154.55 & 14 & 120712 & 158388 & 549988.85 & 92.99 &\\
3375wp & 148217 & 159692 & 3.18 & 29810390.36 & 6.79 & 5.13e-08 & 30318794.35 & 10.92 & 16457 & 1.67e-07 & 171.70 & 16 & 117364 & 151844 & 30410712.08 & 63.05 &\\
6468rte & 319517 & 323396 & 6.57 & 349845.41 & 13.77 & 1.29e-04 & 355995.14 & 11.83 & 46902 & 2.72e-06 & 354.29 & 15 & 247296 & 321888 & 356969.89 & 2596.76 &\\
9241pegase & 524028 & 538308 & 10.93 & 1250007.83 & 27.06 & 1.92e-07 & 1269609.76 & 258.55 & 318602 & 5.63e-05 & 1403.50 & 7 & 412248 & 543530 & 1297205.96 & 172.18 &\\
ACTIVSg10k & 456684 & 490360 & 9.98 & 10035162.94 & 24.82 & 7.07e-09 & 10222977.88 & 27.19 & 52082 & 2.10e-08 & 399.43 & 13 & 363940 & 470558 & 10265810.55 & 228.28 &\\
10000goc-api & 465676 & 495636 & 10.04 & 8916264.44 & 24.76 & 1.50e-09 & 9495292.75 & 20.41 & 61212 & 5.28e-07 & 492.16 & 19 & 368928 & 478650 & 10019708.68 & 1049.69 & \\
10000goc-sad & 465676 & 495636 & 10.03 & 5373568.66 & 28.48 & 8.39e-07 & 5456260.77 & 29.84 & 61388 & 1.12e-06 & 403.37 & 11 & 368928 & 478650 & 5866821.40 & 241.57 & \\
13659pegase & 723245 & 735740 & 15.46 & 1528552.09 & 31.02 & 1.29e-08 & 1551553.61 & 422.84 & 469991 & 9.30e-07 & 1380.08 & 6 & 567716 & 739636 & 1584844.81 & 3083.66 &\\
13659pegase-api & 723245 & 735740 & 15.71 & 34787242.54 & 39.17 & 2.84e-05 & 35703095.45 & 207.32 & 316521 & 2.24e-08 & 1207.94 & 9 & 567716 & 739636 & 36163923.41 & 1288.60 &\\
13659pegase-sad & 723245 & 735740 & 15.7 & 33913060.97 & 40.32 & 1.46e-08 & 34547415.27 & 97.98 & 243868 & 9.98e-08 & 1218.18 & 10 & 567716 & 739636 & 35199898.78 & 545.24 &\\
20758epigrids-api & 1106271 & 1127168 & 23.52 & 11540550.41 & 175.25 & 3.38e-08 & 11834277.30 & 373.71 & 453981 & 2.33e-06 & 1294.19 & 5 & 859924 & 1131168 & 12088587.29 & 1018.17 &\\
20758epigrids-sad & 1106271 & 1127168 & 24.30 & 10060842.01 & 86.20 & 4.97e-08 & 10219196.77 & 306.29 & 410645 & 1.22e-07 & 1294.96 & 7 & 859924 & 1131168 & 10317105.84 & 709.85 &\\
ACTIVSg25k & 1145783 & 1212392 & 24.89 & 24097056.99 & 80.20 & 1.93e-08 & 24710933.54 & 120.53 & 239603 & 9.12e-08 & 1228.25 & 8 & 902588 & 1170088 & 24835784.33 & 451.45 &\\
30000goc-api & 1253699 & 1366264 & 27.91 & 4916231.92 & 73.67 & 5.00e-07 & 5705500.71 & 82.73 & 112230 & 1.31e-06 & 1238.19 & 12 & 990172 & 1280368 & 6563873.25 & 1669.13 &\\
30000goc-sad & 1253699 & 1366264 & 27.62 & 4151693.26 & 87.16 & 7.53e-06 & 4340911.74 & 84.31 & 104180 & 7.58e-07 & 1213.47 & 11 & 990172 & 1280368 & - & 3666.85 & \\
ACTIVSg70k & 3115935 & 3316972 & 68.93 & 65302843.92 & 269.38 & 1.85e-08 & 68917427.01 & 389.02 & 799462 & 2.03e-08 & 1268.46 & 3 & 2448820 & 3175256 & 70148300.24 & 2034.37 &\\
78484epigrids-api & 4179320 & 4225032 & 89.55 & 60235793.42 & 347.62 & 1.49e-04 & 61087514.34 & 677.62 & 1496304 & 1.51e-04 & 1816.21 & 3 & 3230648 & 4252514 & 62011780.96 & 4831.70 & \\
78484epigrids-sad & 4179320 & 4225032 & 89.43 & 57816998.73 & 387.48 & 1.23e-04 & 58689459.71 & 571.38 & 1081287 & 1.18e-04 & 1361.42 & 2 & 3230648 & 4252514 & - & - &\\
\bottomrule
\end{tabular}
}
\label{table:cp_scratch_T4}
\end{minipage}
}
\end{adjustbox}
\end{sidewaystable}
\endgroup

\subsubsection{Warm-Started Cutting-Plane Algorithm: T = 4, 12, 24}\label{subsubsection:cutplane_ws}

We use a simple methodology to warm-start our Algorithm~\ref{subsubsection:basicalgorithm_multi}: we propagate cuts for the original instances to every time-period. To be precise, we create a pool that includes, for each branch $\{k,m\} \in \mathcal{E}$, all the associated single-period cuts -- which are assumed to have been previously computed (c.f.~\ref{subsection:cutcomputations}).  The cuts in the pool are propagated to every time-period $t \in \{0, \dots, T-1\}$ and are added to the base model $M_{0}^{T}$. 

In Tables~\ref{table:ws_loads_T4_jabrsocp} and~\ref{table:ws_loads_T4_i2socp+} we present the results of our warm-started experiments with $T=4$. In Table~\ref{table:ws_loads_T4_jabrsocp} we compare our lower bounds (obtained by warm-starting our base model with precomputed cuts) against the lower bounds obtained by solving the Jabr SOCPs with each of the three commercial solvers. In Tables~\ref{table:ws_loads_T4_i2socp+},~\ref{table:ws_loads_T12_i2socp+}, and~\ref{table:ws_loads_T24_i2socp+} we compare our bounds against the performance of the solvers on the i2+($\rho$) relaxation plus Jabr inequalities for every branch whose i2 inequality is bad; this relaxation is numerically better behaved than the i2  SOCP (see discussion in~\ref{subsection:newrelaxation}). 

In both tables the columns ``\# Vars", ``\# Cons", and ``FTime" denote the number of variables, number of constraints and formulation time of the base model $M_{0}^{T}$ (before adding the previously computed cuts). `First Iteration" reports on the objective value, running time, the number of cuts with which the instance was warm-started, and the maximum dual infeasibility error (c.f.~\ref{section:accuracy}) of the warm-started relaxation $M^{0}_{T}$. Multi-column ``Last Iteration" presents the objective value, running time, total number of cuts held by the model $M_{last}^{T}$ on the last iteration of the algorithm, and maximum dual infeasibility error of model $M_{last}^{T}$. ``Total" presents the total running time of our algorithm and the total number of rounds of cuts. Under the header ``Jabr SOCP", columns ``\#Vars" and ``\#Cons" denote the number of variables, number of constraints of the Jabr SOCP (which are the same for the three solvers given that we used a common AMPL model file), and ``Obj" and ``Time" show the objective value and running time of the Jabr SOCP for each solver. Finally, under column ``ACOPF" we report objective value of an ACOPF primal solution.

As in the single-period experiments (c.f. tables~\ref{table:ws_loads_1_1} and~\ref{table:ws_onelineoff}), ``First Iteration" shows that our warm-started relaxation, i.e., model $M_{0}^{T}$ loaded with the propagated precomputed cuts, attains sharp bounds fairly quickly. A notable example, in which at least 2 solvers converge in Table~\ref{table:ws_loads_T4_jabrsocp}, is case ACTIVSg25k for which Gurobi solves the convex QP in 94 seconds resulting in $2.1\%$ optimality gap. Moreover, after 9 rounds of cuts our algorithm is able to drive the down the optimality gap to $0.5\%$ at very high accuracy (c.f. Figure~\ref{fig:improvement}).

 Dual infeasibility in~\eqref{table:ws_loads_T4_jabrsocp}, as in the not warm-started experiments, is not significantly worsened by our cuts, i.e., they do not seem to ill-condition our instances. On the contrary, we see that for some instances, our cut management heuristics are able to improve on this metric. In Figure~\ref{fig:cutmanagement} we illustrates the evolution of the number of outer-envelope cuts with respect to iterations of our cutting-plane algorithm for two $T=4$ warm-started cases. For instance, ACTIVSg10k was warm-started with 28,184 cuts, and peaked 111,323 cuts at iteration $6$ which corresponds to the chosen threshold $T_{age}$ (c.f.~\ref{subsection:basicalgorithm}). From the $7$th iteration onwards, there is a steady cut-cleansing ending up with $51,188$ cuts. As expected, both instances end up with more cuts than were present when warm-started. 

With respect to the solvers' performance on the SOCPs, we see that all of the solvers struggle. In Table~\ref{table:ws_loads_T4_jabrsocp} we see differences up to the second digit in the objective values of instances on which at least 2 solvers converge, which corroborates that these are numerically challenging instances. Table~\ref{table:ws_loads_T4_i2socp+} exhibits even more striking results; here we ran the solvers on the i2+$(\rho)$ relaxation (the relaxation that we are outer-approximating with our cutting-plane algorithm). Recall that this relaxation is numerically better behaved than the i2 SOCP, hence even worse performance is expected on the latter relaxation (c.f.~\ref{subsubsection:performance_socps}).

Next, tables~\ref{table:ws_loads_T12_i2socp+} and~\ref{table:ws_loads_T24_i2socp+} show computational results for our $T=12$ and $T=24$ instances. We remark that both tables follow the same column structure as Table~\ref{table:ws_loads_T4_jabrsocp}. We see than Gurobi is able to robustly handle our very large LPs and convex QPs, such as case ACTIVSg70k, with almost 19 million variables and 20 million constraints. For this instance, we load 1.5 million precomputed cuts, and after two rounds of cuts we end up with 3.35 million cuts, and the solver outputs a lower bound with very small dual infeasibilities. 

We were only able to compute an AC primal bound for case 1354pegase $T=12$; our heuristics failed in all the other instances with $T=12$ and $T=24$.
\color{black}

\begin{figure}[H]
  \centering
  \begin{minipage}[b]{0.48\textwidth}
    \includegraphics[width=0.9\textwidth]{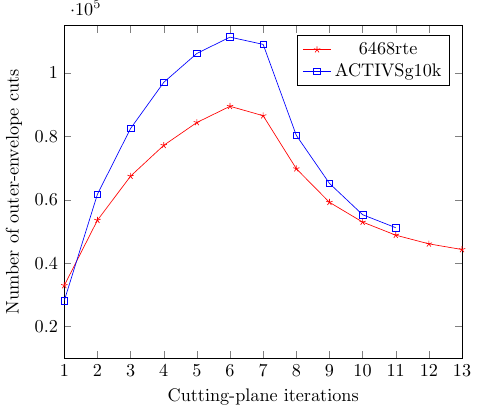}
    \caption{Cut management for $T=4$ warm-started instances.}\label{fig:cutmanagement}
  \end{minipage}
  \hfill
  \begin{minipage}[b]{0.48\textwidth}
    \includegraphics[width=\textwidth]{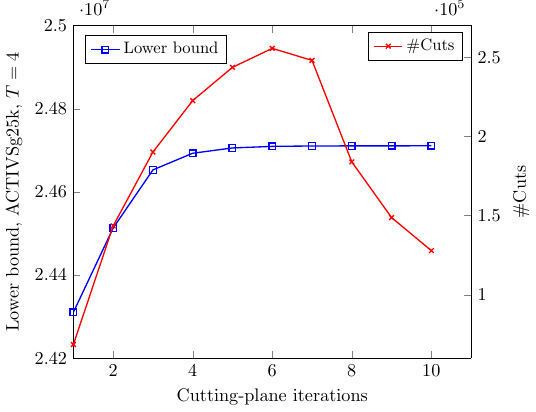}
    \caption{Lower bounds for case ACTIVSg25k and number of cuts.}\label{fig:improvement}
  \end{minipage}
\end{figure}

On tables~\ref{table:ws_dcopf_12} and~\ref{table:ws_dcopf_24} we compare our linearly-constrained warm-started relaxation against multi-period DCOPF\footnote{Multi-period DCOPF is solved with Gurobi using the same parameter setup as our cutting-plane algorithm.}. We report on the number of variables and constraints for each instance, as well as the objective, total running time, and dual infeasibility\footnote{For our relaxation, we report on the last iteration of our cutting-plane algorithm. Moreover, the number of constraints includes the total number of cuts kept during the last iteration.}. We remind the reader that DCOPF is a linearly-constrained approximation of ACOPF, not a relaxation~\cite{baker20}; hence, DCOPF infeasibility does not necessarily imply that the associated ACOPF is infeasible, in sharp contrast to what any relaxation for ACOPF guarantees.

Our ACOPF relaxations quadruple the size (in terms of number of variables and constraints) of the corresponding multi-period DCOPF instances. Gurobi clearly runs faster on multi-period DCOPF. Interestingly, there are some DCOPF instances for which Gurobi does not converge\footnote{We remind the reader that there are alternative formulations of DCOPF. Rather than directly solving an LP or convex QP, such alternative approaches could significantly enhance computational efficiency by exploiting problem structure (deploying a delayed constrained generation algorithm where the decision variables are nodal injections).} -- either because of suboptimal termination or because it runs into numerical trouble -- while it does converge for our ACOPF relaxations attaining small dual infeasibilities.

\begingroup
\setlength{\tabcolsep}{2pt} 
\begin{sidewaystable}
\centering
\tiny
\begin{adjustbox}{valign=m,center}
\hspace{-48em}
\resizebox{\textwidth}{!}{
\begin{minipage}{\textheight}
\vspace{12cm}
\caption{Warm-Started Relaxation versus Jabr SOCP, $T=4$ and loads perturbed $\text{Unif}(0,0.025) + 0.01 \cdot t$ for $t \in \{1,2,3\}$, ramping rates $50\%$}
\begin{tabular}{ @{} l r r r r r r r r r r r r r | r r r r r r r r | r r @{} }
\toprule
& \multicolumn{13}{c}{Cutting-Plane} & \multicolumn{8}{c}{Jabr SOCP} & \multicolumn{1}{c}{ACOPF}  \\
\cmidrule(l{0.5em}r{0.40em}){2-14} \cmidrule(l{0.5em}r{0.40em}){15-22} \cmidrule(l{0.5em}r{0.40em}){23-24}
& & & & \multicolumn{4}{c}{First Iteration} & \multicolumn{4}{c}{Last Iteration} & \multicolumn{2}{c}{Total} & & & \multicolumn{3}{c}{Obj} & \multicolumn{3}{c}{Time} &  & \\
\cmidrule(l{0.5em}r{0.40em}){5-8} \cmidrule(l{0.5em}r{0.40em}){9-12} \cmidrule(l{0.5em}r{0.40em}){13-14} \cmidrule(l{0.5em}r{0.40em}){17-19} \cmidrule(l{0.5em}r{0.40em}){20-22} 
\multicolumn{1}{l}{Case} & \#Vars & \#Cons & FTime & Obj & Time & \#Cuts & DInfs & Obj & Time & \#Cuts & DInfs & Time & Rnds & \#Vars & \#Cons &  Gurobi & Knitro & Mosek & Gurobi & Knitro & Mosek & PBound &  \\
\midrule
1354pegase & 66897 & 72564 & 1.41 & 301493.85 & 3.42 & 6000 & 2.55e-08 & 303438.51 & 1.93 & 10345 & 2.95e-08 & 36.04 & 12 & 55280 & 66149 & - & 303420.78 & 303415.82 & 15.34 & 11.65 & 14.79 & 303727.82 & \\
3375wp & 148217 & 159692 & 3.17 & 30132477.23 & 8.76 & 10448 & 1.42e-06 & 30318337.40 & 8.59 & 15786 & 4.44e-07 & 124.03 & 13 & 118128 & 142915 & - & - & 30233269.55 & 60.00 & TLim & 36.2 & 30410712.08 & \\
6468rte & 319517 & 323396 & 6.81 & 352645.45 & 18.24 & 32960 & 2.68e-06 & 355986.22 & 11.76 & 42631 & 2.72e-06 & 205.53 & 13 & 252240 & 302520 & 356099.86 & - & 355981.27 & 66.75 & TLim & 57.88 & 356969.89 & \\
9241pegase & 524028 & 538308 & 11.36 & 1258422.42 & 48.99 & 68472 & 3.61e-06 & 1270138.75 & 100.47 & 189019 & 5.18e-08 & 1268.90 & 7 & 433700 & 515921 & 1262979.52 & - & - & 182.32 & TLim & 89.64 & 1297205.96 &\\ 
ACTIVSg10k & 456684 & 490360 & 10.27 & 10117471.23 & 30.65 & 28184 & 5.72e-09 & 10222978.04 & 23.94 & 51188 & 5.01e-08 & 306.72 & 10 & 364824 & 437972 & - &  9894331.31 & 9876639.32 & 243.69 & 357.69 & 98.62 & 10265810.55 & \\
10000goc-api & 465676 & 495636 & 9.88 & 9129427.52 & 28.01 & 35856 & 6.70e-09 & 9495973.72 & 20.46 & 59211 & 3.92e-07 & 301.67 & 11 & 383978 & 467735 & - & - & 9479414.94 & 304.47 & TLim & 132.55 & 10019708.68 &  \\
10000goc-sad & 465676 & 495636 & 10.21 & 5414194.06 & 30.94 & 35224 & 1.29e-06 & 5456194.0 & 26.47 & 53520 & 1.31e-06 & 336.21 & 10 & 383978 & 467735 & - & - & 5448060.51 & 316.29 & TLim & 127.01 & 5866821.40 & \\
13659pegase & 723245 & 735740 & 15.67 & - & 365.52 & 97036 & - & - & 365.52 & 97036 & - & 384.05 & 0 & 578580 & 686433 & 1488477.49 & - & - & 308.57 & TLim & 100.23 & 1584844.81 & \\ 
13659pegase-api & 723245 & 735740 & 15.28 & 35162915.42 & 74.46 & 85584 & 2.42e-06 & 35703716.41 & 178.18 & 218790 & 4.06e-07 & 1348.20 & 8 & 594950 & 719171 & 35070776.73 & - & - & 717.60 & TLim & 479.87 & 36163923.41 & \\
13659pegase-sad & 723245 & 735740 & 15.60 & 34093210.18 & 52.16 & 78612 & 2.89e-08 & 34547590.80 & 98.42 & 136983 & 5.16e-07 & 1273.65 & 12 & 594950 & 719171 & 34196027.45 & - & 33517709.2 & 523.11 & TLim & 466.32 & 35199898.78 & \\
20758epigrids-api & 1106271 & 1127168 & 23.34 & - & 195.65 & 80904 & 5.54 & 11833435.09 & 658.47 & 260055 & 9.29e-05 & 1575.77 & 3 &918078 & 1105639 & - & - & 11830133.74 & 1093.22 & TLim & 322.49 & 12088587.29 & \\
20758epigrids-sad & 1106271 & 1127168 & 23.64 & 10138699.01 & 92.13 & 77228 & 2.76e-08 & 10219195.88 & 181.38 & 185780 & 8.84e-08 & 1374.00 & 8 & 918078 & 1105639 & - & - & 10211969.97 & 398.31 & TLim & 497.46 & 10317105.84 & \\
ACTIVSg25k & 1145783 & 1212392 & 25.46 & 24310520.83 & 93.95 & 68676 & 1.01e-07 & 24710662.67 & 92.3 & 127912 & 1.61e-08 & 1206.14 & 9 & 912168 & 1099187 & 23293596.46 & - & 23640456.52 & 768.78 & TLim & 395.86 & 24835784.33 & \\
30000goc-api & 1253699 & 1366264 & 30.24 & 5185673.79 & 97.36 & 59732 & 3.54e-07 & 5706550.06 & 80.15 & 105806 & 9.79e-07 & 1146.25 & 12 & 1013234 & 1246465 & - & - & 5637560.70 & TLim & TLim & 644.65 & 6563873.25 & \\
30000goc-sad & 1253699 & 1366264 & 28.03 & 4316717.55 & 138.08 & 179704 & 7.84e-07 & 4341119.26 & 77.37 & 91071 & 1.06e-06 & 994.70 & 9 & 1013234 & 1246465 & - & - & 4325675.85 & 
985.49 & TLim & 569.21 & - & \\
ACTIVSg70k & 3115935 & 3316972 & 68.7 & 67721489.41 & 350.58 & 258228 & 3.04e-08 & 69422696.57 & 369.05 & 597271 & 1.11e-07 & 1301.25 & 3 & 2480088 & 3003929 & - & - & 65293153.18 & TLim & TLim & 925.00 & 70148300.24 & \\
78484epigrids-api & 4179320 & 4225032 & 96.42 & 60947456.94 & 494.72 & 625780 & 1.53e-04 & 61271585.98 & 563.69 & 1047809 & 1.52e-04 & 1523.35 & 2 & 3420774 & 4126501 & - & - & - & TLim & TLim & 1109.78 & 62011780.96 & \\
78484epigrids-sad & 4179320 & 4225032 & 97.61 & 58930587.37 & 743.6 & 888916 & 1.20e-04 & 59051305.24 & 845.85 & 1133203 & 1.20e-04 & 1569.46 & 1 & 3420774 & 4126501 & 59124882.16 & - & - &1586.33 & TLim & 1103.22 & - & \\
\bottomrule
\end{tabular}
\label{table:ws_loads_T4_jabrsocp}
\caption{Warm-Started Relaxation versus i2 SOCP+, $T=4$ and loads perturbed $\text{Unif}(0,0.025) + 0.01 \cdot t$ for $t \in \{1,2,3\}$, ramping rates $50\%$}
\begin{tabular}{ @{} l r r r r r r r r r r r r r | r r r r r r r r | r r @{} }
\toprule
& \multicolumn{13}{c}{Cutting-Plane} & \multicolumn{8}{c}{i2 SOCP+} & \multicolumn{1}{c}{ACOPF}  \\
\cmidrule(l{0.5em}r{0.40em}){2-14} \cmidrule(l{0.5em}r{0.40em}){15-22} \cmidrule(l{0.5em}r{0.40em}){23-24}
& & & & \multicolumn{4}{c}{First Iteration} & \multicolumn{4}{c}{Last Iteration} & \multicolumn{2}{c}{Total} & & & \multicolumn{3}{c}{Obj} & \multicolumn{3}{c}{Time} &  & \\
\cmidrule(l{0.5em}r{0.40em}){5-8} \cmidrule(l{0.5em}r{0.40em}){9-12} \cmidrule(l{0.5em}r{0.40em}){13-14} \cmidrule(l{0.5em}r{0.40em}){17-19} \cmidrule(l{0.5em}r{0.40em}){20-22} 
\multicolumn{1}{l}{Case} & \#Vars & \#Cons & FTime & Obj & Time & \#Cuts & DInfs & Obj & Time & \#Cuts & DInfs & Time & Rnds & \#Vars & \#Cons &  Gurobi & Knitro & Mosek & Gurobi & Knitro & Mosek & PBound &  \\
\midrule
1354pegase & 66897 & 72564 & 1.41 & 301493.85 & 3.42 & 6000 & 2.55e-08 & 303438.51 & 1.93 & 10345 & 2.95e-08 & 36.04 & 12 & 56326 & 84155 & 303424.53 & - & - & 41.94 & TLim & 7.86 & 303727.82 & \\
3375wp & 148217 & 159692 & 3.17 & 30132477.23 & 8.76 & 10448 & 1.42e-06 & 30318337.40 & 8.59 & 15786 & 4.44e-07 & 124.03 & 13 & 121822 & 179665 & - & - & - & 60.00 & TLim & 15.28 & 30410712.08 & \\
6468rte & 319517 & 323396 & 6.81 & 352645.45 & 18.24 & 32960 & 2.68e-06 & 355986.22 & 11.76 & 42631 & 2.72e-06 & 205.53 & 13 & 269070 & 373246 & 355828.72 & - & - & 468.73 & TLim & 41.89 & 356969.89 & \\
9241pegase & 524028 & 538308 & 11.36 & 1258422.42 & 48.99 & 68472 & 3.61e-06 & 1270138.75 & 100.47 & 189019 & 5.18e-08 & 1268.90 & 7 & 451546 & 643811 & - & - & - & 411.55 & TLim & 92.33 & 1297205.96 &\\ 
ACTIVSg10k & 456684 & 490360 & 10.27 & 10117471.23 & 30.65 & 28184 & 5.72e-09 & 10222978.04 & 23.94 & 51188 & 5.01e-08 & 306.72 & 10 & 364824 & 437972 & - & - & - & 243.69 & TLim & 79.58 & 10265810.55 & \\
10000goc-api & 465676 & 495636 & 9.88 & 9129427.52 & 28.01 & 35856 & 6.70e-09 & 9495973.72 & 20.46 & 59211 & 3.92e-07 & 301.67 & 11 & 390042 & 567215 & - & - & - & 559.76 & TLim & 148.77 & 10019708.68 &  \\
10000goc-sad & 465676 & 495636 & 10.21 & 5414194.06 & 30.94 & 35224 & 1.29e-06 & 5456194.0 & 26.47 & 53520 & 1.31e-06 & 336.21 & 10 & 390042 & 567215 & - & - & - & 419.79 & TLim & 104.86 & 5866821.40 & \\
13659pegase & 723245 & 735740 & 15.67 & - & 365.52 & 97036 & - & - & 365.52 & 97036 & - & 384.05 & 0 & 618066 & 859791 & - & - & - & 457.45 & TLim & 117.18 & 1584844.81 & \\ 
13659pegase-api & 723245 & 735740 & 15.28 & 35162915.42 & 74.46 & 85584 & 2.42e-06 & 35703716.41 & 178.18 & 218790 & 4.06e-07 & 1348.20 & 8 & 618066 & 859791 & - & - & - & 1093.04 & TLim & 155.84 & 36163923.41 & \\
13659pegase-sad & 723245 & 735740 & 15.60 & 34093210.18 & 52.16 & 78612 & 2.89e-08 & 34547590.80 & 98.42 & 136983 & 5.16e-07 & 1273.65 & 12 & 618066 & 859791 & - & - & - & 891.76 & TLim & 150.4 & 35199898.78 & \\
20758epigrids-api & 1106271 & 1127168 & 23.34 & - & 195.65 & 80904 & 5.54 & 11833435.09 & 658.47 & 260055 & 9.29e-05 & 1575.77 & 3 & 950270 & 1340191 & - & - & - & 1529.93 & TLim & 356.82 & 12088587.29 & \\
20758epigrids-sad & 1106271 & 1127168 & 23.64 & 10138699.01 & 92.13 & 77228 & 2.76e-08 & 10219195.88 & 181.38 & 185780 & 8.84e-08 & 1374.00 & 8 & 950270 & 1340191 & - & - & - & 1397.93 & TLim & 247.43 & 10317105.84 & \\
ACTIVSg25k & 1145783 & 1212392 & 25.46 & 24310520.83 & 93.95 & 68676 & 1.01e-07 & 24710662.67 & 92.3 & 127912 & 1.61e-08 & 1206.14 & 9 & 961854 & 1387817 & - & - & - & 589.52 & TLim & 200.22 & 24835784.33 & \\
30000goc-api & 1253699 & 1366264 & 30.24 & 5185673.79 & 97.36 & 59732 & 3.54e-07 & 5706550.06 & 80.15 & 105806 & 9.79e-07 & 1146.25 & 12 & 1018714 & 1524129 & - & - & - & 1761.14 & TLim & 286.2 & 6563873.25 & \\
30000goc-sad & 1253699 & 1366264 & 28.03 & 4316717.55 & 138.08 & 179704 & 7.84e-07 & 4341119.26 & 77.37 & 91071 & 1.06e-06 & 994.70 & 9 & 1018714 & 1524129 & - & - & - & 727.32 & TLim & 247.6 & - & \\
ACTIVSg70k & 3115935 & 3316972 & 68.7 & 67721489.41 & 350.58 & 258228 & 3.04e-08 & 69422696.57 & 369.05 & 597271 & 1.11e-07 & 1301.25 & 3 & 2584042 & 3765747 & - & - & - & 1312.16 & TLim & 405.88 & 70148300.24 & \\
78484epigrids-api & 4179320 & 4225032 & 96.42 & 60947456.94 & 494.72 & 625780 & 1.53e-04 & 61271585.98 & 563.69 & 1047809 & 1.52e-04 & 1523.35 & 2 & 3558322 & 4997073 & - & - & - & TLim & TLim & 1117.23 & 62011780.96 & \\
78484epigrids-sad & 4179320 & 4225032 & 97.61 & 58930587.37 & 743.6 & 888916 & 1.20e-04 & 59051305.24 & 845.85 & 1133203 & 1.20e-04 & 1569.46 & 1 & 3558322 & 4997073 & - & - & - & TLim & TLim & 1117.35 & - & \\
\bottomrule
\end{tabular}
\label{table:ws_loads_T4_i2socp+}
\end{minipage}
}
\end{adjustbox}
\end{sidewaystable}
\endgroup

\begingroup
\setlength{\tabcolsep}{2pt} 
\begin{sidewaystable}
\centering
\tiny
\begin{adjustbox}{valign=m,center}
\hspace{40em}
\resizebox{\textwidth}{!}{
\rotatebox{180}{
\begin{minipage}{\textheight}
\vspace{-12cm}
\caption{Warm-Started Relaxation versus i2 SOCP+, $T=12$ hr active power load curve, ramping rates $50\%$}
\begin{tabular}{ @{} l r r r r r r r r r r r r r | r r r r r r r r | r r @{} }
\toprule
& \multicolumn{13}{c}{Cutting-Plane} & \multicolumn{8}{c}{i2 SOCP+} & \multicolumn{1}{c}{ACOPF}  \\
\cmidrule(l{0.5em}r{0.40em}){2-14} \cmidrule(l{0.5em}r{0.40em}){15-22} \cmidrule(l{0.5em}r{0.40em}){23-24}
& & & & \multicolumn{4}{c}{First Iteration} & \multicolumn{4}{c}{Last Iteration} & \multicolumn{2}{c}{Total} & & & \multicolumn{3}{c}{Obj} & \multicolumn{3}{c}{Time} &  & \\
\cmidrule(l{0.5em}r{0.40em}){5-8} \cmidrule(l{0.5em}r{0.40em}){9-12} \cmidrule(l{0.5em}r{0.40em}){13-14} \cmidrule(l{0.5em}r{0.40em}){17-19} \cmidrule(l{0.5em}r{0.40em}){20-22} 
\multicolumn{1}{l}{Case} & \#Vars & \#Cons & FTime & Obj & Time & \#Cuts & DInfs & Obj & Time & \#Cuts & DInfs & Time & Rnds & \#Vars & \#Cons &  Gurobi & Knitro & Mosek & Gurobi & Knitro & Mosek & PBound &  \\
\midrule
1354pegase & 201209 & 219772 & 4.25 & 855287.79 & 9.75 & 18000 & 1.70e-08 & 860664.27 & 6.16 & 30236 & 1.06e-08 & 92.63 & 12 & 168974 & 257483 & - & - & - & 245.39 &177.08 & 35.72 & 861473.50 & \\
9241pegase & 1574972 & 1626484 & 34.46 & 3554525.89 & 163.27 & 205416 & 5.04e-07 & 3565252.49 & 271.12 & 354209 & 4.32e-04 & 3811.42 & 8 & 1354634 & 1969307 & - & - & - & 1439.33 & TLim & 352.61 & - & \\
ACTIVSg10k & 1375020 & 1490960 & 30.21 & 28360735.35 & 106.03 & 84552 & 7.50e-09 & 28690135.28 & 94.69 & 148893 & 3.38e-07 & 1308.80 & 11 & 1151066 & 1714202 & - & - & - & 1543.11 & TLim & 190.39 & - & \\
10000goc-api & 1401204 & 1503620 & 32.03 & 27353308.25 & 96.60 & 107568 & 3.27e-07 & 28459989.11 & 74.63 & 183308 & 2.40e-08 & 1066.39 & 11 & 1170122 & 1736383 & - & - & - & TLim & TLim & 256.71 & - &\\
10000goc-sad & 1401204 & 1503620 & 30.42 & 16363341.55 & 133.30 & 105672 & 2.59e-06 & 16474320.51 & 105.74 & 144064 & 3.36e-06 & 1450.12 & 10 & 1170122 & 1736383 & - & - & - & 1918.80 & TLim & 247.71 & - & \\
ACTIVSg25k & 3447015 & 3675848 & 74.34 & 68451480.40 & 357.88 & 206028 & 2.68e-07 & 69534077.06 & 972.97 & 796784 & 9.75e-07 & 4073.97 & 5 & 2885558 & 4247241 & - & - & - & 3539.20 & TLim & 523.18 & - & \\
30000goc-api & 3768147 & 4127000 & 83.88 & 15513235.10 & 367.19 & 179196 & 7.00e-07 & 17079448.81 & 291.10 & 357763 & 1.06e-05 & 3667.40 & 9 & 3056138 & 4657273 & - & - & - & TLim & TLim & 756.77 & - & \\
30000goc-sad & 3768147 & 4127000 & 86.66 & 13129167.74 & 521.03 & 539112 & 1.68e-06 & 13190466.44 & 329.66 & 246646 & 2.45e-06 & 3901.76 & 8 & 3056138 & 4657273 & - & - & - & 1694.63 & TLim & 888.11 & - & \\
ACTIVSg70k & 9368583 & 10034036 & 207.22 & 184410148.60 & 1199.80 & 774684 & 4.07e-08 & 188736994.5 & 1276.64 & 1879714 & 3.66e-08 & 4702.30 & 3 & 7752122 & 11515211 & - & - & - & 3596.20 & TLim & 1124.19 & - & \\
78484epigrids-api & 12551704 & 12730080 & 263.06 & INF & 137.79 & 1877340 &  - & INF & 137.79 & 1877340 & - & 447.39 & 0 & 10674890 & 15270737 & - & - & - & 2875.97 & TLim & 1254.95 & - & \\  
78484epigrids-sad & 12551704 & 12730080 & 273.21 & 176961641.05 & 2697.25 & 2666748 & 1.36e-04 & 177335668.05 & 2828.46 & 3455898 & 1.36e-04 & 5465.67 & 1 & 10674890 & 15270737 & - & - & - & TLim & TLim & 1241.86 & - &\\
\bottomrule
\end{tabular}
\label{table:ws_loads_T12_i2socp+}
\vspace{1em}
\caption{Warm-Started Relaxation versus i2 SOCP+, $T=24$ hr active power load curve, ramping rates $50\%$}
\begin{tabular}{ @{} l r r r r r r r r r r r r r | r r r r r r r r | r r @{} }
\toprule
& \multicolumn{13}{c}{Cutting-Plane} & \multicolumn{8}{c}{i2 SOCP+} & \multicolumn{1}{c}{ACOPF}  \\
\cmidrule(l{0.5em}r{0.40em}){2-14} \cmidrule(l{0.5em}r{0.40em}){15-22} \cmidrule(l{0.5em}r{0.40em}){23-24}
& & & & \multicolumn{4}{c}{First Iteration} & \multicolumn{4}{c}{Last Iteration} & \multicolumn{2}{c}{Total} & & & \multicolumn{3}{c}{Obj} & \multicolumn{3}{c}{Time} &  & \\
\cmidrule(l{0.5em}r{0.40em}){5-8} \cmidrule(l{0.5em}r{0.40em}){9-12} \cmidrule(l{0.5em}r{0.40em}){13-14} \cmidrule(l{0.5em}r{0.40em}){17-19} \cmidrule(l{0.5em}r{0.40em}){20-22} 
\multicolumn{1}{l}{Case} & \#Vars & \#Cons & FTime & Obj & Time & \#Cuts & DInfs & Obj & Time & \#Cuts & DInfs & Time & Rnds & \#Vars & \#Cons &  Gurobi & Knitro & Mosek & Gurobi & Knitro & Mosek & PBound &  \\
\midrule
1354pegase & 402677 & 440584 & 8.66 & 1656321.34 & 20.38 & 36000 & 2.46e-08 & 1666609.97 & 12.74 & 63049 & 2.97e-08 & 180.61 & 11 & 337946 & 517475 & - & - & - & 345.44 & TLim & 53.32 & - &\\
9241pegase & 3151388 & 3258748 & 66.05 & 6865741.86 & 368.23 & 410832 & 4.33e-08 & 6925029.02 & 2427.02 & 1354277 & 1.25e-03 & 8419.66 & 4 & 2709266 & 3957551 & - & - & - & 1972.91 & TLim & 525.04 & - &\\
ACTIVSg10k & 2752524 & 2991860 & 66.13 & 54565119.63 & 226.82 & 169104 & 7.5e-09 & 55184650.09 & 211.53 & 298812 & 7.18e-08 & 2869.25 & 11 & 2302130 & 3446078 & - & - & - & 3918.14 & TLim & 396.91 & - &\\
10000goc-api & 2804496 & 3015596 & 62.15 & 51667906.99 & 208.71 & 215136 & 5.822e-07 & 53800966.46 & 175.46 & 388226 & 3.18e-07 & 2273.58 & 10 & 2340242 & 3490135 & - & - & - & 1332.84 & TLim & 542.88 & - &\\
10000goc-sad & 2804496 & 3015596 & 60.63 & 32512183.58 & 284.34 & 211344 & 3.46e-06 & 32678842.78 & 276.70 & 246720 & 7.97e-06 & 3099.33 & 9 & 2340242 & 3490135 & - & - & - & 2994.53 & TLim & 714.14 & - &\\
ACTIVSg25k & 6898863 & 7371032 & 150.16 & 132201219.46 & 816.18 & 412056 & 4.74e-07 & 134142620.00 & 2042.59 & 1385811 & 2.08e-06 & 6991.16 & 3 & 5771114 & 8536377 & - & - & - & 4451.39 & TLim & 963.91 & - &\\
30000goc-api & 7539819 & 8268104 & 167.43 & 28586461.36 & 755.74 & 358392 & 1.7313e-06 & 31616476.13 & 884.66 & 1064587 & 2.06e-05 & 6799.92 & 7 & 6112274 & 9356989 & - & - & - & TLim & TLim & 1139.21 & - &\\
30000goc-sad & 7539819 & 8268104 & 169.75 & 25717079.58 & 1091.27 & 1078224 & 2.52e-06 & 25811075.96 & 1021.24 & 1393213 & 1.23e-05 & 6249.53 & 5 & 6112274 & 9356989 & - & - & - & 3213.23 & TLim & 1125.14 & - &\\
ACTIVSg70k & 18747555 & 20109632 & 410.49 & 350310736.78 & 2641.18 & 1549368 & 2.407e-07 & 357083446.24 & 2687.31 & 3359811 & 2.98e-08 & 8496.71 & 2 & 1550424 & 23139407 & - & - & - & 5061.29 & TLim & 1331.11 & - &\\
78484epigrids-api & 25110280 & 24855336 & 533.07 & - & 1521.64 & 3754680 & - & - & 1521.64 & 3754680 & - & 2138.78 & 0 & 21349922 & 30681233 & - & - & - & TLim & 287.62 & 1434.04 & - &\\
78484epigrids-sad & 25110280 & 24855336 & 532.93 & 344805634.00 & 4749.24  & 5333496 & 1.34e-04 & - & TLim & 6915253 & - & 11561.02 & 1 & 21349922 & 30681233 & - & - & - & TLim & 284.28 & 1448.87 & - &\\
\bottomrule
\end{tabular}
\label{table:ws_loads_T24_i2socp+}
\end{minipage}
}
}
\end{adjustbox}
\end{sidewaystable}
\endgroup

\begingroup
\setlength{\tabcolsep}{2pt}
\centering
\begin{table}
\tiny
\caption{Warm-Started Relaxation versus DCOPF, $T=12$}
\resizebox{\textwidth}{!}{
\begin{tabular}{ @{} l r r r r r | r r r r r r @{} }
\toprule
& \multicolumn{5}{c}{Warm-Started Relaxation} & \multicolumn{5}{c}{DCOPF} & \\
\cmidrule(l{0.5em}r{0.40em}){2-6} \cmidrule(l{0.5em}r{0.40em}){7-11}
\multicolumn{1}{l}{Case}  & \#Vars & \#Cons & Obj & Time & DInfs & \#Vars & \#Cons &Obj & Time & Dinfs & \\
\midrule
1354pegase & 201209 & 250008 & 860664.27 & 92.63 & 1.06e-08 & 62369 & 67828 & 850023.48 & 2.44 & 1.72e-10 &  \\
9241pegase & 1574972 & 1980693 & 3565252.49 & 3811.42 & 4.32e-04 & 447608 & 477952 & 3627501.88 & 18.06  & 3.76e-10 &  \\
ACTIVSg10k & 1375020 & 1639853 & 28690135.28 & 1308.80 & 3.38e-07 & 449628 & 501812 & 28187885.77 & 19.50  & 2.34e-07 &  \\
10000goc-api & 1401204 & 1686928 & 28459989.11 & 1066.39 & 2.40e-08  & 446364 & 490232 & INF & 20.15 & - &  \\
10000goc-sad & 1401204 & 1647684 & 16474320.51 & 1450.12 & 3.36e-06 & 446364 & 490232 & 16114986.72 & 19.96 & 2.21e-07 &  \\
ACTIVSg25k & 3447015 & 4472632 & 69534077.06  & 4073.97  & 9.75e-07  & 1097943 & 1199456 & 67992943.20  & 48.82 & 7.88e-07 &  \\
30000goc-api & 3768147 & 4484763 & 17079448.81 & 3667.40  & 1.06e-05 & 1225815 & 1299860 & - & 75.87 & - &  \\
30000goc-sad & 3768147 & 4373646 & 13190466.44 & 3901.76  & 2.45e-06 & 1225815 & 1299860 & 12846216.92 & 53.33 & 7.65e-08 &  \\
ACTIVSg70k & 9368583 & 11913750 & 188736994.5  & 4702.30  & 3.66e-08  & 2977455  & 3195644 & - & 175.35 & - &  \\
78484epigrids-api & 12551704 & 14607420 & INF & 447.39  & - & 3555448 & 3699780 & INF & 4.19 & - &  \\
78484epigrids-sad & 12551704 & 16185978 & 177335668.05 & 5465.67 & 1.36e-04 & 3555448 & 3699780 & - & 175.91 & - &  \\
\bottomrule
\end{tabular}
\label{table:ws_dcopf_12}
}
\end{table}
\endgroup

\begingroup
\setlength{\tabcolsep}{2pt}
\centering
\begin{table}
\tiny
\caption{Warm-Started Relaxation versus DCOPF, $T=24$}
\resizebox{\textwidth}{!}{
\begin{tabular}{ @{} l r r r r r | r r r r r r @{} }
\toprule
& \multicolumn{5}{c}{Warm-Started Relaxation} & \multicolumn{5}{c}{DCOPF} & \\
\cmidrule(l{0.5em}r{0.40em}){2-6} \cmidrule(l{0.5em}r{0.40em}){7-11}
\multicolumn{1}{l}{Case}  & \#Vars & \#Cons & Obj & Time & DInfs & \#Vars & \#Cons &Obj & Time & Dinfs & \\
\midrule
1354pegase & 402677 & 503633 & 1666609.97  & 180.61   & 2.97e-08 & 124997 & 136696 & 1646661.83 &  4.61 & 2.15e-10 &  \\
9241pegase & 3151388 & 4613025  & 6925029.02 & 8419.66  & 1.25e-03 & 896660 & 961684 & 7013508.62 &  36.02 & 5.25e-10 &  \\
ACTIVSg10k & 2752524 & 3290672 & 55184650.09 & 2869.25  & 7.18e-08  & 901740 & 1013564 & 54241448.25 & 37.77 & 5.44e-08 &  \\
10000goc-api & 2804496 & 3403822 & 53800966.46 & 2273.58 & 3.18e-07 & 894816 & 988820 & INF & 40.25 & - &  \\
10000goc-sad & 2804496 & 3262316 & 32678842.78 & 3099.33 & 7.97e-06 & 894816 & 988820 & 32120062.03 & 41.02 & 1.97e-06 &  \\
ACTIVSg25k & 6898863 & 8756843  & 134142620.00 & 6991.16 & 2.08e-06 & 2200719 & 2418248 & 131325272.36 & 95.18 & 5.69e-09 &  \\
30000goc-api & 7539819 & 9332691 & 31616476.13 & 6799.92  & 2.06e-05 & 2455155 & 2613824 & - & 171.42 & - &  \\
30000goc-sad & 7539819 & 9661317 & 25811075.96 & 6249.53  & 1.23e-05 & 2455155 & 2613824 & 25279287.29 & 108.86 & 2.22e-08 &  \\
ACTIVSg70k & 18747555 & 23469443 & 357083446.24 & 8496.71 & 2.98e-08  & 5965299 & 6432848 & - & 829.66 & - &  \\
78484epigrids-api & 25110280 & 24855336 & - & 2138.78 & - & 7117768 & 7427052 & INF & 380.67 & - &  \\
78484epigrids-sad & 25110280 & 30188832 & 344805634.00  & 5399.62 & 1.34e-04 & 7117768 & 7427052 & - & 513.08 & - &  \\
\bottomrule
\end{tabular}
\label{table:ws_dcopf_24}
}
\end{table}
\endgroup

\subsection{Stability of dual Variables for Active-Power Balance Constraints}

In current energy pricing rules~\cite{oneill+etal05,gribik+etal07}, dual variables associated with active-power balance constraints~\eqref{AC:activepowerbal} of a welfare maximization problem constitute a key input for computing Locational Marginal Prices (LMPs). In general, linear models such as DCOPF are used for pricing because of their computational tractability. The trade-offs of using such approximations are well-known. As stated in~\cite{bichler+etal23}, ``If DCOPF does not adequately reflect the physics of the power grid, the prices retrieved from the DCOPF solution will not be accurate. Thus, tighter relaxations of ACOPF could lead to prices that better reflect scarcity in the physical network." Among their findings, the authors empirically show that nonlinear tighter relaxations such as the Jabr SOCP or the QC relaxation mitigate biasedness in price signals, in particular in congested networks, and lead to higher welfare.

Given the theoretical and empirical evidence presented in this paper, we believe that our linearly constrained relaxation is well-placed as a potential substitute for DCOPF in electricity market pricing. A first step in evaluating its potential is to understand whether the dual variables associated to active-power balance constraints of our dynamically built linearly constrained relaxations numerically converge to some vector of duals. In Table~\ref{table:convergence_prices} and Figure~\ref{fig:convergence_prices} below we provide preliminary evidence of numerical convergence for medium-sized $T=4$ instances. At each iteration $k > 1$ of our cutting-plane algorithm, we obtain the corresponding vector of duals $\lambda^{k}$ and compute its Euclidean distance to its predecessor $\lambda^{k-1}$.

\begingroup
\setlength{\tabcolsep}{3pt} 
\begin{table}
    \centering
    \footnotesize
    \caption{On the convergence of dual variables of active-power balance constraints $T=4$}
    \label{table:convergence_prices}
    \resizebox{\textwidth}{!}{
    \begin{tabular}{@{} l r r r r r r r r r r r @{}}
    \toprule
    & \multicolumn{10}{c}{$||\lambda^{k} - \lambda^{k-1} ||_{2}$} & \\
    \cmidrule(l{0.5em}r{0.40em}){2-11} 
    \multicolumn{1}{c}{Case} & \multicolumn{1}{c}{$k=2$} & \multicolumn{1}{c}{$k=4$} & \multicolumn{1}{c}{$k=6$} & \multicolumn{1}{c}{$k=8$} & \multicolumn{1}{c}{$k=10$} &\multicolumn{1}{c}{$k=12$} & \multicolumn{1}{c}{$k=14$} & \multicolumn{1}{c}{$k=16$} & \multicolumn{1}{c}{$k=18$} &  \\
    \midrule
    1354pegase & 42.5161 & 9.2024 & 2.9110 & 2.0099 & 1.0506 & 0.4326 & 0.189 & 0.1038 & 0.0494 & \\
    ACTIVSg20000 & 2590.985 & 481.1567 & 133.0116 & 67.344 & 38.3696 & 13.2228 &  7.9999 & 2.8572 & 2.5617 & \\
    2869pegase & 84.6365 & 14.806 & 5.7585 & 3.4687 & 2.4777 & 0.9821 & 0.8868 &  0.4770 & 1.8608 & \\
    6468rte & 232.8488 & 42.6396 & 14.7121 & 7.3318 & 4.1112 & 1.5908 & 1.2435 &  0.3554 & 0.6279 & \\
    ACTIVSg10k & 6394.2841 & 1712.8186 & 475.1414 & 104.5385 & 54.3611 & 17.6938 &  8.4489 & 7.7985 & 4.2688 & \\
    \bottomrule
    \end{tabular}
    }
\end{table}
\endgroup

\begin{figure}
\centering
\includegraphics{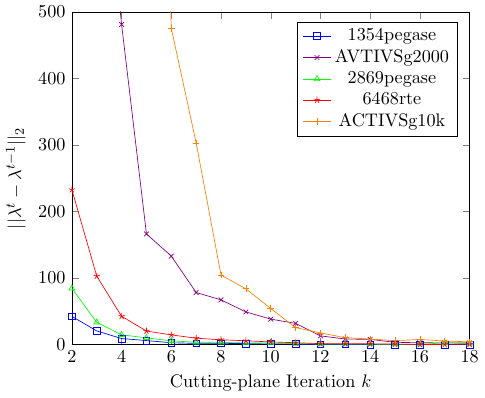} 
\caption{Distance between dual variables of active-power balance constraints of consecutive cutting-plane iterations.}\label{fig:convergence_prices}
\end{figure}

\section{Conclusions and Future Work}\label{section:futurework}

In this paper we present a fast linear cutting-plane method used to obtain tight relaxations for even the largest single and multi-period ACOPF instances, by appropriately outer-approximating the SOC relaxations. Our relaxations can be constructed and solved robustly and quickly via a cutting-plane algorithm that relies on proper cut management and leverages mature linear programming technology. Moreover, we provide a theoretical justification for the tightness of two SOC relaxations for ACOPF as well as for the use of our linearly-constrained relaxations.

The central focus on this paper concerns \textit{reoptimization}. We show that our procedure possesses efficient warm-starting capabilities -- previously computed cuts, for some given instance, can be re-utilized and loaded into new runs of \emph{related} instances, hence leveraging previous computational effort. As a main contribution we demonstrate, through extensive numerical testing in medium to (very) large multi-period instances, that the warm-start feature for our cutting-plane algorithm yields tight and accurate bounds far faster than otherwise possible. It is worth noting that this capability stands in contrast to what is possible using nonlinear (convex) solvers.

We bring forth to the ACOPF literature a discussion on approximate solutions to convex relaxations to ACOPF and their accuracy as lower bounds. We show that linearly-constrained relaxations with convex quadratic objective possess robust theoretical guarantees for bounding challenging optimization problems, in sharp contrast to what nonlinear relaxations such as SOCPs can offer.

We believe our work paves the way for promising new research directions. For instance, since our relaxations are linear they could be deployed for practical pricing schemes which could increase welfare and mitigate biasedness in price signals~\cite{bichler+etal23}. Moreover, we believe our relaxation can be used for harder multi-period ACOPF features such as unit-commitment or security constraints, hence it would be interesting to evaluate its performance on these challenging problems. We are also interested in developing new heuristics for finding multi-period ($T=12$ and $24$ time-periods) ACOPF solutions for large-scale networks exploiting solutions to our linearly-constrained relaxations.

\section*{Acknowledgements}

We would like to thank Erling Andersen, Bob Fourer, Ed Klotz and Richard Waltz.

\newpage

\section{Appendix}\label{appendix}

\subsection{i2 SOCP strictly stronger than Jabr SOCP}\label{appendix:i2strongerjabr}

We numerically show that for case 1354pegase its i2 SOCP relaxation is strictly stronger than its Jabr SOCP relaxation. Knitro attains an optimal solution to the Jabr SOCP of value 74009.28 while an optimal solution to the i2 SOCP has value 74013.68. As a sanity check, we fixed, within tolerance $\pm 10^{-5}$, the solution to the Jabr SOCP in the i2 SOCP formulation;  Gurobi declared the resulting SOCP infeasible and provided the following Irreducible Inconsistent Subsystem (IIS):

\begin{align*}
      i^{(2)}_{549,5002} &= 11822.45384038167 \, v^{(2)}_{549} \\
    &+ 11822.45384038167 \, v^{(2)}_{5002} \\ 
    &- 23644.8888107824 \, c_{549,5002} \\
    &- 29.87235441454166 \, s_{549,5002} \\
     v^{(2)}_{549}  &\geq 1.209989999283128 \\
     v^{(2)}_{5002} &\geq  1.19745626014781 \\
    c_{549,5002} &\leq 1.195643087927643 \\
    s_{549,5002} &\leq 0.0246578041355137 \\
     i^{(2)}_{549,5002} &\leq 39.69.
\end{align*}

\subsection{Definition of the i2 variable}\label{appendix:i2def}

Let the admittance matrix for line $\{km\}$ be
\begin{equation*}
    Y := \begin{pmatrix}
        \left(y + \frac{y^{sh}}{2}\right)\frac{1}{\tau^{2}} & -y \frac{1}{\tau e^{-j\sigma}} \\
        -y \frac{1}{\tau e^{j\sigma}} & y + \frac{y^{sh}}{2}
        \end{pmatrix} 
        = G + j B
\end{equation*}
where $y = g + jb $ denotes it's series admittance, shunt admittance is denoted by $y^{sh} = g^{sh} + j b^{sh}$, and $N := \tau e^{j\sigma}$ denotes the transformer ratio of magnitude $\tau > 0$ and phase shift angle $\sigma$.

Recall that Ohm's Law states that the current flowing from bus $k$ to $m$ via transmission line $\{km\}$ is given by $\begin{pmatrix} I_{km} \\ I_{mk} \end{pmatrix} = Y \begin{pmatrix}
    V_{k} \\
    V_{m}
\end{pmatrix} 
$, hence $I_{km} = \frac{y}{\tau}\left( \frac{1}{\tau}V_{k} - e^{j \sigma} V_{m} \right) + V_{k} \frac{y^{sh}}{2\tau^{2}}$. Therefore, the modulus of $I_{km}$ squared is  
\begin{align}
    |I_{km}|^{2} &= I_{km} I_{km}^{*} \nonumber \\
    & = \left( \frac{y}{\tau}\left( \frac{1}{\tau}V_{k} - e^{j \sigma} V_{m} \right) + V_{k} \frac{y^{sh}}{2\tau^{2}} \right) \left( \frac{y^{*}}{\tau}\left( \frac{1}{\tau}V_{k}^{*} - e^{-j \sigma} V_{m}^{*} \right) + V_{k}^{*} \frac{y^{sh*}}{2\tau^{2}} \right) \nonumber \\
    &= \frac{|y|^{2}}{\tau^{2}} \left| \frac{1}{\tau}V_{k} - e^{j \sigma} V_{m}  \right|^{2} + 2 \text{Re} \left[ \frac{y}{\tau}\left( \frac{1}{\tau}V_{k} - e^{j \sigma} V_{m} \right)  V_{k}^{*} \frac{y^{sh*}}{2\tau^{2}}  \right] \nonumber \\
    &+ \frac{1}{4 \tau^{4}} |y^{sh}|^{2} |V_{k}|^{2}. \label{i2def:firstRHS}
\end{align}

Notice that
\begin{align*}
    \text{Re} \left[ \frac{y}{\tau}\left( \frac{1}{\tau}V_{k} - e^{j \sigma} V_{m} \right)  V_{k}^{*} \frac{y^{sh*}}{\tau^{2}}  \right] &= \frac{1}{\tau^{3}}\text{Re} \left[ y y^{sh*}\left( \frac{1}{\tau}V_{k} - e^{j \sigma} V_{m} \right)  V_{k}^{*}  \right] \\
    &= \frac{1}{\tau^{3}}\text{Re} \left[ y y^{sh*} \left( \frac{1}{\tau}|V_{k}|^{2} - e^{j \sigma} V_{m}V_{k}^{*} \right)   \right] \\
    &= \frac{1}{\tau^{3}} [ (gg^{sh} + bb^{sh}) \left( (1/\tau)|V_{k}|^{2} \right.\\
    &\left. - |V_{k}| |V_{m}| \cos( \theta_{km} - \sigma) \right) \\
    &- \left. (bg^{sh} - gb^{sh})|V_{k}| |V_{m} | \sin(\theta_{km} - \sigma)  \right] 
\end{align*}
since $y y^{sh*} = (gg^{sh} + bb^{sh}) + j (bg^{sh} - gb^{sh})$. Moreover, the first term of the RHS of \eqref{i2def:firstRHS} can be written as
\begin{align*}
    \frac{|y|^{2}}{\tau^{2}} \left| \frac{1}{\tau}V_{k} - e^{j \sigma} V_{m}  \right|^{2} &= \frac{|y|^{2}}{\tau^{2}} \left( \frac{1}{\tau^{2}} |V_{k}|^{2} + |V_{m}|^{2} - \frac{2}{\tau} |V_{k}||V_{m}| \cos(\theta_{km} - \sigma) \right)
\end{align*}

Putting everything together yields
\begin{align*}
    |I_{km}|^{2} &= \frac{(g^{2} + b^{2})}{\tau^{2}} \left( \frac{1}{\tau^{2}} |V_{k}|^{2} + |V_{m}|^{2} - \frac{2}{\tau} |V_{k}||V_{m}| \cos(\theta_{km} - \sigma) \right) \\
    &+ \frac{1}{\tau^{3}} \left[ (gg^{sh} + bb^{sh}) \left( \frac{1}{\tau}|V_{k}|^{2} -  |V_{k}| |V_{m}| \cos( \theta_{km} - \sigma) \right) \right. \\
    &- \left. (bg^{sh} - gb^{sh})|V_{k}| |V_{m} | \sin(\theta_{km} - \sigma)  \right] \\
    &+ \frac{1}{4\tau^{4}}(g^{sh2} + b^{sh2})|V_{k}|^{2}.
\end{align*}
Since $|V_{k}| | V_{m}| \cos(\theta_{km} - \sigma) = c_{km} \cos(\sigma) + s_{km} \sin(\sigma)$ and $|V_{k}| | V_{m}| \sin(\theta_{km} - \sigma) = s_{km} \cos(\sigma) - c_{km} \sin(\sigma)$ we can represent $i^{(2)}_{km} := |I_{km}|^{2}$, linearly, in terms of the fundamental variables $v_{k}^{(2)}$,$v_{m}^{(2)}$, $c_{km}$, and $s_{km}$ as
\begin{align}
    i^{2}_{km} &= \frac{(g^{2} + b^{2})}{\tau^{2}} \left( \frac{1}{\tau^{2}} v_{k}^{(2)} + v_{m}^{(2)} - \frac{2}{\tau} (c_{km} \cos(\sigma) + s_{km} \sin(\sigma) ) \right) \nonumber\\
    &+ \frac{1}{\tau^{3}} \left[ (gg^{sh} + bb^{sh}) \left( \frac{1}{\tau}v_{k}^{(2)} -  (c_{km} \cos(\sigma) + s_{km} \sin(\sigma) ) \right) \right. \nonumber\\
    &- \left. (bg^{sh} - gb^{sh}) (s_{km} \cos(\sigma) - c_{km} \sin(\sigma))  \right] \nonumber \\
    &+ \frac{1}{4\tau^{4}}(g^{sh2} + b^{sh2})v_{k}^{(2)} \\
    &= \alpha_{km} v_{k}^{(2)} + \beta_{km} v_{m}^{(2)} + \gamma_{km} c_{km} + \zeta_{km} s_{km}, \label{appendix:i2def_eq}
\end{align}
where 
\begin{align*}
    \alpha_{km} :&= \frac{1}{\tau^{4}} \left( (g^{2} + b^{2}) + (gg^{sh} + bb^{sh}) + \frac{(g^{sh2} + b^{sh2})}{4} \right) \\
    \beta_{km} :&= \frac{(g^{2} + b^{2})}{\tau^{2}} \\
    \gamma_{km} :&= \frac{1}{\tau^{3}} \left( \cos(\sigma)( -2 (g^{2} + b^{2}) - (gg^{sh} + bb^{sh}) ) + \sin(\sigma)(bg^{sh} - gb^{sh}) \right) \\
    \zeta_{km} :&= \frac{1}{\tau^{3}} \left( \sin(\sigma)( -2 (g^{2} + b^{2}) - (gg^{sh} + bb^{sh}) ) - \cos(\sigma)(bg^{sh} - gb^{sh}) \right).
\end{align*}

\subsection{Proof of Proposition~\ref{prop:jabrouter}}\label{appendix:jabr_implies_lossineq}
By \eqref{eq:rotatedrewrite} we have that the Jabr inequality $c_{km}^{2} + s_{km}^{2} \leq v_{k}^{(2)} v_{m}^{(2)}$ can be written as $||(2 c_{km}, 2s_{km}, v_{k}^{(2)} - v_{m}^{(2)} )||_{2} \leq v_{k}^{(2)} + v_{m}^{(2)}$. Hence, taking $\lambda = (1,0,0)^{\top}$, and using \eqref{eq:genericouter} we have
\begin{equation}\label{loss:3}
    v_{k}^{(2)} + v_{m}^{(2)} - 2c_{km} \geq 0.
\end{equation}
On the other hand, by summing up equations~\eqref{SOCP:def_from_activepower} and~\eqref{SOCP:def_to_activepower} we have
\begin{align*}
    \ell_{km} &= G_{kk} v_{k}^{(2)} + G_{mm} v_{m}^{(2)} - 2 G_{km} c_{km} \\
    &\geq \min \{ G_{kk}, G_{mm} \} (v_{k}^{(2)} + v_{m}^{(2)} - 2c_{km})
\end{align*}
given our assumptions on $G_{kk},G_{mm},G_{km}, G_{mk}$. Therefore, \eqref{loss:3} implies $\ell_{km} \ge 0$. 

\subsection{Proof of Proposition~\ref{proposition:i2_implies_lossineq}}\label{appendix:i2_implies_lossineq}
First, we note that if $y_{km}^{sh} = 0$, a simple calculation yields
\begin{align*}
    \ell_{km} &= g_{km} \left| \frac{1}{\tau}V_{k} - e^{j \sigma} V_{m}  \right|^{2}, \\
    I_{km} &= \frac{y}{\tau} \left( \frac{1}{\tau}V_{k} - e^{j \sigma} V_{m} \right)
\end{align*}
see~\ref{appendix:i2def}. Therefore $i^{(2)}_{km} \geq 0$ clearly implies $\ell_{km} \geq 0$.

\subsection{Proof of Proposition~\ref{proposition:projection}}\label{appendix:proofprojection}
    First, we show that for every $(x,s) \in C$ the inequality $(x')^{\top}x \leq || x' || s$ holds. Indeed, by Cauchy-Schwartz inequality we have that $(x')^{\top}x \leq ||x'|| || x ||$, and since $||x|| \leq s$ the inequality follows. Next, we show that the point $(x_{0},s_{0})$ defined by $s_{0} := (||x'|| + s')/2$ and $x_{0} := (s_{0} / ||x'||) x'$ lies on $(x')^{\top} x = ||x'|| s$. Certainly, $(x')^{\top} x_{0} = (s_{0} / ||x'||) ||x'||^{2} = s_{0} ||x'||$. Moreover, $(x_{0},s_{0})$ lies on $C$ since $||x_{0}|| = ||(s_{0}/||x'||) x'|| = s_{0}$. Finally, we show that the vector $(x_{0} - x',s_{0} - s')$ is orthogonal to the plane (it suffices to take the vector $(x_{0},s_{0})$ since $(0,0)$ lies on $C$). It can be readily checked that $||(x_{0},s_{0})||^{2} - (x')^{\top} x_{0} - s's_{0} = 0$. 
 
\subsection{Proof of Proposition~\ref{proposition:thermalcuts}}\label{appendix:proofthermalcuts}
    Since $(x')^{2} + (y')^{2} > r^{2}$, there exists some $0 < t_{0} < 1$ such that $(t_{0} x')^{2} + (t_{0} y')^{2} = r^{2}$. It can be readily checked that $t_{0}(x',y')$ is the projection of $(x',y')$ onto $S_{r}$. Therefore, the normal vector of the separating hyperplane is $(1-t_{0})(x',y')$ and the RHS is $(1-t_{0})t_{0}||(x',y')||_{2}^{2}$, in other words, $(x')x + (y')y \leq r ||(x',y')||_{2}$ is the desired valid inequality since $0<t_{0}<1$.

\subsection{Proof of Proposition~\ref{proposition:boundedcoeffs}}\label{appendix:boundedcoeffs}
Let $\{k,m\} \in \mathcal{E}$ be an arbitrary transmission line, and denote by $A:= g^{2} + b^{2}$, $B:= gg^{sh} + bb^{sh}$ and $C:= (g^{sh2} + b^{sh2})/4$, see~\eqref{appendix:i2def_eq}. We drop the subscripts for clarity of exposition. Then, the coefficients in~\eqref{badi2s} can be written as
\begin{align*}
    \frac{\beta}{\alpha} &= \tau^{2} \frac{A}{A + B + C} \\
    \frac{\gamma}{\alpha} &= \tau \frac{- \cos(\sigma) ( 2 A + B ) + \sin(\sigma) (bg^{sh} - g b^{sh})}{A+B+C} \\
    \frac{\zeta}{\alpha} &= -\tau \frac{\sin(\sigma) ( 2 A + B ) + \cos(\sigma) (bg^{sh} - g b^{sh})}{A+B+C}
\end{align*}
Since we assumed $| b | > b^{sh}$, $g^{sh} = 0$, and $|b| > g$ we have that $A + B > 0$, hence the ratios $\frac{\gamma}{\alpha}$ and $\frac{\zeta}{\alpha}$ can be upper-bounded by
\begin{equation*}
    \tau \left( \frac{2A + |bb^{sh}|}{A + B + C} + \frac{|g b^{sh}|}{A + B + C} \right).
\end{equation*}
In consequence, we have the following upper-bounds
\begin{equation*}
    \frac{\beta}{\alpha} \leq \tau^{2}, \quad \frac{\gamma}{\alpha} \leq 3 \tau, \quad \frac{\zeta}{\alpha} \leq 3\tau.
\end{equation*}

\subsection{Proof of Proposition~\ref{proposition:nastyACOPF}}\label{appendix:nastyACOPF}
Consider a network consisting of two buses and line joining them. There is one generator located at bus $1$ and a load located at bus $2$. The transmission line is simple (i.e., no shunts or transformers) and has series admittance $y = g + jb$ with $g > 0$, $b <0$ and $-b > g$. Assume that the voltages at both buses are fixed to $1$, that reactive power generation is unconstrained and that the line limit is very large. Finally, suppose that the objective consists in minimizing active generation. Additionally, there is no explicit reactive power load. We can make the following observations:
    \begin{enumerate}
        \item Active power balance constraints: \(P_{1,2} = P_{1}^{g}\) and \(P_{2,1} = - P_{2}^{d}\);
        \item Power flows: $P_{1,2} = g - g c_{1,2} - b s_{1,2}$ and $P_{2,1} = g - g c_{1,2} + b s_{1,2}$;
        \item Jabr inequality: $c_{1,2}^{2} + s_{1,2}^{2} \leq 1$.
    \end{enumerate}
    Since reactive power generation is unconstrained and the line limit is very large, we do not have to write reactive power flows. Therefore, we can write the Jabr SOCP for this AC network as
    \begin{subequations}\label{example:nastysocp}
    \begin{align}
        OPT := \min \hspace{2em} g - g &c_{1,2} - b s_{1,2} \nonumber\\
        \text{subject to:}\hspace{3em} & \nonumber\\
        g - g c_{1,2} + b s_{1,2} &= - P^{d}_{2} \label{example:nastysocp_pd}\\
        c_{1,2}^{2} + s_{1,2}^{2} &\leq 1 \label{example:nastycsocp_jabr}
    \end{align}
    \end{subequations}
Since the network is radial, we know that imposing~\eqref{example:nastycsocp_jabr} as an \textit{equation} yields an exact ACOPF formulation~\cite{jabr06,farivar+low13}. Given that~\eqref{example:nastysocp_pd} is linear in $c_{1,2}$, $s_{1,2}$ and constraint~\eqref{example:nastycsocp_jabr} represents the unit circle, for a particular choice of parameters $g,b$, and $P^{d}_{2}$, the line represented by~\eqref{example:nastysocp_pd} will intersect the boundary of the unit circle at exactly two points. These points will be the unique two AC feasible solutions for this instance, while their convex hull represents the feasible region of the SOCP. Finally unless the objective is orthogonal to the hyperplane defined by constraint~\eqref{example:nastysocp_pd} the SOCP optimum will be exactly one of the two AC feasible points. We derive analytically these two AC points. Indeed, by substituting the expression for $s_{1,2}$ in~\eqref{example:nastysocp_pd} into~\eqref{example:nastycsocp_jabr} we obtain
    \begin{align*}
        \left( \frac{g^{2} + b^{2}}{b^{2}} \right) c_{1,2}^{2} - 2 \left( \frac{g}{b} \right) \gamma c + (\gamma^{2} - 1) = 0, \hspace{1em} &\gamma := \frac{P^{d} + g}{b} \\
        \implies c_{1,2} = \frac{(g/b) \gamma \pm \sqrt{\alpha}}{(g^{2} + b^{2})/b^{2}}, \hspace{1em} &\alpha:= 1 + \left(\frac{g}{b}\right)^{2} - \gamma^{2}.
    \end{align*}
    Moreover, solving for $P^{d}$ in terms of $\alpha$ we have
    \begin{equation*}
        P^{d} = - g \pm \sqrt{g^{2} - b^{2}(\alpha-1)}.
    \end{equation*}
Letting $g = 3$, $b = -8$, $\alpha = 3/4$, gives us $P^{d} = 2$ and $\gamma = -5/8$. This choice of parameters implies that the $s_{1,2}$-intercept in the $c_{1,2}-s_{1,2}$ plane of the line~\eqref{example:nastysocp_pd} is strictly between $(-1,1)$, hence there are exactly two AC feasible solutions and moreover they are irrational
    \begin{align*}
        \tilde{c}_{1,2} = \frac{(15/64) + (\sqrt{3}/2)}{(9/64) + 1}, \,\, \tilde{s}_{1,2} = - (3/8) \left( \frac{(15/64) + (\sqrt{3}/2)}{(9/64) + 1} \right) + (5/8) \\
        \hat{c}_{1,2} = \frac{(15/64) - (\sqrt{3}/2)}{(9/64) + 1}, \,\, \hat{s}_{1,2} = - (3/8) \left( \frac{(15/64) - (\sqrt{3}/2)}{(9/64) + 1} \right) + (5/8)
    \end{align*}
Given that the objective is not orthogonal to the linear constraint, and that it lies in the second quadrant, we conclude that the optimal solution to the Jabr SOCP is $(\tilde{c}_{1,2},\tilde{s}_{1,2}) \approx (0.9647,0.2632)$.

We conclude by showing that the dual of~\eqref{example:nastysocp} has a unique maximizer which is irrational as well. Indeed, let $\lambda \in \R$ and $\delta \in \R_{+}$ be the dual variables associated to the constraints~\eqref{example:nastysocp_pd} and~\eqref{example:nastycsocp_jabr}, respectively, and consider the following Lagrangean dual:
\begin{align*}
    L(c_{1,2},s_{1,2}; \lambda, \delta) :&= g - g c_{1,2} - bs_{1,2} + \lambda ( g - g c_{1,2} + b s_{1,2} + P^{d}_{2} ) \\ &+ \delta ( c_{1,2}^{2} + s_{1,2}^{2} - 1)
\end{align*}
The function $L(c_{1,2},s_{1,2}; \lambda, \delta)$ is convex quadratic in the variables $c_{1,2}, s_{1,2}$. Given that Slater's condition holds for~\eqref{example:nastysocp}, we have that an optimal primal-dual pair must satisfy the KKT condition (Lagrangean optimality)
\begin{align}
    OPT &= \min_{c_{1,2},s_{1,2}} \hspace{1em} L(c_{1,2},s_{1,2}; \lambda^{*}, \delta^{*}) \label{KKT1}\\
    &= L(c_{1,2}^{*},s_{1,2}^{*}; \lambda^{*}, \delta^{*}) \label{KKT2}
\end{align}
where $(c_{1,2}^{*},s_{1,2}^{*},\lambda^{*},\delta^{*})$ is an optimal primal-dual pair. It can be readily checked that the first-order stationarity condition implies that
\begin{align}
    2 \delta^{*} c_{1,2}^{*} - g (1 + \lambda^{*}) = 0 \label{example:nastysocp_FOC1}\\
    2 \delta^{*} s_{1,2}^{*} + b (\lambda^{*} - 1) = 0. \label{example:nastysocp_FOC2}
\end{align}
Since we previously showed that the unique minimizer of~\eqref{example:nastysocp} is the irrational solution $(\tilde{c}_{1,2},\tilde{s}_{1,2})$, we have that $c_{1,2}^{*} = \tilde{c}_{1,2}$ and $s_{1,2}^{*} = \tilde{s}_{1,2}$. Therefore, equations~\eqref{example:nastysocp_FOC1} and~\eqref{example:nastysocp_FOC2} imply that $(\lambda^{*},\delta^{*})$ must also be irrational.

\subsection{Proof of Proposition~\ref{proposition:superoptimality}}\label{appendix:superoptimality}
Let $\tilde{x} \in \R^{n}$ such that $A\tilde{x} \geq b - \epsilon e$, where $e$ denotes the all-ones vector. Let $[D]$ denote the Lagrangean Dual problem of $[P]$ on variables $\lambda \in \R^{m}$. Consider a primal-dual optimal pair $(\overline{x},\overline{\lambda}) \in \R^{n} \times \R^{m}$ such that $\overline{\lambda}$ is a rational vector that satisfies $||\overline{\lambda}||_{1} \leq g(A,b,H,c)$, where $g$ is a polynomial in the length of the input $A,b,H,c$. Such a solution $\overline{\lambda}$ exists by strong duality for convex QPs, we assumed $[P]$ is feasible and bounded, and by Section 2 in~\cite{vavasis90}). Strong duality also implies that $(\overline{x},\overline{\lambda})$ satisfies KKT conditions (i) Lagrangean optimality $\nabla f(\overline{x}) = A^{\top} \overline{\lambda}$ and (ii) complementary slackness $\overline{\lambda}^{\top}(b - A\overline{x}) = 0$. In addition, given that $f$ is convex and differentiable we have that $f(\overline{x}) + \nabla f(\overline{x})^{\top} (x - \overline{x}) \leq f(x)$ for every $x \in \R^{n}$. In consequence,
\begin{align*}
    f(\tilde{x}) &\geq f(\overline{x}) + \nabla f(\overline{x})^{\top} (\tilde{x} - \overline{x}) \\
    &= f(\overline{x}) + (A^{\top} \overline{\lambda})^{\top} (\tilde{x} - \overline{x}) \\
    &= f(\overline{x}) + \overline{\lambda}^{\top} (A\tilde{x}) - \overline{\lambda}^{\top} (A\overline{x}) \\
    &\geq f(\overline{x}) + \overline{\lambda}^{\top}(b - A\overline{x}) - \epsilon ||\overline{\lambda} ||_{1} \\
    &= f(\overline{x}) - \epsilon || \overline{\lambda} ||_{1} \\
    &\geq f(\overline{x}) - \epsilon g(A,b,H,c).
\end{align*}

\subsection{Superoptimality and SOCPs}\label{appendix:superopt_socps}
Consider the following SOCP:

\begin{align*}
[1] \hspace{2.5em} \overline{p} := \quad \min \hspace{2em} &c^\top x \\
    \text{subject to:}\quad\quad & \nonumber\\
    || &y_{i} ||_{2} \leq t_{i} \quad i \in [m] \\
    y_{i} = A_{i}&x + b_{i}, \quad t_{i} = c^{\top}_{i}x + d_{i} \quad i \in [m] \\
    A&x \geq b
\end{align*}

\noindent Its dual is given by:
\begin{align*}
[2] \hspace{5em} \overline{d} := \quad \min \hspace{2em} &\lambda^\top b - \sum_{i=1}^{m} (u_{i}d_{i} + v_{i}^{\top} b_{i}) \\
    \text{subject to:}\quad\quad& \nonumber\\
    || &v_{i} ||_{2} \leq u_{i} \quad i \in [m] \\
    c^{\top} = &\lambda^{\top} A + \sum_{i=1}^{m} (v_{i}^{\top} A_{i} + u_{i} c_{i}^{\top}) \\
    &\lambda, u_{i} \geq 0
\end{align*}

\begin{lemma}\label{proposition:superopt_socps}
Suppose that [1] is strictly feasible and $\overline{p} > - \infty$. Then strong duality holds, i.e., $\overline{p} = \overline{d}$, and there exists an optimal primal-dual pair $(\overline{x},(\overline{y}_{i},\overline{t}_{i})_{i\in[m]})$, $(\overline{\lambda}, (\overline{v}_{i},\overline{u}_{i})_{i\in[m]})$. Let $(\tilde{x},(\tilde{y}_{i},\tilde{t}_{i})_{i\in[m]})$ be an $\epsilon$-feasible solution to the primal problem [1], i.e., it satisfies
\begin{equation*}
    \epsilon := \max \left\{ \max \{ b_{i} - A_{i}^{\top} \tilde{x} \, : \, i \in [m] \}, \max \{ || \tilde{y}_{i} ||_{2} - \tilde{t}_{i} \, : \, i \in [m] \} \right\}.
\end{equation*}
Then, 
\begin{equation}
     c^{\top} \tilde{x} \geq \overline{p} - \epsilon ( ||\overline{\lambda}||_{1} + ||\overline{u} ||_{1}).
\end{equation}
\end{lemma}
\begin{proof}
Let $e_{m}$ the all-ones vector in $\R^{m}$. Then,
    \begin{align*}
        c^\top \tilde{x} &= \overline{\lambda}^{\top} A \tilde{x} + \sum_{i=1}^{m} (\overline{v}_{i}^{\top} A_{i}\tilde{x} + \overline{u}_{i} c_{i}^{\top} \tilde{x}) \\
        &\geq \overline{\lambda}^{\top} (b - \epsilon e_{m}) + \sum_{i=1}^{m} (\overline{v}_{i}^{\top}(\tilde{y}_{i} - b_{i}) + \overline{u}_{i} (\tilde{t}_{i} - d_{i}))  \\
        &= \overline{\lambda}^{\top} b - \epsilon ||\overline{\lambda}||_{1} + \sum_{i=1}^{m} ( \overline{v}_{i}^{\top} \tilde{y}_{i} + \overline{u}_{i} \tilde{t}_{i} ) - \sum_{i=1}^{m} (\overline{v}_{i}^{\top} b_{i} + \overline{u}_{i}d_{i}) \\
        &\geq \overline{p} - \epsilon ||\overline{\lambda}||_{1} + \sum_{i=1}^{m} ( \overline{v}_{i}^{\top} \tilde{y}_{i} + \overline{u}_{i} ( ||\tilde{y}_{i}||_{2} - \epsilon ) ) \\
        &= \overline{p} - \epsilon ||\overline{\lambda}||_{1} - \epsilon ||\overline{u} ||_{1} +  \sum_{i=1}^{m} ( \overline{v}_{i}^{\top} \tilde{y}_{i} + \overline{u}_{i} ||\tilde{y}_{i}||_{2} ) \\
        &\geq \overline{p} - \epsilon (||\overline{\lambda}||_{1} + ||\overline{u} ||_{1}),
    \end{align*}
    where the last inequality follows since if $||\overline{v}_{i}||_{2} \leq \overline{u}_{i}$ then $-\overline{v}_{i}^{\top} \tilde{y}_{i} \leq || \overline{v}_{i} ||_{2} || \tilde{y}_{i} ||_{2} \leq \overline{u}_{i} || \tilde{y}_{i} ||_{2} $ for all $i \in [m]$.
\end{proof}

\tiny Wed.Sep..11.151000.2024

\bibliographystyle{spmpsci}  

\bibliography{paper_mpc.bib}

\end{document}